\documentclass[11pt,a4paper]{article}
\usepackage[english]{babel}
\usepackage{lmodern} 
\usepackage[a4paper,left=1in,right=1in,top=1.4in,bottom=1.4in]{geometry} 
\usepackage[backend=biber,style=alphabetic,maxnames=6,maxalphanames=6,minalphanames=6]{biblatex}
\DeclareFieldFormat[article,misc,incollection]{title}{#1} 
\renewbibmacro{in:}{} 
\DeclareFieldFormat[misc]{howpublished}{\textit{#1}}
\usepackage{csquotes}
\usepackage{hyperref}
\usepackage{amsmath,amssymb,amsthm}
\usepackage{cancel} 
\usepackage{bm}
\usepackage{bbm}
\usepackage{tikz}
\usepackage{tikz-cd}
\usetikzlibrary{positioning}

\usepackage{xcolor}

\title{Spin refinement of moduli spaces of residueless meromorphic differentials and the BKP hierarchy}
\author{David Klompenhouwer, Stijn Velstra}
\date{4 June 2025}

\newtheorem{theorem}{Theorem}[section]
\newtheorem{lemma}[theorem]{Lemma}
\newtheorem{proposition}[theorem]{Proposition}
\theoremstyle{definition}
\newtheorem{definition}[theorem]{Definition}
\newtheorem{remark}[theorem]{Remark}

\renewcommand{\a}{\alpha}
\renewcommand{\b}{\beta}
\newcommand{\g}{\gamma}

\newcommand{\e}{\varepsilon}
\renewcommand{\d}{\partial}
\newcommand{\Z}{\mathbb{Z}}
\newcommand{\C}{\mathbb{C}}
\renewcommand{\P}{\mathbb{P}}
\newcommand{\R}{\mathcal{R}}
\newcommand{\PO}{\mathcal{PO}}
\newcommand{\dd}[2]{\frac{\partial#1}{\partial#2}}
\newcommand{\coef}{\mathrm{Coef}_{a_1^{k_1}\cdots a_n^{k_n}}}
\newcommand{\res}{\operatorname{res}}
\newcommand{\Mg}{\mathcal{M}_g}
\newcommand{\Mgn}{\mathcal{M}_{g,n}}

\newcommand{\Mgnb}{\overline{\mathcal{M}}_{g,n}}
\newcommand{\M}[1]{\mathcal{M}_{#1}}
\newcommand{\Mb}[1]{\overline{\mathcal{M}}_{#1}}
\newcommand{\DR}{\mathrm{DR}}
\newcommand{\dz}[1]{\delta_0^{\{#1\}}}
\newcommand{\Hg}[2]{[\overline{\mathcal{H}}_{#1}{(#2)}]}
\newcommand{\Hres}[2]{[\overline{\mathcal{H}}^\mathrm{res}_{#1}(#2)]}
\newcommand{\Hspin}[2]{[\overline{\mathcal{H}}_{#1}(#2)]^\mathrm{spin}}
\newcommand{\HresSpin}[2]{[\overline{\mathcal{H}}^\mathrm{res}_{#1}(#2)]^\mathrm{spin}}
\newcommand{\Pspin}[1]{P^{\mathrm{spin}}_{#1}}
\newcommand{\Qspin}[1]{Q^{\mathrm{spin}}_{#1}}
\newcommand{\at}[2]{\left.#1\right\vert_{#2}}
\newcommand{\Aspin}[2]{\mathcal{A}_{#1}^\mathrm{spin}(#2)}

\DeclareMathOperator{\Aut}{Aut}
\renewcommand{\H}{\mathcal{H}}
\newcommand{\Pic}{\operatorname{Pic}}
\newcommand{\Q}{\mathbb{Q}}

\begin{document}

\maketitle

\abstract{\noindent We consider strata of curves carrying a residueless meromorphic differential inducing a spin structure on the curve. The cohomology classes of the closures of these strata, weighted by the parity of the spin structures, form a partial cohomological field theory (CohFT) of infinite rank. After applying the DR hierarchy construction to this partial CohFT and reducing to differentials with two zeros and arbitrarily many poles, we show that the resulting system of evolutionary PDEs coincides with the BKP hierarchy up to a coordinate transformation. This is a spin refinement of an analogous result from \cite{brz24}. Our proof relies on a new result regarding the reconstruction of the BKP hierarchy from a limited amount of information in the Lax formalism.

\addcontentsline{toc}{section}{Introduction}
\tableofcontents

\begin{center}\section*{Introduction}\end{center}
\subsection*{Overview} 

The Hodge bundle $\Omega\Mg\rightarrow\Mg$ over the moduli space of smooth curves of genus $g\ge2$ is a rank $g$ vector bundle whose fibre at $[C]\in\mathcal{M}_g$ is the space of holomorphic differentials on $C$. The complement of the zero section is naturally stratified into strata indexed by partitions of $2g-2$, which prescribe the multiplicities of the zeros and poles of the differentials. There is a natural $\C^*$-action by scaling the differentials which preserves the stratification, so it is natural to study the projectivized strata. \newline

One can extend the situation above to meromorphic differentials. Let $g,n\ge0$ be integers satisfying $2g-2+n>0$, and let $\a=(\a_1,\dots,\a_n)\in\Z^n$ satisfy $\sum_{i=1}^n\a_i=2g-2$. The moduli space of projectivized meromorphic differentials $\P\Omega\Mgn(\a)$ parametrizes objects $(C,x_1,\dots,x_n,\omega)$ consisting of an $n$-pointed smooth curve $(C,x_1,\dots,x_n)\in\Mgn$ and a $\C^*$-equivalence class of meromorphic differentials $\omega$ on $C$ satisfying $\operatorname{ord}_{x_i}\omega=\a_i$. There is a natural morphism $\P\Omega\Mgn(\a)\rightarrow\Mgn$ given by forgetting the differential, whose image is denoted by $\mathcal{H}_g(\a)\subset\Mgn$. The closure of $\mathcal{H}_g(\a)$ in the moduli space of stable curves $\Mgnb$ is denoted by $\overline{\mathcal{H}}_g(\a)\subset\Mgnb$. By a slight abuse of language, we also call this the moduli space of projectivized meromorphic differentials. \newline

Much work has been done to understand the geometry and topology of these spaces. The connected components of $\mathcal{H}_g(\a)$ were classified in \cite{kz03} in the holomorphic case, and the meromorphic case was treated in \cite{boi15}. Several compactifications of $\P\Omega\Mgn(\a)$ have appeared in the literature. The incidence variety compactification (IVC) of $\P\Omega\Mgn(\a)$ is its closure inside a suitable twist of the projectivized Hodge bundle $\P\Omega\Mgnb\rightarrow\Mgnb$, and is extensively described in \cite{bcggm18}. A similar construction is presented in \cite{sau19}. In \cite{cmsz20} numerical quantities associated with the geometric and dynamical structures of $\Omega\Mgn(\a)$, the Masur-Veech volumes and Siegel-Veech constants, are computed in terms of intersection numbers on the IVC. On the other hand, Farkas-Pandharipande \cite{fp18} construct a proper moduli space of twisted canonical divisors $\widetilde{\mathcal{H}}_g(\a)\subset\Mgnb$ which contains the closure $\overline{\mathcal{H}}_g(\a)\subset\widetilde{\mathcal{H}}_g(\a)$. In the strictly meromorphic case (when there exists a $\a_i<0$), the moduli space $\widetilde{\mathcal{H}}_g(\a)$ is of pure codimension $g$ in $\Mgnb$ and the closure $\overline{\mathcal{H}}_g(\a)\subset\widetilde{\mathcal{H}}_g(\a)$ is a union of irreducible components. In \cite{bcggm19b} a smooth complex orbifold with normal crossing boundary $\P\Xi\Mgnb(\a)$ compactifying $\P\Omega\Mgn(\a)$ was constructed, called the moduli space of multi-scale differentials. Its properties, of which we give more details in \hyperref[subsec:multi]{Section \ref*{subsec:multi}}, make it similar to the Deligne-Mumford compactification $\Mgnb$ of $\Mgn$. \newline

The inspiration for this paper is a theorem of Buryak-Rossi-Zvonkine \cite{brz24}. The authors showed that the cohomology classes $\Hres{g}{\a_1,\dots,\a_n}\in H^*\bigl(\Mgnb\bigr)$ of residueless meromorphic differentials form a partial cohomological field theory (pCohFT). They applied the DR hierarchy construction of \cite{bur15} to this pCohFT to produce a Hamiltonian system of evolutionary PDEs. After reducing to the case where the differentials have exactly two zeros and any number of poles, they showed that the resulting system of PDEs coincides with the KP hierarchy, up to a change of variables. The strategy is to prove that the KP hierarchy can be uniquely reconstructed from a few properties and that the reduced DR hierarchy satisfies these properties. \newline

In this paper we prove a spin refinement of the above result \hyperref[thm:main]{(Theorem \ref*{thm:main})}. The spin refinement involves considering strata $\mathcal{H}_g^\mathrm{res}(2\a_1,\dots,2\a_n)\subset\Mgn$ of smooth curves $C$ carrying a meromorphic differential with even orders at zeros and poles. Such a differential induces a spin structure on $C$, i.e. a line bundle $L\rightarrow C$ such that $L^{\otimes2}\simeq\omega_C$. The parity of $h^0(C,L)$ is invariant in families of spin structures on curves, so it splits $\mathcal{H}_g^\mathrm{res}(2\a_1,\dots,2\a_n)$ into an even and odd component. This leads us to consider the cohomology classes
\begin{equation*}
    \HresSpin{g}{2\a_1,\dots,2\a_n} := [\overline{\mathcal{H}}^\mathrm{res}_g(2\a_1,\dots,2\a_n)^\mathrm{even}] - [\overline{\mathcal{H}}^\mathrm{res}_g(2\a_1,\dots,2\a_n)^\mathrm{odd}] \in H^*(\Mgnb),
\end{equation*}
which form a partial CohFT. Our result establishes an equivalence between a reduction of the DR hierarchy associated to this partial CohFT and the BKP hierarchy, a well-known reduction of the KP hierarchy \cite{djkm82}. \newline

Our strategy is similar to the one in \cite{brz24}. We show that the BKP hierarchy (in its normal coordinates) can be uniquely reconstructed from the linear term in the dispersionless limit of all the flows and only the first non-trivial potential, using the properties of commutativity of the flows, homogeneity, tau-symmetry, and compatibility with spatial translations (see \hyperref[thm:bkpr]{Theorem \ref*{thm:bkpr}}). Note that our reconstruction can be adapted to the KP hierarchy and requires less information than \cite[Theorem 3.12]{brz24}. This result is expected to be adaptable to other reductions of the KP hierarchy. \newline

Furthermore, since the BKP hierarchy is concerned with only the odd flows of the KP hierarchy, corresponding to the even order-profile of the differentials, the determination of the first non-trivial potential on the geometric side requires intersection-theoretic computations on moduli spaces of genus $g=2$ stable curves (see \hyperref[lem:g2int]{Lemma \ref*{lem:g2int}}). This gives rise to novel computations, which call for the need of a variety of results concerning double ramification cycles and strata of differentials in the literature. \newline

Lately there has been an interest in spin refinements of previously obtained results, see \cite{is22} for example. Our result is in line with the general idea that if a geometric problem is governed by the KP hierarchy, its spin version is governed by the BKP hierarchy. A well-studied example of this phenomenon is in the context of Hurwitz numbers enumerating branched covers of the Riemann sphere. It is known that the generating function of ordinary Hurwitz numbers is a $\tau$-function of the KP hierarchy \cite{kl07}. Spin Hurwitz numbers were first introduced in \cite{eop08}, and their generating function is a $\tau$-function of the BKP hierarchy \cite{mmn20,mmno21}. Further results involving more general spin Hurwitz numbers can be found in \cite{lee20,glk21,as23,wy23}. Another expression of this phenomenon is found in \cite{gkls25}, to be compared with \cite{op06}, stating that the equivariant potential of the spin Gromov-Witten theory of $\P^1$ is a $\tau$-function of the 2-component BKP hierarchy.

\subsection*{Organization of the paper}

In \hyperref[sec:spin]{Section \ref*{sec:spin}} we introduce some preliminary notions about multi-scale differentials from \cite{cmz22} and of spin structures on curves carrying meromorphic differentials. We then give the definition, for $g\ge0$ and $\a_1,\dots,\a_n\in\Z$ satisfying $\a_1+\cdots+\a_n=g-1$, of the cohomology classes
\begin{equation*}
    \HresSpin{g}{2\a_1,\dots,2\a_n}=[\overline{\mathcal{H}}^\mathrm{res}(2\a_1,\dots,2\a_n)^\mathrm{even}]-[\overline{\mathcal{H}}^\mathrm{res}(2\a_1,\dots,2\a_n)^\mathrm{odd}]\in H^*(\Mgnb)
\end{equation*} 
parametrizing curves carrying residueless meromorphic differentials with zeros/poles at the marked points of orders $2\a_1,\dots,2\a_n$ and a spin structure of even or odd parity induced by the differential. Our first result (\hyperref[prop:cohft]{Proposition \ref*{prop:cohft}}) is that these classes form a homogeneous partial cohomological field theory (pCohFT) of infinite rank. \newline

In \hyperref[sec:dr]{Section \ref*{sec:dr}} we first give an overview of the DR hierarchy construction of \cite{bur15}. We then apply this construction to the partial CohFT above, and reduce to the case where the differentials have exactly two zeros. The resulting system of evolutionary PDEs is
\begin{equation*}
    \dd{u_\a}{t^\b}=\d_x\Pspin{\a\b},\qquad\a,\b\ge1,
\end{equation*}
where the differential polynomials $\Pspin{\a\b}$ depend on intersection-theoretic properties of \newline $\HresSpin{g}{2\a_1,\dots,2\a_n}$. We perform a coordinate transformation to bring the system to a more symmetric form. \newline

In \hyperref[sec:comp]{Section \ref*{sec:comp}} we compute the first three nontrivial polynomials $\Pspin{1,2},\Pspin{1,3}$ and $\Pspin{2,2}$. This requires performing intersection-theoretic computations on moduli spaces of stable curves of genera $g=0,1,2$, involving $\psi$-classes, $\lambda$-classes, double ramification cycles and spin strata of residueless meromorphic differentials. We use results from \cite{brz24}, \cite{hai13}, \cite{css21} and \cite{bhs22}, among others. \newline

In \hyperref[sec:kp]{Section \ref*{sec:kp}} we outline the construction of the Kadomtsev-Petviashvili (KP) hierarchy from the Lax operator viewpoint, and describe how one obtains the BKP hierarchy as a reduction of the KP hierarchy by imposing a constraint on the Lax operator. We then exhibit some key properties of the BKP hierarchy in its normal coordinates. \newline

In \hyperref[sec:rec]{Section \ref*{sec:rec}} we state the main result of this paper (\hyperref[thm:main]{Theorem \ref*{thm:main}}), which states that the DR hierarchy associated to the spin refinement of moduli spaces of residueless meromorphic differentials with exactly two zeros corresponds to the BKP hierarchy, up to a coordinate transformation. The main content of this section is the proof of a reconstruction result for the BKP hierarchy (\hyperref[thm:bkpr]{Theorem \ref*{thm:bkpr}}) which, together with the results from the previous sections, immediately implies our main result.

\subsection*{Further developments}

The main result of this paper and \cite{brz24} deals with a reduction of the DR hierarchy to the case where the differentials have exactly two zeros and any number of poles. A natural goal for the future is to identify which integrable hierarchy corresponds to the non-reduced DR hierarchy, in both the spin and non-spin case. \newline

The Dubrovin-Zhang (DZ) hierarchy \cite{dz01} gives another procedure to obtain an integrable hierarchy from a semi-simple CohFT. The DR/DZ equivalence conjecture \cite{bur15,bdgr18}, the proof of which was completed in \cite{bls24}, states that the DR and DZ hierarchies associated to finite rank semi-simple CohFTs are normal Miura equivalent. It would be interesting to investigate if this equivalence extends to partial CohFTs of infinite rank, such as the one we consider here. \newline

Many of the constructions and results mentioned at the start have been generalized to the case of $k$-differentials ($k\ge1$) \cite{bcggm19a,sch18,cmz24}. In particular, in \cite{css21} the intersection theoretic properties of the moduli space of multi-scale $k$-differentials are related to the $k$-twisted double ramification cycle \cite{jppz17, hol21}. Using these, it could be possible to obtain a result for $k$-differentials similar in spirit to the one that we present here.

\subsection*{Notation and conventions}
\begin{itemize}
    \item For a topological space $X$, let $H^*(X)$ denote the cohomology ring of $X$ with coefficients in $\C$.
    \item For $n\ge0$, the notation $[n]$ denotes the set $\{1,2,\dots,n\}$.
    \item In \hyperref[sec:spin]{Section \ref*{sec:spin}} and \hyperref[sec:dr]{Section \ref*{sec:dr}} we sum over repeated Greek indices.
    \item In absence of ambiguity, we use the symbol $*$ to indicate any value, in the appropriate range, of a subscript or superscript.
    \item Given a collection of variables $\{u_i:i\in\Z\}$, the notation $u_{\le m}$ denotes the collection of all $u_i$ with $i\le m$.
\end{itemize}

\subsection*{Acknowledgments}
We would like to thank P. Rossi and A. Sauvaget for suggesting this problem to us, and for their continued guidance.  Moreover, we are grateful to A. Buryak, M. Costantini and P. Spelier for useful discussions. The Sage package \textit{admcycles} \cite{admcycles} was of great use to us when performing and checking intersection-theoretic computations. \newline

\noindent S. V. was partially supported by the ERC Starting Grant 101164820 ``Spin Curves Enumeration''. D. K. was supported by the University of Padova and is affiliated to the INFN under the national project MMNLP.

\begin{center}\section{Residueless meromorphic differentials and spin refinement}\label{sec:spin}\end{center}
\subsection{Multi-scale residueless differentials}\label{subsec:multi}

Fix integers $g,n\ge0$ satisfying $2g-2+n>0$. Let $\a=(\a_1,\dots,\a_n)\in\Z^n$ be a partition of $2g-2$. The space of projectivized meromorphic differentials
\begin{equation*}
    \mathcal{H}_g(\a_1,\dots,\a_n)\subset\Mgn
\end{equation*}
is the locus of smooth marked curves $(C,x_1,\dots,x_n)\in\Mgn$ on which there exists a meromorphic differential $\omega$ whose associated divisor is $(\omega)=\sum_{i=1}^n\a_ix_i$. Equivalently, there is an isomorphism
\begin{equation*}
    \omega_C \simeq \mathcal{O}_C\left(\sum_{i=1}^n\a_ix_i\right),
\end{equation*}
with $\omega_C$ the canonical bundle of $C$. Let $\overline{\mathcal{H}}_g(\a_1,\dots,\a_n)$ denote the closure of $\mathcal{H}_g(\a_1,\dots,\a_n)$ in $\Mgnb$. Similarly, the space of residueless projectivized meromorphic differentials
\begin{equation*}
    \mathcal{H}_g^\textrm{res}(\a_1,\dots,\a_n)\subset\Mgn
\end{equation*}
is the locus of smooth marked curves with the same property as above, with the added condition that the residues of the meromorphic differential vanish at all poles. Its closure in $\Mgnb$ is denoted by $\overline{\mathcal{H}}_g^\textrm{res}(\a_1,\dots,\a_n)$. Let $N_{\a_{[n]}}:=\lvert\{i\in[n]:\a_i<0\}\rvert$ be the number of negative entries of $\a$. Then \cite{cmz22}
\begin{equation*}
    \operatorname{codim}\overline{\mathcal{H}}_g^\textrm{res}(\a_1,\dots,\a_n)=g-1+N_{\a_{[n]}}.
\end{equation*}
We use the same notation for the homology class $\Hres{g}{\a_1,\dots,\a_n}\in H_{2(2g-2-N_{\a_{[n]}})}\bigl(\Mgnb\bigr)$ and its Poincaré dual $\Hres{g}{\a_1,\dots,\a_n}\in H^{2(g-1+N_{\a_{[n]}})}\bigl(\Mgnb\bigr)$, which we call the cycle of residueless meromorphic differentials. \newline

The remainder of this subsection is a repeat of Section 1.2 of \cite{brz24}, which we include for completeness' sake. We give an overview of the properties of $\overline{\mathcal{H}}_g^\textrm{res}(\a_1,\dots,\a_n)$ from the point of view of multi-scale differentials. \newline

In \cite[Section 2]{bcggm19b} and \cite[Sections 3 and 4.1]{cmz22} the authors identify the space $\mathcal{H}^\textrm{res}_g(\a_1,\dots,\a_n)$ with the corresponding stratum $B_g^\textrm{res}(\a_1,\dots,\a_n)$ inside the projectivized Hodge bundle twisted by the polar part of $\a=(\a_1,\dots,\a_n)$,
\begin{equation*}
    \P\left(\pi_*\,\omega\Biggl(\,\sum_{i\in [n]\,\vert\,\a_i<0}\a_ix_i\Biggr)\right).
\end{equation*}
Here $\omega$ is the relative dualizing sheaf of the universal curve over $\Mgn$ and $\pi$ is the projection to $\Mgn$. The incidence variety compactification (IVC) of $B_g^\textrm{res}(\a_1,\dots,\a_n)$ is its closure inside the projectivized twisted Hodge bundle above. By taking a normalization of a blowup of the IVC, the authors constructed a proper smooth Deligne-Mumford stack, which we denote by $\overline{B}_g^\mathrm{res}(\a_1,\dots,\a_n)$, containing $B_g^\textrm{res}(\a_1,\dots,\a_n)$ as a dense open substack and whose boundary is a normal crossing divisor. The stack $\overline{B}_g^\textrm{res}(\a_1,\dots,\a_n)$ is a moduli stack for equivalence classes of projectivized multi-scale differentials with residue conditions, which we define after giving some preliminary definitions.

\begin{definition}
    An \emph{enhanced level graph} is a stable graph $\Gamma$ of genus $g$ with a set $L(\Gamma)$ of $n$ marked legs together with:
    \begin{itemize}
        \item[(1)] A total preorder on the set $V(\Gamma)$ of vertices. The preorder is described by a surjective level function $\ell:V(\Gamma)\rightarrow\{0,-1,\dots,-L\}$. An edge is called \emph{horizontal} if it is attached to vertices of the same level and \emph{vertical} otherwise.
        \item[(2)] A function $\kappa:E(\Gamma)\rightarrow\Z_{\ge0},e\mapsto\kappa_e$, called an \emph{enhancement}, such that $\kappa_e=0$ if and only if $e$ is horizontal.
    \end{itemize}
    For every level $-L\le j\le0$ let $C_{(j)}$ be the (possibly disconnected) stable curve obtained from $C$ by removing all irreducible components whose level is not $j$, and let $C_{(>j)}$ be the (possibly disconnected) stable curve obtained from $C$ by removing all irreducible components whose level is smaller than or equal to $j$.
\end{definition}
\begin{definition}
    Given a meromorphic differential $\omega$ on a smooth curve $C$ and a point $x\in C$, if $\omega$ has order $\operatorname{ord}_x\omega=\a\neq-1$ at $x$, then for a local coordinate $z$ in a neighbourhood of $x$ such that $z(x)=0$ we have, locally, $\omega=(cz^\a+O(z^{\a+1}))dz$ for some $c\in\C^*$. Then the $k=\lvert\a+1\rvert$ numbers $\zeta\in\C$ satisfying $\zeta^{\a+1}=c^{-1}$ determine $k$ projectivized vectors $\zeta\at{\frac{\partial}{\partial z}}{x}\in T_xC/\mathbb{R}_{>0}$ (if $\a\ge0$) or $-\zeta\at{\frac{\partial}{\partial z}}{x}\in T_xC/\mathbb{R}_{>0}$ (if $\a\le-2)$. These projectivized vectors called \emph{outgoing} and \emph{incoming prongs} of $\omega$, respectively. The sets of outgoing and incoming prongs at $x$ are denoted by $P^\text{out}_x$ and $P^\text{in}_x$.
\end{definition}
\begin{definition}
    Let $(C,x_1,\dots,x_n)\in\Mgnb$ be a stable marked curve and $(\a_1,\dots,\a_n)\in\Z^n$ satisfy $\sum_{i=1}^n\a_i=2g-2$. A \emph{multi-scale differential} on $C$ of profile $\a$ with \emph{zero residues} at $x_1,\dots,x_n\in C$ consists of:
    \begin{itemize}
        \item[(1)] A structure of enhanced level graph $(\Gamma_C,\ell,\kappa)$ on the dual graph $\Gamma_C$ on $C$, where a node is said to be horizontal or vertical if the corresponding edge is;
        \item[(2)] A collection of meromorphic differentials $\{\omega_v\}_{v\in V(\Gamma)}$, one on each irreducible component $C_v$ of $C$, holomorphic and non-vanishing outside of the marked points and nodes, such that:
        \begin{itemize}
            \item[(i)] $\operatorname{ord}_{x_i} \omega_{v(x_i)}=\a_i$ for each $1\le i\le n$;
            \item[(ii)] $\res_{x_i}\omega_{v(x_i)}=0$ for each $1\le i\le n$;
            \item[(iii)] For all pairs $v_1,v_2\in V(\Gamma)$, if $q_1\in C_{v_1}$ and $q_2\in C_{v_2}$ form a node $e\in E(\Gamma_C)$ then
            \begin{equation*}
                \operatorname{ord}_{q_1}\omega_{v_1} + \operatorname{ord}_{q_2}\omega_{v_2} = -2;
            \end{equation*}
            \item[(iv)] For all pairs $v_1,v_2\in V(\Gamma)$, if $q_1\in C_{v_1}$ and $q_2\in C_{v_2}$ form a node $e\in E(\Gamma_C)$ then $\ell(v_1)\ge\ell(v_2)$ if and only if $\operatorname{ord}_{q_1}\omega_{v_1}\ge-1$. Together with property (iii) this implies that $\ell(v_1)=\ell(v_2)$ if and only if $\operatorname{ord}_{q_1}\omega_{v_1}=\operatorname{ord}_{q_2}\omega_{v_2}=-1$;
            \item[(v)] For all pairs $v_1,v_2\in V(\Gamma)$, if $q_1\in C_{v_1}$ and $q_2\in C_{v_2}$ for a horizontal node $e\in E(\Gamma_C)$ (i.e. $\kappa_e=0$) then
            \begin{equation*}
                \res_{q_1}\omega_{v_1} + \res_{q_2}\omega_{v_2} = 0;
            \end{equation*}
            \item[(vi)] For every negative level $-L\le j\le-1$ of $\Gamma_C$ and for every connected component  $Y$ of $C_{(>j)}$:
            \begin{equation*}
                \sum_{q\in Y\cap C_{(j)}}\res_{q^-}\omega_{v(q^-)} = 0,
            \end{equation*}
            where $q^+\in Y$ and $q^-\in C_{(j)}$ form the vertical node $q\in Y\cap C_{(j)}$.
        \end{itemize}
        \item[(3)] A cyclic order reversing bijection $\sigma_q:P^\text{in}_{q^-}\rightarrow P_{q^+}^\text{out}$ for each vertical node $q$ formed by identifying $q^-$ on the upper level with $q^+$ on the lower level, where $\kappa_q=\lvert P^\text{in}_{q^-}\rvert=\lvert P^\text{out}_{q^+}\rvert$.
    \end{itemize}
\end{definition}
In the notation in \cite[Section 4.1]{cmz22}, condition (2)(vi) above is a reformulation of the $\mathfrak{R}$-global residue condition in the particular case when $\lambda$ is the partition of $H_p$ in one-element subsets and $\lambda_\mathfrak{R}=\lambda$. \newline

Lastly, there is an action by the universal cover of the torus $\C^L\rightarrow(\C^*)^L$ on multi-scale residueless differentials $\{\omega_v\}_{v\in V(\Gamma_C)}$ by rescaling the differentials with strictly negative levels and rotating the prong matchings between levels accordingly, producing fractional Dehn twists. The stabilizer $\text{Tw}_{\Gamma_C}$ of this action is called the \emph{twist group} of the enhanced level graph $\Gamma_C$. Two multi-scale differentials are defined to be equivalent if they differ by the action of $T_{\Gamma_C}:=\C^L/\text{Tw}_{\Gamma_C}$. A further quotient by the action of $\C^*$ rescaling the differentials on all levels and leaving all prong matchings untouched, produces equivalence classes of projectivized multi-scale residueless differentials. A special case of \cite[Proposition 4.2]{cmz22} corresponding to the choice of the $\mathfrak{R}$-global residue condition outlined above is the following proposition.
\begin{proposition}
    Given $(\a_1,\dots,\a_n)\in\Z^n$ and $g\ge0$ satisfying $\sum_{i=1}^n\a_i=2g-2$, there exists a proper Deligne-Mumford stack $\overline{B}^\mathrm{res}_g(\a_1,\dots,\a_n)$ containing $B^\mathrm{res}_g(\a_1,\dots,\a_n)$ as a dense open substack whose boundary is a normal crossing divisor. The stack $\overline{B}^\mathrm{res}_g(\a_1,\dots,\a_n)$ is a moduli stack for for families of equivalence classes of projectivized multi-scale residueless differentials. Its dimension is
    \begin{equation*}
        \dim\overline{B}^\mathrm{res}_g(\a_1,\dots,\a_n)=2g-2+n-N_{\a_{[n]}}.
    \end{equation*}
    Let $D_{(\Gamma,\ell,\kappa)}=D_\Gamma$ denote the closure of the stratum parametrizing multi-scale differentials whose enhanced level graph is $(\Gamma,\ell,\kappa)$. Then $D_\Gamma$ is a proper smooth closed substack of $\overline{B}_g^\mathrm{res}(\a_1,\dots,\a_n)$ of codimension
    \begin{equation*}
        \operatorname{codim}D_\Gamma=h+L,
    \end{equation*}
    where $h$ is the number of horizontal edges in $(\Gamma,\ell,\kappa)$ and $L+1$ is the number of levels.
\end{proposition}
There is a forgetful morphism $p:\overline{B}^\mathrm{res}_g(\a_1,\dots,\a_n)\rightarrow\Mgnb$ associating to each projectivized multi-scale residueless differential on a stable curve $C$ the stable curve itself. It restricts to an isomorphism $p:B_g^\mathrm{res}(\a_1,\dots,\a_n)\rightarrow\mathcal{H}_g^\mathrm{res}(\a_1,\dots,\a_n)$ of Deligne-Mumford stacks, and 
\begin{equation*}
    p_*[\overline{B}^\mathrm{res}_g(\a_1,\dots,\a_n)] = \Hres{g}{\a_1,\dots,\a_n}.
\end{equation*}

\subsection{Spin refinement}\label{subsec:spin}

A \emph{spin structure} on a smooth curve $C$ is a line bundle $L\rightarrow C$ satisfying $L^{\otimes2}\simeq\omega_C$. 
A pointed smooth curve $(C,x_1,\dots,x_n)\in\Mgn$ lies in $\mathcal{H}_g(\a_1^\prime,\dots,\a_n^\prime)$ if and only if there is a line bundle isomorphism
\begin{equation*}
    \omega_C \simeq \mathcal{O}_C\left(\sum_{i=1}^n\a_i^\prime x_i\right).
\end{equation*}
Consider the case where each $\a_i^\prime=2\a_i$ is an even number. Then a curve $(C,x_1,\dots,x_n)\in\mathcal{H}_g(2\a_1,\dots,2\a_n)$ carries a natural spin structure $L\rightarrow C$ defined by
\begin{equation*}
    L = \mathcal{O}_C\left(\sum_{i=1}^n\a_ix_i\right).
\end{equation*}
The \emph{parity} of a spin structure $L\rightarrow C$ is defined as the parity of $h^0(C,L)$, and is known to coincide with the Arf invariant of the associated theta characteristic. This quantity is invariant in families of spin structures over smooth curves (see \cite{mum71} or \cite{ati71}). Therefore one can decompose the space $\mathcal{H}_g(2\a_1,\dots,2\a_n)$ as a disjoint union of \textit{spin-parity components} according to even or odd parity of the spin structure defined above.
\begin{equation*}
    \mathcal{H}_g(2\a_1,\dots,2\a_n) = \mathcal{H}_g(2\a_1,\dots,2\a_n)^\mathrm{even} \sqcup \mathcal{H}_g(2\a_1,\dots,2\a_n)^\mathrm{odd}.
\end{equation*}
Then one can consider the respective closures of the spin-parity components in the moduli space of stable curves, and take their cycles:
\begin{equation}\label{eq:hspin}
    \Hspin{g}{2\a_1,\dots,2\a_n} := [\overline{\mathcal{H}}_g(2\a_1,\dots,2\a_n)^\mathrm{even}] - [\overline{\mathcal{H}}_g(2\a_1,\dots,2\a_n)^\mathrm{odd}].
\end{equation}
In what follows, we will consider the corresponding spin-parity cycles for residueless projectivized meromorphic differentials:
\begin{equation*}
    \HresSpin{g}{2\a_1,\dots,2\a_n} := [\overline{\mathcal{H}}^\mathrm{res}_g(2\a_1,\dots,2\a_n)^\mathrm{even}] - [\overline{\mathcal{H}}^\mathrm{res}_g(2\a_1,\dots,2\a_n)^\mathrm{odd}].
\end{equation*}
The idea to consider such cycles was taken from \cite{css21}. In \cite{won24}, based on a few assumptions, an algorithm to compute the spin-parity classes \hyperref[eq:hspin]{(\ref*{eq:hspin})} for explicit order conditions is given.

\subsection{The classes \texorpdfstring{$\HresSpin{g}{2\a_1,\dots,2\a_n}$}{} as a partial cohomological field theory}

The following is a generalization of the notion of a cohomological field theory (CohFT) of \cite{km94}, first introduced in \cite{lrz15}.
\begin{definition}
    A \emph{partial CohFT} is a system of linear maps $c_{g,n}:V^{\otimes n}\rightarrow H^\textrm{even}\bigl(\Mgnb\bigr)$ for all pairs $(g,n)$ of nonnegative integers satisfying $2g-2+n>0$, where $V$ is a finite-dimensional $\C$-vector space (the \emph{phase space}) together with a distinguished element $e_\mathbbm{1}\in V$ (the \emph{unit}) and a symmetric nondegenerate bilinear form $\eta\in(V^*)^{\otimes2}$ (the \emph{metric}), such that for any basis $\{e_\a\}_{\a\in A}$ of $V$ the following axioms are satisfied:
    \begin{itemize}
        \item[(i)] The maps $c_{g,n}$ are equivariant with respect to the $S_n$-action permuting the $n$ copies of $V$ and the $n$ marked points of $\Mgnb$, respectively.
        \item[(ii)] If $\pi:\Mb{g,n+1}\rightarrow\Mgnb$ is the morphism forgetting the last marked point, then for any $\a_1,\dots,\a_n\in A$ it holds that 
        \begin{equation*}
            \pi^*c_{g,n}\bigl(\otimes_{i=1}^ne_{\a_i}\bigr)=c_{g,n}\bigl(\otimes_{i=1}^ne_{\a_i}\otimes e_\mathbbm{1}\bigr),
        \end{equation*}
        Moreover, for any $\a,\b\in A$ it holds that $c_{0,3}(e_\a\otimes e_\b\otimes e_\mathbbm{1})=\eta(e_\a\otimes e_\b):=\eta_{\a\b}$, where we identify $H^*\bigl(\Mb{0,3}\bigr)=\C$.
        \item[(iii)] Let $g=g_1+g_2$ and $n=n_1+n_2$, with $2g_1-1+n_1>0$ and $2g_2-1+n_2>0$. Let $\mathrm{gl}:\Mb{g_1,n_1+1}\times\Mb{g_2,n_2+1}\rightarrow\Mb{g,n}$ be the morphism that glues two stable curves at their last marked points. Let $I$ and $J$ be the indexing sets of the points on the first and second curve that are not glued, so $I\sqcup J=[n]$ with $\lvert I\rvert=n_1,\lvert$ and $J\rvert=n_2$. Then for every $\a_1,\dots,\a_n\in A$:
        \begin{equation*}
            \mathrm{gl}^*c_{g,n} \bigl(\otimes_{i=1}^ne_{\a_i}) = c_{g_1,n_1+1}\bigl(\otimes_{i\in I}\,e_{\a_i}\otimes e_\mu\bigr)\eta^{\mu\nu}c_{g_2,n_2+1}\bigl(\otimes_{j\in J}\,e_{\a_j}\otimes e_\nu\bigr).
        \end{equation*}
    \end{itemize}
    A \emph{CohFT} is a partial CohFT that also satisfies the axiom
    \begin{enumerate}
        \item[(iv)] For every $\a_1,\dots,\a_n\in A$:
        \begin{equation*}
            \sigma^*c_{g,n+1}(\otimes_{i=1}^ne_{\a_i})=c_{g,n+2}(\otimes_{i=1}^ne_{\a_i}\otimes e_\mu\otimes e_\nu)\eta^{\mu\nu},
        \end{equation*}
        where $\sigma:\Mb{g,n+2}\rightarrow\Mb{g+1,n}$ is the morphism that identifies the last two marked points of a stable curve, creating a non-separating node that increases the genus of the curve.
    \end{enumerate}
\end{definition}
\begin{definition}\label{def:pcohft}
    A partial CohFT $c_{g,n}:V^{\otimes n}\rightarrow H^\textrm{even}\bigl(\Mgnb\bigr)$ is called \emph{homogeneous} if $V$ is a graded vector space with a homogeneous basis $\{e_\a\}_{\a\in A}$, with $q_\a:=\deg e_\a$, satisfying: the degree of the unit $e_\mathbbm{1}$ is $0$, the metric $\eta:V^{\otimes2}\rightarrow\C$ is homogeneous with $\delta:=-\deg\eta$, and there exist constants $\gamma$ and $r^\a$ for each $\a\in A$ such that:
    \begin{equation}\label{eq:hom}
        \operatorname{Deg}c_{g,n}(\otimes_{i=1}^ne_{\a_i}) + \pi_*c_{g,n}(\otimes_{i=1}^ne_{\a_i}\otimes r^\a e_\a) = \left(\sum_{i=1}^nq_{\a_i} + \gamma g-\delta\right)c_{g,n}(\otimes_{i=1}^ne_{\a_i}),
    \end{equation}
    where $\operatorname{Deg}:H^*\bigl(\Mgnb\bigr)\rightarrow H^*\bigl(\Mgnb\bigr)$ is the operator that acts on $H^i\bigl(\Mgnb\bigr)$ by multiplication by $\frac{i}{2}$ and $\pi:\Mb{g,n+1}\rightarrow\Mgnb$ is the morphism forgetting the last marked point. The constant $\gamma$ is the \emph{conformal dimension} of the partial CohFT.
\end{definition}
As remarked in \cite[Section 3]{br21}, sufficient conditions for the definition of a partial CohFT to make sense when the phase space $V$ is countably generated, for example when $V=\operatorname{span}(\{e_\a\}_{\a\in\Z})$, are the following:
\begin{itemize}
    \item The set $\{\a_n\in\Z:c_{g,n}(\otimes_{i=1}^ne_{\a_i})\neq0\}$ is finite for every $g,n$ in the stable range and $\a_1,\dots,\a_{n-1}\in\Z$;
    \item The metric $\eta_{\a\b}$ has a unique two-sided inverse $\eta^{\a\b}$.
\end{itemize}
Below we adapt \cite[Proposition 1.8]{brz24} to the spin context.

\begin{proposition}\label{prop:cohft}
    Let $V:=\mathrm{span}\left(\{e_{\a}\}_{\a\in\mathbb{Z}}\right)$, and $\eta$ be the nondegenerate symmetric bilinear form on $V$ satisfying $\eta_{\a\b}=\eta(e_\a\otimes e_\b)=\delta_{\a+\b,-1}$. Then the classes $c_{g,n}:V^{\otimes n}\rightarrow H^\mathrm{even}(\Mgnb)$ with
    \begin{equation*}
        c_{g,n}(e_{\a_1}\otimes\cdots\otimes e_{\a_n}):=\HresSpin{g}{2\a_1,\dots,2\a_n}\in H^{2(g-1+N_{\a_{[n]}})}(\Mgnb),\quad\a_1,\dots,\a_n\in\mathbb{Z},
    \end{equation*}
    form an infinite rank homogeneous partial CohFT with unit $e_0$ and metric $\eta$.
\end{proposition}
\begin{proof}
    Firstly, for fixed $g,n$ in the stable range and for $\a_1,\dots,\a_{n-1}\in\Z$ the set $\{\a_n\in\Z:c_{g,n}(\otimes_{i=1}^ne_{\a_i})\neq0\}$ consists of the single element $\a_n=g-1-\sum_{i=1}^{n-1}\a_i$. Moreover, the metric $\eta_{\a\b}=\delta_{\a+\b,-1}$ has a unique two-sided inverse $\eta^{\a\b}=\delta_{\a+\b,-1}$. Let us now check the three axioms of \hyperref[def:pcohft]{Definition \ref*{def:pcohft}}.
    \begin{itemize}
        \item[(i)] The $S_n$-equivariance of $c_{g,n}$ is clear from the definition.
        
        \item[(ii)] A spin structure on $\P^1$ is isomorphic to $\mathcal{O}(-1)$, so it has no global sections. Thus every spin structure on $\P^1$ is even, so in particular $c_{0,3}(e_\a\otimes e_\b\otimes e_0):=\Hspin{0}{2\a,2\b,0}=[\overline{\mathcal{H}}_0(2\a,2\b,0)]$ for every $\a,\b\in\Z$. On $(\P^1,0,\infty,1)\in\mathcal{M}_{0,3}$ a unique (up to $\C^*$ multiplication) meromorphic differential whose divisor is $2\a\cdot[0]+2\b\cdot[\infty]+0\cdot[1]$ exists if and only if $\b=-\a-1$, and is given by $\omega=z^\a dz$. Thus $c_{0,3}(e_\a\otimes e_\b\otimes e_0)=\delta_{\a+\b,-1} = \eta_{\a\b}$.

        Now let $2g-2+n>0$, $\a_1,\dots,\a_n\in\Z$ and $\pi:\Mb{g,n+1}\rightarrow\Mgnb$ be the morphism forgetting the last marked point. We show that
        \begin{equation}\label{eq:(ii)}
            \pi^*\HresSpin{g}{2\a_1,\dots,2\a_n} = \HresSpin{g}{2\a_1,\dots,2\a_n,0}.
        \end{equation}
        Let $p:\overline{B}_g^\mathrm{res}(2\a_1,\dots,2\a_n) \rightarrow\Mgnb$ and $\tilde{p}:\overline{B}_g^\mathrm{res}(2\a_1,\dots,2\a_n,0) \rightarrow\Mb{g,n+1}$ be the forgetful morphisms associating to each projectivized multi-scale residueless differential on a stable curve $C$ the stable curve itself. For $\star\in\{\mathrm{even},\mathrm{odd}\}$ denote the preimage $p^{-1}\bigl(\overline{\mathcal{H}}_g^\mathrm{res}(2\a_1,\dots,2\a_n)^\star\bigr)$ by $\overline{B}_g^\mathrm{res}(2\a_1,\dots,2\a_n)^\star$, and similarly denote $\tilde{p}^{-1}\bigl(\overline{\mathcal{H}}_g^\mathrm{res}(2\a_1,\dots,2\a_n,0)^\star\bigr)$ by $\overline{B}_g^\mathrm{res}(2\a_1,\dots,2\a_n,0)^\star$. Consider the lifting map $\tilde{\pi}:\overline{B}_g^\mathrm{res}(2\a_1,\dots,2\a_n,0)\rightarrow\overline{B}_g^\mathrm{res}(2\a_1,\dots,2\a_n)$ of $\pi$ through $\tilde{p}$ and $p$. Given a meromorphic differential on a curve lying in $B_g^\mathrm{res}(2\a_1,\dots,2\a_n)$, the parity of the associated spin structure is not affected by adding a marked point on which the differential does not have a zero or pole. Therefore $\tilde{\pi}$ restricts to a morphism $\tilde{\pi}^\star:\overline{B}_g^\mathrm{res}(2\a_1,\dots,2\a_n,0)^\star\rightarrow\overline{B}_g^\mathrm{res}(2\a_1,\dots,2\a_n)^\star$. Let $X^\star$ be the fiber product of $\overline{B}_g^\mathrm{res}(2\a_1,\dots,2\a_n)^\star$ and $\Mb{g,n+1}$ over $\Mgnb$ with respective projections $a^\star$ and $b^\star$, and $f^\star:\overline{B}_g^\mathrm{res}(2\a_1,\dots,2\a_n,0)^\star\rightarrow X^\star$ the induced morphism:
        \[\begin{tikzcd}
            \overline{B}_g^\mathrm{res}(2\a_1,\dots,2\a_n,0)^\star \\
            & X^\star & \overline{B}_g^\mathrm{res}(2\a_1,\dots,2\a_n)^\star \\
            & \Mb{g,n+1} & \Mgnb 
            \arrow["{f^\star}" description, dotted, from=1-1, to=2-2]
            \arrow["{\tilde{\pi}^\star}", bend left, from=1-1, to=2-3]
            \arrow["{\tilde{p}}"', bend right, from=1-1, to=3-2]
            \arrow["{a^\star}", from=2-2, to=2-3]
            \arrow["{b^\star}"', from=2-2, to=3-2]
            \arrow["p", from=2-3, to=3-3]
            \arrow["\pi"', from=3-2, to=3-3]
        \end{tikzcd}\]
        This is the same situation as in the argument given in \cite[Proposition 1.8]{brz24}, from which we deduce
        \begin{equation*}
            \pi^*[\overline{\mathcal{H}}_g^\mathrm{res}(2\a_1,\dots,2\a_n)^\star] = [\overline{\mathcal{H}}_g^\mathrm{res}(2\a_1,\dots,2\a_n,0)^\star], \quad \star\in\{\mathrm{even},\mathrm{odd}\}.
        \end{equation*}
        Then equation \hyperref[eq:(ii)]{(\ref*{eq:(ii)})} follows.
        
        \item[(iii)] Now consider the gluing morphism $\mathrm{gl}:\Mb{g_1,n_1+1}\times\Mb{g_2,n_2+1} \rightarrow \Mgnb$, where $g=g_1+g_2,\,n=n_1+n_2$ and $I\sqcup J=[n]$ with $\lvert I\rvert=n_1,\lvert J\rvert = n_2$. We show that
        \begin{equation}\label{eq:(iii)}
            \mathrm{gl}^*\HresSpin{g}{2\a_1,\dots,2\a_n} = \HresSpin{g_1}{2\a_I,2\b} \otimes \HresSpin{g_2}{2\a_J,-2\b-2},
        \end{equation}
        where $2\a_I:=(2\a_i)_{i\in I}$ and $2\a_J:=(2\a_j)_{j\in J}$ and $\b=g_1-1-\sum_{i\in I}\a_i=-g_2+\sum_{j\in J}\a_j$. By the same reasoning as \cite[Proposition 1.8]{brz24} in the proof of the non-spin analogue of \hyperref[eq:(iii)]{(\ref*{eq:(iii)})}, there is a morphism 
        \begin{equation*}
            \widetilde{\mathrm{gl}}: \overline{B}_{g_1}^\mathrm{res}(2\a_I,2\b) \times \overline{B}_{g_2}^\mathrm{res}(2\a_J,-2\b-2) \rightarrow \overline{B}_g^\mathrm{res}(2\a_1,\dots,2\a_n)
        \end{equation*}
        lifting the morphism $\mathrm{gl}$, which is an isomorphism onto its image. Therefore 
        \begin{equation*}
            \mathrm{gl}^{-1} \bigl(\overline{\mathcal{H}}_g^\mathrm{res}(2\a_1,\dots,2\a_n)\bigr) = \overline{\mathcal{H}}_{g_1}(2\a_I,2\b)\times\overline{\mathcal{H}}_{g_2}(2\a_J,-2\b-2).
        \end{equation*}
        In what follows, for $\star\in\{\mathrm{even},\mathrm{odd}\}$ denote $\overline{\mathcal{H}}_g^\mathrm{res}(2\a_1,\dots,2\a_n)^\star$ by $\overline{\mathcal{H}}^\star$, and similarly denote $\overline{\mathcal{H}}_{g_1}^\mathrm{res}(2\a_I,2\b)^\star$ and $\overline{\mathcal{H}}_{g_2}^\mathrm{res}(2\a_I,-2\b-2)^\star$ by $\overline{\mathcal{H}}_1^\star$ and $\overline{\mathcal{H}}_2^\star$. Moreover, denote the smooth parts of $\overline{\mathcal{H}}_1^\star$ and $\overline{\mathcal{H}}_2^\star$ by $\mathcal{H}_1^\star\subset\overline{\mathcal{H}}_1^\star$ and $\mathcal{H}_2^\star\subset\overline{\mathcal{H}}_2^\star$. By \cite[Proposition 5.3]{css21}, the spin parity of a multi-scale differential over a stable curve of compact type is given by the sum of the parity of each vertex of the underlying stable graph, modulo 2. Thus 
        \begin{gather*}
            \mathrm{gl}\bigl((\mathcal{H}_1^\mathrm{even}\times\mathcal{H}_2^\mathrm{even}) \sqcup (\mathcal{H}_1^\mathrm{odd}\times\mathcal{H}_2^\mathrm{odd})\bigr) \subset \overline{\mathcal{H}}^\mathrm{even}, \\
            \mathrm{gl}\bigl((\mathcal{H}_1^\mathrm{even}\times\mathcal{H}_2^\mathrm{odd}) \sqcup (\mathcal{H}_1^\mathrm{odd}\times\mathcal{H}_2^\mathrm{even})\bigr) \subset \overline{\mathcal{H}}^\mathrm{odd},
        \end{gather*}
        and therefore 
        \begin{gather*}
            \mathrm{gl}^{-1}\bigl(\overline{\mathcal{H}}^\mathrm{even}\bigr) = (\overline{\mathcal{H}}_1^\mathrm{even}\times\overline{\mathcal{H}}_2^\mathrm{even}) \sqcup (\overline{\mathcal{H}}_1^\mathrm{odd}\times\overline{\mathcal{H}}_2^\mathrm{odd}), \\
            \mathrm{gl}^{-1}\bigl(\overline{\mathcal{H}}^\mathrm{odd}\bigr) = (\overline{\mathcal{H}}_1^\mathrm{even}\times\overline{\mathcal{H}}_2^\mathrm{odd}) \sqcup (\overline{\mathcal{H}}_1^\mathrm{odd}\times\overline{\mathcal{H}}_2^\mathrm{even}).
        \end{gather*}
        These considerations lead one to conclude that
        \begin{gather*}
            \mathrm{gl}^*[\overline{\mathcal{H}}^\mathrm{even}] = a\,[\overline{\mathcal{H}}_1^\mathrm{even}]\otimes[\overline{\mathcal{H}}_2^\mathrm{even}] + b\,[\overline{\mathcal{H}}_1^\mathrm{odd}]\otimes[\overline{\mathcal{H}}_2^\mathrm{odd}], \\
            \mathrm{gl}^*[\overline{\mathcal{H}}^\mathrm{odd}] = c\,[\overline{\mathcal{H}}_1^\mathrm{even}]\otimes[\overline{\mathcal{H}}_2^\mathrm{odd}] + d\,[\overline{\mathcal{H}}_1^\mathrm{odd}]\otimes[\overline{\mathcal{H}}_2^\mathrm{even}],
        \end{gather*}
        for some constants $a,b,c,d$. By the local coordinate computation in \cite[Proposition 1.8]{brz24}, the intersection of $\overline{\mathcal{H}}$ with the image of $\mathrm{gl}$ along $\overline{\mathcal{H}}_1\times\overline{\mathcal{H}}_2$ is generically transversal. Thus, in particular, the intersection of $\overline{\mathcal{H}}^\mathrm{even}$ with the image along $\overline{\mathcal{H}}_1^\mathrm{even}\times\overline{\mathcal{H}}_2^\mathrm{even}$ is generically transversal, so $a=1$. Similarly $b=c=d=1$. Equation \hyperref[eq:(iii)]{(\ref*{eq:(iii)})} then follows.
    \end{itemize}
    Lastly, the fact that $\operatorname{codim}\overline{\mathcal{H}}_g^\mathrm{res}(2\a_1,\dots,2\a_n)=g-1+N_{\a_{[n]}}$ implies that 
    \begin{equation*}
        \operatorname{Deg}c_{g,n}(\otimes_{i=1}^ne_{\a_i})=(g-1+N_{\a_{[n]}})c_{g,n}(\otimes_{i=1}^ne_{\a_i}).
    \end{equation*}
    By choosing the constants $r^\a=0$ for all $\a\in\Z$, $q_\a=0$ if $\a\ge0$ and $q_\a=1$ if $\a\le-1$, and $\gamma=\delta=1$ (which are compatible with $\deg e_0=0$ and $\deg\eta=-\delta$), we see that equation \hyperref[eq:hom]{(\ref*{eq:hom})} is satisfied. Thus the partial CohFT $c_{g,n}$ is homogeneous.
\end{proof}

\begin{center}\section{The associated DR hierarchy}\label{sec:dr}\end{center}
\subsection{The DR hierarchy construction}

In \cite{bur15} a construction associating an integrable Hamiltonian system of evolutionary PDEs to a CohFT was given. In \cite{bdgr18} the construction was shown to be valid for partial CohFTs, and in \cite{br21} the first example of the DR hierarchy associated to an infinite rank partial CohFT was given. We outline this construction below. \newline

For $i=1,\dots,n$ let $\psi_i\in H^2(\Mgnb)$ denote the $i^\mathrm{th}$ $\psi$-class, i.e. the first Chern class of the line bundle $\mathbb{L}_i\rightarrow\Mgnb$ whose fiber at a stable curve is the cotangent line at the $i^\mathrm{th}$ marked point. For $j=0,\dots,g$ let $\lambda_j\in H^{2j}(\Mgnb)$ denote the $j^\mathrm{th}$ $\lambda$-class, i.e. the $j^\mathrm{th}$ Chern class of the Hodge bundle $\Omega\Mgnb\rightarrow\Mgnb$ mentioned in the introduction. For $a_1,\dots,a_n\in\Z$ satisfying $\sum_{i=1}^na_i=0$ let $\DR_g(a_1,\dots,a_n)\in H^{2g}(\Mgnb)$ denote the Poincaré dual cohomology class of the untwisted \emph{double ramification (DR) cycle}. This can be constructed \cite{li02,gv05} as the pushforward, through the forgetful map to $\Mgnb$, of the virtual fundamental class of the moduli space of projectivized stable maps to $\P^1$ relative to $0$ and $\infty$, with ramification profile $a_1,\dots,a_n$ at the marked points. There are various other constructions of the DR cycle on $\Mgnb$. A list of these constructions, and the equivalences between them, can be found in \cite[Section 1.6]{hs21}. \newline

Let $\mathcal{M}_{g,n}^\mathrm{ct}\subset\Mgnb$ denote the moduli space of stable curves of compact type. The restriction $\at{\DR_g(a_1,\dots,a_n)}{\mathcal{M}_{g,n}^\mathrm{ct}}$ is a homogeneous polynomial in $a_1,\dots,a_n$ of degree $2g$ with coefficients in $H^{2g}(\mathcal{M}_{g,n}^\mathrm{ct})$ \cite{jppz17}. Hain's formula \cite{hai13} gives a useful expression for $\at{\DR_g(a_1,\dots,a_n)}{\mathcal{M}_{g,n}^\mathrm{ct}}$ which we will use in our subsequent computations. Moreover, the fact that $\lambda_g$ vanishes on $\Mgnb\setminus\mathcal{M}_{g,n}^\mathrm{ct}$ (see \cite[Section 0.4]{fp00}) implies that the cohomology class $\lambda_g\DR_g(a_1,\dots,a_n)\in H^{4g}(\Mgnb)$ is a homogeneous polynomial in the variables $a_1,\dots,a_n$. In the particular case when $n=2$ there is a useful formula for $\lambda_g\DR_g(-a,a)$ from \cite{bhs22} which we will also use later. \newline

For a countable indexing set $A$, consider the formal variables $\{u_k^\a:\a\in A,k\ge0\}$ together with an additional variable $\e$. Heuristically, given a vector space $V=\operatorname{span}(\{e_\a\}_{\a\in A})$, one can think of $u^\a=u^\a(x)$ as the component of a loop $u:S^1\rightarrow V$ along $e_\a$, with $x$ being the coordinate on the circle $S^1$, and of $u^a_k$ as the $k^\mathrm{th}$ derivative of $u^\a$ with respect to $x$ (the variable $\e$ plays the role of counting how many $x$-derivatives appear, in an appropriate sense). This is just a heuristic interpretation as we work in a formal algebraic setting by describing an appropriate ring of functions for the loop space of $V$ as follows. The ring $\C[[u_*^*]]$ is graded by the differential grading $\deg_{\d_x}u^\a_k:=k$, with the degree $d$ part being denoted by $\mathcal{A}_A^{[d]}$. Define $\mathcal{A}_A:=\oplus_{d\ge0}\mathcal{A}_A^{[d]}$ and $\widehat{\mathcal{A}}_A:=\mathcal{A}_A[[\e]]$. There is an operator $\d_x:\widehat{\mathcal{A}}_A\rightarrow\widehat{\mathcal{A}}_A$ corresponding heuristically to differentiation with respect to $x$:
\begin{equation*}
    \d_xf := \sum_{k\ge0}u^\a_{k+1}\frac{\d f}{\d u_k^\a},\quad f\in\widehat{\mathcal{A}}_A.
\end{equation*}
Let $\widehat{\Lambda}_A:=\widehat{\mathcal{A}}_A\Big/\left(\d_x\widehat{\mathcal{A}}_A\oplus\C[[\e]]\right)$. The spaces $\widehat{\mathcal{A}}_A$ and $\widehat{\Lambda}_A$ are called the spaces of \emph{differential polynomials} and of \emph{local functionals}, respectively. By assigning $\deg_{\d_x}\e=-1$, denote the degree $d$ parts of $\widehat{\mathcal{A}}_A$ and $\widehat{\Lambda}_A$ by $\widehat{\mathcal{A}}_A^{[d]}$ and $\widehat{\Lambda}_A^{[d]}$, respectively. Note that these constructions are pertinent to the case where the indexing set $A$ is countable, while for finite $A$ the corresponding definitions of $\widehat{\mathcal{A}}_A$ and $\widehat{\Lambda}_A$ are slightly different, see \cite[Section 2.1]{ros17}. The natural projection is denoted by $\int dx:\widehat{\mathcal{A}}_A\rightarrow\widehat{\Lambda}_A$, suggesting that the quotient corresponds to the possibility of integrating by parts on $S^1$. For $f\in\widehat{\mathcal{A}}_A$ we will use the shorthand notation $\overline{f}:=\int f dx$. For each $\a\in A$, the variational derivative $\frac{\delta}{\delta u^\a}:\widehat{\Lambda}_A\rightarrow\widehat{\mathcal{A}}_A$ is the operator given by
\begin{equation*}
    \frac{\delta\overline{f}}{\delta u^\a} := \sum_{k\ge0}(-\d_x)^k\frac{\d f}{\d u^\a_k},\quad \overline{f}\in\widehat{\Lambda}_A.
\end{equation*}
One can define Poisson brackets on the space of local functionals $\widehat{\Lambda}_A$. A result of \cite{get02} states that, after a suitable coordinate transformation, all Poisson brackets on this space are of the form
\begin{equation}\label{eq:poisson}
    \{\overline{f},\overline{g}\} := \int\left(\frac{\delta\overline{f}}{\delta u^\a}\pi^{\a\b}\frac{\delta\overline{g}}{\delta u^\b}\right)dx, \quad \overline{f},\overline{g}\in\widehat{\Lambda}_A,
\end{equation}
with $\pi^{\a\b}$ a constant, symmetric and nondegenerate bilinear form. \newline

Differential polynomials and local functionals can also be described using another set of formal variables $\{p_a^\a:\a\in A,a\in\Z\}$, corresponding heuristically to the Fourier components of $u^\a=u^\a(x)$. These are defined by
\begin{equation*}
    u^\a_k=\sum_{a\in\Z}(ia)^kp_a^\a e^{iax},\quad \a\in A,
\end{equation*}
The Poisson bracket \hyperref[eq:poisson]{(\ref*{eq:poisson})} takes the form $\{p_a^\a,p_b^\b\}=ia\pi^{\a\b}\delta_{a+b,0}$ in the Fourier variables. \newline

Let $c_{g,n}:V^{\otimes n}\rightarrow H^*(\Mgnb)$ be a partial CohFT with phase space $V=\operatorname{span}(\{e_\a\}_{\a\in A})$, unit $e_\mathbbm{1}$ and metric $\eta$. The \emph{Hamiltonian densities} of the DR hierarchy associated to this partial CohFT \cite{br16} are the generating series $g_{\a,-1}:=\eta_{\a\mu}u^\mu$ for all $\a\in A,$ and
\begin{align*}
    g_{\a,d} = & \hspace{-1em}\sum_{\substack{g\ge0,n\ge1 \\ 2g-1+n>0}} \frac{\e^{2g}}{n!} \sum_{k_1,\dots,k_n\ge0}\, \prod_{i=1}^nu_{k_i}^{\a_i} \\
    & \times\coef\left(\,\int_{\Mb{g,n+1}}\hspace{-1.5em}\DR_g\bigl(-\sum_{i=1}^na_i,a_1,\dots,a_n\bigr)\lambda_g\psi_1^d\,c_{g,n}(e_\a\otimes\otimes_{i=1}^ne_{\a_i})\right)\in\widehat{\mathcal{A}}_A^{[0]},
\end{align*}
for all $\a\in A$ and $d\in\Z_{\ge0}$ (we remind that we sum over the repeated Greek indices $\a_i$). The notation $\coef(P)$ indicates the coefficient of $a_1^{k_1}\cdots a_n^{k_n}$ of the polynomial $P\in\mathbb{Q}[a_1,\dots,a_n]$. In the Fourier variables $\{p_a^\a:\a\in A,a\in\Z\}$, the Hamiltonian densities are
\begin{align}
    \nonumber g_{\a,d} = & \hspace{-1em}\sum_{\substack{g\ge0,n\ge1 \\ 2g-1+n>0}} \frac{(-\e^2)^g}{n!} \sum_{a_1,\dots,a_n\in\Z}\left(\,\int_{\Mb{g,n+1}}\hspace{-1.5em}\DR_g\bigl(-\sum_{i=1}^na_i,a_1,\dots,a_n\bigr)\lambda_g\psi_1^d\,c_{g,n}(e_\a\otimes\otimes_{i=1}^ne_{\a_i})\right) \\
    \label{eq:gadfourier} & \times \left(\prod_{i=1}^np_{a_i}^{\a_i}\right)e^{i\sum_{i=1}^na_ix}.
\end{align}

The \emph{Hamiltonians} of the DR hierarchy associated to the above partial CohFT are the local functionals $\overline{g}_{\a,d}\in\widehat{\Lambda}_A^{[0]},\a\in A,d\ge-1$. In Fourier variables this corresponds to taking the constant term in $x$ of $g_{\a,d}$. It was proved in \cite{bur15} that the Hamiltonians of the DR hierarchy are in involution with respect to the Poisson bracket \hyperref[eq:poisson]{(\ref*{eq:poisson})} on $\widehat{\Lambda}_A$, where we take $\pi$ to be the metric $\eta$ of the CohFT:
\begin{equation*}
    \{\overline{g}_{\a_1,d_1},\overline{g}_{\a_2,d_2}\}=0, \qquad \a_1,\a_2\in A, \quad d_1,d_2\ge-1.
\end{equation*}
This implies that the following infinite system of evolutionary PDEs, called the \emph{DR hierarchy},
\begin{equation*}
    \dd{u^\a}{t^\b_d} = \eta^{\a\mu}\d_x\frac{\delta\overline{g}_{\b,d}}{\delta u^\mu},\quad \a,\b\in A,\quad d\ge0,
\end{equation*}
satisfies the compatibility conditions
\begin{equation*}
    \frac{\d}{\d t^{\b_2}_{d_2}}\dd{u^\a}{t^{\b_1}_{d_1}} = \frac{\d}{\d t^{\b_1}_{d_1}}\dd{u^\a}{t^{\b_2}_{d_2}}, \quad \a,\b_1,\b_2\in A,\quad d_1,d_2\ge0.
\end{equation*}

\subsection{The DR hierarchy associated to \texorpdfstring{$\HresSpin{g}{2\a_1,\dots,2\a_n}$}{}}\label{sec:DRcohft}

Let us apply the above construction to the partial CohFT 
\begin{equation*}
    c_{g,n}(\otimes_{i=1}^ne_{\a_i})=\HresSpin{g}{2\a_1,\dots,2\a_n} 
\end{equation*}
of \hyperref[prop:cohft]{Proposition \ref*{prop:cohft}}. The indexing set in this case is $A=\Z$. The Hamiltonian densities are given by $g_{\a,-1}=\eta_{\a\mu}u^\mu=u_{-\a-1}$ for all $\a\in\Z$, and
\begin{align}\label{eq:densities} 
    g_{\a,d} = & \hspace{-1em}\sum_{\substack{g\ge0,n\ge1 \\ 2g-1+n>0}} \frac{\e^{2g}}{n!} \sum_{k_1,\dots,k_n\ge0}\, \prod_{i=1}^nu_{k_i}^{\a_i} \\
    \nonumber & \hspace{-2.5em} \times\coef\left(\,\int_{\Mb{g,n+1}}\hspace{-1.5em}\DR_g\bigl(-\sum_{i=1}^na_i,a_1,\dots,a_n\bigr)\lambda_g\psi_1^d\,\HresSpin{g}{2\a,2\a_1,\dots,2\a_n}\right)\in\widehat{\mathcal{A}}_\Z^{[0]}
\end{align}
for all $\a\in\Z$ and $d\ge0$. The DR hierarchy in this case is the infinite system of compatible evolutionary PDEs
\begin{equation}\label{eq:dr}
    \dd{u^\a}{t_d^\b} = \d_x\frac{\delta\overline{g}_{\beta,d}}{\delta u^{-\a-1}},\quad\a,\b\in\Z,\quad d\ge0.
\end{equation}
Let $q_\a=0$ if $\a\ge0$ and $q_\a=1$ if $\a\le-1$. It is useful to introduce the following triple grading on $\widehat{\mathcal{A}}_\Z$:
\begin{equation*}
    \overline{\deg}\,u^\a_k:=(k,1-q_\a,-2\a)=\begin{cases}(k,1,-2\a),&\a\ge0,\\(k,0,-2\a),&\a\le-1,\end{cases}\qquad\overline{\deg}\,\e:=(-1,0,1).
\end{equation*}
Consider, moreover, the Euler differential operator on $\widehat{\mathcal{A}}_\Z$ given by
\begin{equation*}
    \widehat{E}:=\sum_{k\ge0}(1-q_\a)u^\a_k\frac{\d}{\d u^\a_k}.
\end{equation*}
\begin{proposition}\label{prop:deg}
    The Hamiltonian densities \hyperref[eq:densities]{(\ref*{eq:densities})} satisfy
    \begin{equation*}
        \overline{\deg}\,g_{\a,d}=(0,d+1+q_\a,2\a+2)=\begin{cases}(0,d+1,2\a+2),&\a\ge0,\\(0,d+2,2\a+2),&\a\le-1.\end{cases}
    \end{equation*}
\end{proposition}
\begin{proof}
    The first entry in $\overline{\deg}$ corresponds to the differential degree $\deg_{\d_x}$, so the first entry of $\overline{\deg}\,g_{\a,d}$ corresponds to the previously mentioned fact $g_{\a,d}\in\widehat{A}_\Z^{[0]}$. The second entry of $\overline{\deg}\,g_{\a,d}$ follows from $\widehat{E}(g_{\a,d})=(d+1+q_\a)g_{\a,d}$, which can be deduced from dimension counting in the expression \hyperref[eq:densities]{(\ref*{eq:densities})}. The third entry is due to the fact that $\overline{\mathcal{H}}_g^\mathrm{res}(2\a,2\a_1,\dots,2\a_n)=0$ unless $-\sum_{i=1}^n2\a_i+2g=2\a+2$.
\end{proof}

We now reduce the system \hyperref[eq:dr]{(\ref*{eq:dr})} to only include cycles of meromorphic differentials with two zeros and any number of poles. This reduction is almost the same as in \cite[Proposition 3.1]{brz24}.
\begin{proposition}
    The subset $\bigl\{\frac{\d}{\d t_0^\b}\bigr\}_{\b\ge0}$ of flows of the DR hierarchy \hyperref[eq:dr]{(\ref*{eq:dr})} preserves the submanifold $\{u^\a_k=0:\a\ge0,k\ge0\}$.
\end{proposition}
\begin{proof}
    The statement to be proved is equivalent to
    \begin{equation*}
        \at{\dd{u^\a}{t_0^\b}}{u_*^{\ge0}} = \at{\d_x\frac{\delta\overline{g}_{\beta,0}}{\delta u^{-\a-1}}}{u_*^{\ge0}} = 0 \quad \text{for all }\a\ge0,\,\b \ge0.
    \end{equation*}
    So fix $\a\ge0$ and $\b\ge0$. By definition $\overline{\deg}\,u^{-\a-1}_k=(k,0,2\a+2)$ for all $k\ge0$, and by \hyperref[prop:deg]{Proposition \ref*{prop:deg}} $\overline{\deg}\,g_{\b,0}=(0,1,2\b+2)$. Thus
    \begin{equation*}
        \overline{\deg}\,\dd{g_{\b,0}}{u^{-\a-1}_k} = (-k,1,2\b-2\a), \quad k\ge0,
    \end{equation*}
    and so
    \begin{equation}\label{eq:deggb}
        \overline{\deg}\,\frac{\delta\overline{g}_{\b,0}}{\delta u^{-\a-1}} = (0,1,2\b-2\a).
    \end{equation}
    Notice that $\overline{\deg}\,u^\gamma_k=(k,1,-2\gamma)$ for all $\g,k\ge0$, so by looking at the second entry of $\overline{\deg}$ we conclude
    \begin{equation*}
        \at{\frac{\delta\overline{g}_{\beta,0}}{\delta u^{-\a-1}}}{u_*^{\ge0}} = 0,
    \end{equation*}
    as desired.
\end{proof}

Motivated by this, introduce some new notation. Let $u_\a^{(k)}:=u_k^{-\a}$, $u_\a:=u^{(0)}_\a$, and $t^\a:=t^{\a-1}_0$ for $\a\ge1,k\ge0$. Consider the three gradings on the new ring of differential polynomials $\R_u:=\C[u_*^{(*)}]$:
\begin{itemize}
    \item The differential grading $\deg_{\d_x}$ is given by $\deg_{\d x}u^{(k)}_\a:=k$. The corresponding homogeneous component of $\R_u$ of degree $d$ is denoted by $\R_u^{[d]}$.
    \item The grading $\deg$ is given by $\deg\,u_\a^{(k)}:=2\a+k$.
    \item The grading $\widetilde{\deg}$ is given by $\widetilde{\deg}\,u_\a^{(k)}:=1$. The corresponding homogeneous component of $\R_u$ of degree $d$ is denoted by $\R_{u;d}$. Moreover, denote $\bigoplus_{d\ge l}\R_{u;d}$ by $\R_{u;\ge l}$.
\end{itemize}
Let $\R^\mathrm{ev}_u:=\bigoplus_{d\ge0}\R_u^{[2d]}$, and extend the three gradings to $\widehat{\R}_u:=\R_u[\e]$ by
\begin{equation*}
    \deg_{\d x}\e:=-1,\qquad \deg\,\e:=0,\qquad \widetilde{\deg}\,\e:=0.
\end{equation*}
Let $\widehat{\R}^\mathrm{ev}_u:=\R^\mathrm{ev}_u[\e]$.
\begin{proposition}
    For integers $\a,\b$, consider the generating series
    \begin{align}
        \label{eq:pspin}&\Pspin{\a\b}:=\sum_{g\ge0,n\ge1}\frac{\e^{2g}}{n!}\sum_{k_1,\dots,k_n\ge0}\,\prod_{i=1}^nu_{\a_i}^{(k_i)} \\ 
        \nonumber & \times \coef \left(\int_{\DR_g\bigl(-\sum_{i=1}^na_i,0,a_1,\dots,a_n\bigr)}\lambda_g\HresSpin{g}{2\a-2,2\b-2,-2\a_1,\dots,-2\a_n}\right).
    \end{align}
    Then $\Pspin{\a\b}\in\widehat{\R}_{u;\ge1}^{\mathrm{ev};[0]}$ and $\deg\Pspin{\a\b}=2\a+2\b-2$, and the system of evolutionary equations
    \begin{equation}\label{eq:redr}
        \dd{u_\a}{t^\b}=\d_x\Pspin{\a\b},\qquad\a,\b\ge1,
    \end{equation}
    satisfies the compatibility conditions $\frac{\d}{\d t^{\beta_2}}\dd{u_\a}{t^{\beta_1}}=\frac{\d}{\d t^{\beta_1}}\dd{u_\a}{t^{\beta_2}}$ for all $\a,\b_1,\b_2\ge1$. Moreover, for $\b\ge1$ we have $\Pspin{1,\b}-u_\b\in\operatorname{im}(\d_x^2)$.
\end{proposition}
\begin{proof}
    System \hyperref[eq:redr]{(\ref*{eq:redr})} is the restriction of system \hyperref[eq:dr]{(\ref*{eq:dr})} to the submanifold $\{u^\a_k=0:\a,k\ge0\}$ in the new coordinates $u_*^{(*)}$. To see this, one may switch to Fourier coordinates $\{p_a^\a:\a\in\Z,a\in\Z\}$ to write the Hamiltonian densities in the form \hyperref[eq:gadfourier]{(\ref*{eq:gadfourier})}, and use the fact that
    \begin{equation*}
        \frac{\delta}{\delta u^{\a-1}}=\sum_{a\in\Z}e^{-iax}\frac{\d}{\d p_a^{\a-1}}.
    \end{equation*}
    Thus the compatibility conditions follow from the compatibility conditions of the original system. The fact that $\Pspin{\a\b}\in\widehat{\R}_{u;\ge1}^{\mathrm{ev};[0]}$ follows from previously mentioned properties of the DR cycle and DR hierarchy, while $\deg\Pspin{\a\b}=2\a+2\b-2$ is equivalent to equation \hyperref[eq:deggb]{(\ref*{eq:deggb})} in the new coordinates. Lastly, $\Pspin{1,\b}-u_\b\in\operatorname{im}(\d_x^2)$ because for $(g,n)\neq(0,1)$ we have
    \begin{align*}
        & \int_{\Mb{g,n+2}}\hspace{-1.9em}\DR_g\bigl(-\sum_{i=1}^na_i,0,a_1,\dots,a_n\bigr)\lambda_g\HresSpin{g}{0,2\b-2,-2\a_1,\dots,-2\a_n} \\
        = & \int_{\Mb{g,n+2}}\hspace{-1.5em}\pi_*\Bigl(\DR_g\bigl(-\sum_{i=1}^na_i,0,a_1,\dots,a_n\bigr)\Bigr)\lambda_g\HresSpin{g}{2\b-2,-2\a_1,\dots,-2\a_n},
    \end{align*}
    where $\pi:\Mb{g,n+2}\rightarrow\Mb{g,n+1}$ forgets the first marked point, and the fact that \newline $\pi_*\Bigl(\DR_g\bigl(-\sum_{i=1}^na_i,0,a_1,\dots,a_n\bigr)\Bigr)\lambda_g$ is a polynomial in $a_1,\dots,a_n$ divisible by $(\sum_{i=1}^na_i)^2$ \cite[Lemma 5.1]{bdgr18}.
\end{proof}

We now focus on the reduced system \hyperref[eq:redr]{(\ref*{eq:redr})}. In the following two results, we determine some initial data for of the differential polynomials $\Pspin{\a\b}$. As it will turn out, this data will be enough to completely determine the system once we switch to normal coordinates (\hyperref[thm:bkpr]{Theorem \ref*{thm:bkpr}}). We adopt the convention $u_\a^{(*)}=0$ for $\a\le 0$, and use the notation $u_{\le\gamma}^{(*)}$ to denote the collection of all variables $u_\a^{(*)}$ with $\alpha\le\gamma$.
\begin{lemma}\label{lem:Pspins}
    For all $\a\ge1$,
    \begin{equation}\label{eq:pa1}
        \Pspin{\a,1}=u_\a,
    \end{equation}
    and for all $\a,\b\ge1$,
    \begin{equation*}
        \Pspin{\a\b}=u_{\a+\b-1}+\widetilde{P}^\mathrm{spin}_{\a\b}\bigl(u_{\le\a+\b-2}^{(*)};\e\bigr),
    \end{equation*}
    where $\widetilde{P}^\mathrm{spin}_{\a\b}\in\widehat{\R}_{u;\ge1}^{\mathrm{ev};[0]}$.
\end{lemma}
\begin{proof}
    For $\beta=1$ and $(g,n)\neq(0,1)$, all cycles in the integral in \hyperref[eq:pspin]{(\ref*{eq:pspin})} are pull-backs via the morphism $\pi:\Mb{g,n+2}\to\Mb{g,n+1}$ forgetting the second marked point, and thus the integral vanishes. If $(g,n)=(0,1)$ the integral is over $\Mb{0,3}$ and all cycles equal $1$. Thus \hyperref[eq:pa1]{(\ref*{eq:pa1})} follows. More generally, on $\Mb{0,3}$ all the cycles in the integral \hyperref[eq:pspin]{(\ref*{eq:pspin})} equal $1$, so the coefficient of $u_{\a+\b-1}$ in $\Pspin{\a\b}$ is $1$.
\end{proof}
\begin{proposition}\label{prop:Pspins}
    The following is true:
    \begin{align*}
        \Pspin{1,2} & = u_2 + \frac{\e^2}{8}u_1^{(2)}, \\
        \Pspin{1,3} & = u_3 + \frac{3\e^2}{8}u_2^{(2)} + \frac{\e^2}{12}u_1^{(2)}u_1 + \frac{\e^2}{12}u_1^{(1)}u_1^{(1)} + \frac{37\e^4}{1152}u_1^{(4)}, \\
        \Pspin{2,2} & = u_3 + u_1u_2 + \frac13u_1^3 + \frac{\e^2}{12}u_2^{(2)} + \frac{\e^2}{24}u_1u_1^{(2)} + \frac{\e^2}{12}u_1^{(1)}u_1^{(1)} + \frac{7\e^4}{5760}\,u_1^{(4)}.
    \end{align*}
\end{proposition}
\noindent We give the proof of the proposition in the next section. For now we just point out that, from the expression for $\Pspin{\a\b}$ in \hyperref[eq:pspin]{(\ref*{eq:pspin})}, the proof will involve calculating intersection numbers on moduli spaces of pointed curves of genera $g=0,1$ and $2$. \newline

For every $\b\ge1$, the differential polynomial $\Pspin{1,\b}-u_\b$ only depends on variables $u_\gamma^{(*)}$ with $\gamma\le\b-1$. Therefore the change of variables $u_\b\mapsto v_\b\bigl(u_*^{(*)},\e\bigr):=\Pspin{1,\b}$ is invertible. Moreover, the fact that $\Pspin{1,\b}-u_\b\in\operatorname{im}(\d_x^2)$ implies that the system \hyperref[eq:redr]{(\ref*{eq:redr})} has the following form in the new variables $v_*$:
\begin{equation}\label{eq:newDR}
    \dd{v_\a}{t^\b}=\d_x\Qspin{\a\b},\qquad\a,\b\ge1,
\end{equation}
where
\begin{gather*}
    \Qspin{\a\b}\in\widehat{\R}_{v;\ge1}^{\mathrm{ev};[0]}, \\
    \deg\Qspin{\a\b}=2\a+2\b-2, \\
    \Qspin{\a,1}=\Qspin{1,\a}=v_\a, \\
    \Qspin{\a\b}=\Qspin{\b\a}.
\end{gather*}
\hyperref[lem:Pspins]{Lemma \ref*{lem:Pspins}} and \hyperref[prop:Pspins]{Proposition \ref*{prop:Pspins}} imply that
\begin{gather*}
    \Qspin{\a,\b} = v_{\a+\b-1} + \widetilde{Q}^\mathrm{spin}_{\a\b}(v^{(*)}_{\le\a+\b-2};\e), \quad \widetilde{Q}^\mathrm{spin}_{\a\b}\in\widehat{\R}^{\mathrm{ev};[0]}_{v;\ge1}, \\
    \Qspin{2,2} = v_3 + v_1v_2 - \frac{\e^2}{6}v_2^{(2)} + \frac13v_1^3 - \frac{\e^2}{6}v_1v_1^{(2)} + \frac{\e^4}{180}v_1^{(4)}.
\end{gather*}
In \hyperref[thm:bkpr]{Theorem \ref*{thm:bkpr}} we will show that the properties above, together with the commutativity of the flows $\left\{\frac{\d}{\d t^\b}\right\}_{\b\ge1}$ of the hierarchy, completely determine all the differential polynomials $\Qspin{\a\b}$. 

\begin{center}\section{Computation of initial values}\label{sec:comp}\end{center}
We prove \hyperref[prop:Pspins]{Proposition \ref*{prop:Pspins}} in this section. The proof involves computing intersection numbers of the form
\begin{equation*}
    \int_{\Mb{g,n+2}}\hspace{-1.9em}\DR_g\bigl(-\sum_{i=1}^na_i,0,a_1,\dots,a_n\bigr)\lambda_g\HresSpin{g}{2\a-2,2\b-2,-2\a_1,\dots,-2\a_n}
\end{equation*}
for $(\a,\b)\in\{(1,2),(1,3),(2,2)\}$. Now since $\deg\Pspin{1,2}=4$ and $\deg\Pspin{1,3}=\deg\Pspin{2,2}=6$ with respect to the grading $\deg u_\a^{(k)}=2\a+k$, the only monomials in $u_*^{(*)}$ that can appear in these differential polynomials have coefficients determined by intersection numbers involving $g\le2$ and $n\le3$. We therefore split up the proof in the Lemmata \hyperref[lem:g0int]{\ref*{lem:g0int}}-\hyperref[lem:g1drint]{\ref*{lem:g1drint}} and \hyperref[lem:g2int]{\ref*{lem:g2int}}, ordered in terms of genus $g$ and method used.\\

\subsection{Intersection numbers in genera \texorpdfstring{$g=0,1$}{}}

For $g=0$, any spin structure on $\P^1$ is isomorphic to $\mathcal{O}(-1)$, and is therefore even. Therefore the spaces $\mathcal{H}^\mathrm{res}_0(2\a_1,\dots,2\a_n)$ have no odd components, so $\HresSpin{0}{2\a_1,\dots,2\a_n}=\Hres{0}{2\a_1,\dots,2\a_n}$. Thus the coefficients of $u_1u_2$ and $u_1^3$ in $\Pspin{2,2}$ can be readily determined in the next lemma.
\begin{lemma}\label{lem:g0int}
    The following is true:
    \begin{gather}
        \label{eq:m04}\int_{\Mb{0,4}}\HresSpin{0}{2,2,-2,-4}=1, \\
        \label{eq:m05}\int_{\Mb{0,5}}\HresSpin{0}{2,2,-2,-2,-2}=2.
    \end{gather}
\end{lemma}
\begin{proof}
    Consider the isomorphism $\C\setminus\{0,1\}\rightarrow\M{0,4}$ given by $t\mapsto(\P^1,1,t,\infty,0)$. There exists a unique (up to $\C^*$ rescaling) meromorphic differential $\omega$ on $\P^1$ whose associated divisor is $2\cdot[1]+2\cdot[t]+(-2)\cdot[\infty]+(-4)\cdot[0]$, namely $\omega=z^{-4}(z-1)^2(z-t)^2dz$. Its residue at $z=0$ is $-2(t+1)$, so $\omega$ is residueless if and only if $t=-1$. Thus $\mathcal{H}_0^\mathrm{res}(2,2,-2,-4)$ consists of a single point, yielding equation \hyperref[eq:m04]{(\ref*{eq:m04})}. To prove the second equality, denote the diagonal in $\left(\C\setminus\{0,1\}\right)^2$ by $\Delta$. There is an isomorphism $\left(\C\setminus\{0,1\}\right)^2\setminus\Delta\rightarrow\M{0,5}$ sending $(s,t)$ to $(\P^1,s,t,0,1,\infty)$. There is a unique (up to $\C^*$ rescaling) meromorhpic differential $\eta$ on $\P^1$ with associated divisor $2\cdot[s]+2\cdot[t]+(-2)\cdot[0]+(-2)\cdot[1]+(-2)\cdot[\infty]$, given by $\eta=z^{-2}(z-1)^{-2}(z-s)^2(z-t)^2dz$. The residues of $\eta$ at $z=0$ and $z=1$ are $2st(st-s-t)$ and $-2(s-1)(t-1)(st-1)$, respectively. Thus $\eta$ is residueless if and only if $(s,t)=(e^{\pm i\pi/3},e^{\mp i\pi/3})$, so $\mathcal{H}^\mathrm{res}_0(2,2,-2,-2,-2)$ consists of two points, yielding equation \hyperref[eq:m05]{(\ref*{eq:m05})}.
\end{proof}

We now pass to $g=1$. Given a point $(E;x_1,\dots,x_n)\in\mathcal{H}_1(2\a_1,\dots,2\a_n)$, the spin structure $L$ on the elliptic curve $E$ induced by the meromorphic differential is isomorphic to $\mathcal{O}_E(\a_1x_1+\dots+\a_nx_n)$. If $L\simeq\mathcal{O}_E$ then it has one section and is odd, whereas if $L\not\simeq\mathcal{O}_E$ it has no sections and is even. Therefore the odd component of $\mathcal{H}_1(2\a_1,\dots,2\a_n)$ is the locus where $\mathcal{O}_E(\a_1x_1+\dots+\a_nx_n)$ is trivial, which corresponds to $\mathcal{H}_1(\a_1,\dots,\a_n)$. Thus
    \begin{align}
        \nonumber\Hspin{1}{2\a_1,\dots,2\a_n} & = \Hg{1}{2\a_1,\dots,2\a_n}-2\cdot\Hg{1}{2\a_1,\dots,2\a_n}^\mathrm{odd} \\
        \label{eq:h1} & = \Hg{1}{2\a_1,\dots,2\a_n}-2\cdot\Hg{1}{\a_1,\dots,\a_n}.
    \end{align}
If at most one of the $\a_i$ is negative (so that the residueless condition is automatic), then the above holds for $\HresSpin{1}{2\a_1,\dots,2\a_n}$ as well. We use this to deduce the coefficients of $u_1^{(2)}$ in $\Pspin{1,2}$ and of $u_2^{(2)}$ in $\Pspin{1,3}$ in this lemma.
\begin{lemma}\label{lem:g1int1}
    The following is true:
    \begin{gather*}
        \int_{\Mb{1,3}}\DR_1(-a,0,a)\lambda_1\HresSpin{1}{0,2,-2}=\frac{a^2}{8}, \\
        \int_{\Mb{1,3}}\DR_1(-a,0,a)\lambda_1\HresSpin{1}{0,4,-4}=\frac{3a^2}{8}.
    \end{gather*}
\end{lemma}
\begin{proof}
    If $\pi:\Mb{1,3}\rightarrow\Mb{1,2}$ is the morphism that forgets the first marked point, then $\pi_*\DR_1(-a,0,a)=a^2\left[\Mb{1,2}\right]$ by \cite[Example 3.7]{bssz15}, \cite[Lemma 5.4]{bdgr18}. This, together with the fact that $\HresSpin{g}{0,2\a_2,2\a_3}=\pi^*\HresSpin{g}{2\a_2,2\a_3}$ for any $\a_1,\a_2\in\mathbb{Z}$, implies that the two equalities to be proved are equivalent to
    \begin{gather*}
        \int_{\Mb{1,2}}\lambda_1\HresSpin{1}{2,-2}=\frac{1}{8}, \\
        \int_{\Mb{1,2}}\lambda_1\HresSpin{1}{4,-4}=\frac{3}{8}.
    \end{gather*}
    It was proved in \cite[Lemma 3.6]{brz24} that $\int_{\Mb{1,2}}\lambda_1\Hres{1}{\a,-\a}=\frac{\a^2-1}{24}$ for $\a\ge2$, so
    \begin{align*}
        \int_{\Mb{1,2}}\lambda_1\HresSpin{1}{2,-2} \overset{\hyperref[eq:h1]{(\ref*{eq:h1})}}{=} \int_{\Mb{1,2}}\lambda_1\Hres{1}{2,-2} - 2\int_{\Mb{1,2}}\lambda_1\Hres{1}{1,-1} = \frac18, \\
        \int_{\Mb{1,2}}\lambda_1\HresSpin{1}{4,-4} \overset{\hyperref[eq:h1]{(\ref*{eq:h1})}}{=} \int_{\Mb{1,2}}\lambda_1\Hres{1}{4,-4} - 2\int_{\Mb{1,2}}\lambda_1\Hres{1}{2,-2} = \frac38,
    \end{align*}
    proving the claim.
\end{proof}
For $I\subset[n]$, denote by $\delta_0^I\in H^2\bigl(\Mgnb\bigr)$ the class of the closure of the substack of stable curves with dual graph
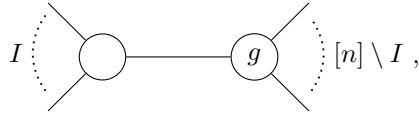
\begin{equation*}
    \begin{tikzpicture}[baseline={([yshift=-.5ex]current bounding box.center)}, node distance={20mm}, pin distance={7mm},every pin edge/.style={thin}]
        \node[draw, circle, pin=above left:{}, pin=below left:{}] (0) {\phantom{0}};
        \node[draw, circle, pin=above right:{}, pin=below right:{}] (g) [right of = 0] {$g$};
        \draw (0) -- (g);
        \draw[line width = 0.25mm, dotted, bend left=35] (-0.75,-0.45) to node[midway, left] {$I$} (-0.75,0.55);
        \draw[line width = 0.25mm, dotted, bend right=35] (2.75,-0.45) to node[midway, right] {$[n]\setminus I$} (2.75,0.5);
    \end{tikzpicture},
\end{equation*}
i.e. of curves having exactly one node separating a genus 0 zero component carrying the points marked by $I$ and a genus $g$ component carrying the points marked by $[n]\setminus I$. The next lemma computes the coefficients of $u_2^{(2)},u_1u_1^{(2)}$ and $u_1^{(1)}u_1^{(1)}$ in $\Pspin{2,2}$.
\begin{lemma}\label{lem:g1hain}
    The following is true:
    \begin{gather}
        \label{eq:g1hain1}\int_{\Mb{1,3}}\DR_1(-a,0,a)\lambda_1\HresSpin{1}{2,2,-4}=\frac{a^2}{12}, \\
        \label{eq:g1hain2}\int_{\Mb{1,4}}\DR_1(-a_1-a_2,0,a_1,a_2)\lambda_1\HresSpin{1}{2,2,-2,-2}=\frac{1}{24}a_1^2+\frac16a_1a_2+\frac{1}{24}a_2^2.
    \end{gather}
\end{lemma}
\begin{proof}
    The tautological class $\lambda_1\in H^2\left(\Mgnb\right)$ vanishes outside the locus of stable curves of compact type $\Mgn^\mathrm{ct}\subset\Mgnb$. Hain's formula \cite{hai13} gives an explicit expression for the double ramification cycle restricted to stable curves of compact type:
    \begin{align*}
        \at{\DR_1(-a,a)}{\M{1,2}^\mathrm{ct}} = & \left(\frac12\lambda_1+\dz{1,2}\right)a^2, \\
        \at{\DR_1(-a_1-a_2,a_1,a_2)}{\M{1,3}^\mathrm{ct}} = & \left(\lambda_1+\dz{1,2}+\dz{1,2,3}\right)a_1^2 + \Bigl(\lambda_1+\dz{1,2}+\dz{1,3}-\dz{2,3} \\ 
        & +\dz{1,2,3}\Bigr)a_1a_2 + \left(\lambda_1+\dz{1,3}+\dz{1,2,3}\right)a_2^2.
    \end{align*}
    The $\psi$-classes in the original statement of Hain's formula have been replace with $\lambda$-classes by using
    \begin{equation*}
        \psi_i=\lambda_1+\sum_{\substack{I\subset[n] \\ \lvert I\rvert\ge2,i\in I}}\delta_0^I\in H^2\left(\Mb{1,n}\right)
    \end{equation*} 
    for every $i=1,\dots,n$. Pulling back along the morphism forgetting the second marked point and the fact that $\lambda_1^2=0$ yields
    \begin{equation}
        \label{eq:drhain1}\DR_1(-a,0,a)\lambda_1 = \left(\dz{1,3}+\dz{1,2,3}\right)\lambda_1\,a^2
    \end{equation}
    and
    \begin{align}
        \nonumber \DR_1&(-a_1-a_2,0,a_1,a_2)\lambda_1 =\left(\dz{1,3}+\dz{1,2,3}+\dz{1,3,4}+\dz{1,2,3,4}\right)\lambda_1\,a_1^2 \\ 
        \nonumber & + \Bigl(\dz{1,3}+\dz{1,4} -\dz{3,4}+\dz{1,2,3}+\dz{1,2,4}+\dz{1,3,4}-\dz{2,3,4}+\dz{1,2,3,4}\Bigr)\lambda_1\,a_1a_2 \\ 
        \label{eq:drhain2} & + \left(\dz{1,4}+\dz{1,2,4}+\dz{1,3,4}+\dz{1,2,3,4}\right)\lambda_1\,a_2^2.
    \end{align}
    We show equation \hyperref[eq:g1hain1]{(\ref*{eq:g1hain1})}. Firstly,
    \begin{align*}
        \int_{\Mb{1,3}}\dz{1,3}\lambda_1\HresSpin{1}{2,2,-4} & = \left(\int_{\Mb{0,3}}\HresSpin{0}{2,-4,0}\right)\left(\int_{\Mb{1,2}}\lambda_1\HresSpin{1}{2,-2}\right) \\ 
        & = \frac18.
    \end{align*}
    The first equality follows from axiom (iii) of a partial CohFT. The integral on $\Mb{0,3}$ is 1 by axiom (ii) of a partial CohFT and the integral on $\Mb{1,1}$ was computed in the proof of \hyperref[lem:g1int1]{Lemma \ref*{lem:g1int1}}. Secondly,
    \begin{equation*}
        \int_{\Mb{1,3}}\hspace{-1em}\dz{1,2,3}\lambda_1\HresSpin{1}{2,2,-4} = \left(\int_{\Mb{0,4}}\HresSpin{0}{2,2,-4,-2}\right)\left(\int_{\Mb{1,1}}\lambda_1\HresSpin{1}{0}\right).
    \end{equation*}
    The integral on $\Mb{0,4}$ is $1$ from \hyperref[lem:g0int]{Lemma \ref*{lem:g0int}}. For the integral on $\Mb{1,1}$ we have $\HresSpin{1}{0} = -\Hres{1}{0}=-\Hg{1}{0}$. But $\mathcal{H}_1(0)=\M{1,1}$, since every smooth elliptic curve carries a nowhere vanishing holomorphic differential. Therefore
    \begin{equation*}
        \int_{\Mb{1,3}}\dz{1,2,3}\lambda_1\HresSpin{1}{2,2,-4}
        = -\int_{\Mb{1,1}}\lambda_1=-\frac{1}{24}.
    \end{equation*}
    Equation \hyperref[eq:g1hain1]{(\ref*{eq:g1hain1})} now follows from \hyperref[eq:drhain1]{(\ref*{eq:drhain1})}. We can prove equation \hyperref[eq:g1hain2]{(\ref*{eq:g1hain2})} in a similar fashion. By previous lemmas and axiom (iii) of the partial CohFT we have
    \begin{gather*}
        \int_{\Mb{1,4}}\dz{1,2,3}\lambda_1\HresSpin{1}{2,2,-2,-2} = \frac18, \\
        \int_{\Mb{1,4}}\dz{1,2,4}\lambda_1\HresSpin{1}{2,2,-2,-2} = \frac18, \\
        \int_{\Mb{1,4}}\dz{1,2,3,4}\lambda_1\HresSpin{1}{2,2,-2,-2} = -\frac{1}{12}.
    \end{gather*}
    All the remaining integrals arising from the expression \hyperref[eq:drhain2]{(\ref*{eq:drhain2})} for $\DR_1(-a_1-a_2,0,a_1,a_2)\lambda_1$ vanish. For example
    \begin{align*}
        & \int_{\Mb{1,4}}\dz{3,4}\lambda_1\HresSpin{1}{2,2,-2,-2} \\
        = \, & \left(\int_{\Mb{0,3}}\HresSpin{0}{-2,-2,2}\right)\left(\int_{\Mb{1,3}}\lambda_1\HresSpin{1}{2,2,-4}\right)
    \end{align*}
    vanishes because both integrals in the product vanish due to degree reasons, and
    \begin{align*}
        & \int_{\Mb{1,4}}\dz{1,3,4}\lambda_1\HresSpin{1}{2,2,-2,-2} \\ 
        = \, & \left(\int_{\Mb{0,4}}\HresSpin{0}{2,-2,-2,0}\right)\left(\int_{\Mb{1,2}}\lambda_1\HresSpin{1}{2,-2}\right)
    \end{align*}
    vanishes because the integrand on $\Mb{0,4}$ is the pull-back via the morphism forgetting the last marked point.
\end{proof}
Before moving on, we introduce some results from \cite{css21} on the intersection theoretic properties of strata of $k$-differentials. For our purposes we only consider $k=1$. In the reference, the authors give a conjectural formula for intersections of spin strata of meromorphic differentials with the top power of a $\psi$-class located at one of the poles, thus generalizing a theorem from \cite{bssz15}. Their notation is as follows. For a pair of nonnegative integers $(g,n)$ satisfying $2g-2+n>0$, consider a vector $a=(a_1,\dots,a_n)\in\Z^n$ with each $a_i$ odd, satisfying $\sum_{i=1}^na_i=2g-2+n$ and $a_1<0$. Set
\begin{equation*}
    \Aspin{g}{a}:=\int_{\Mgnb}\psi_1^{2g-3+n}\Hspin{g}{a_1-1,\dots,a_n-1}.
\end{equation*}
We remark that in \cite{css21} the space $\Hspin{g}{a_1-1,\dots,a_n-1}$ is denoted by $[\overline{\mathcal{M}}_g(a)]^\mathrm{spin}$. The conjectural formula given by \cite[Theorem 1.4]{css21} is
\begin{equation}\label{eq:csspsi}
    \Aspin{g}{a}=2^{-g}[z^{2g}]\exp\left(\frac{a_1z\cdot\mathcal{S}^\prime(z)}{\mathcal{S}(z)}\right)\frac{\cosh(z/2)}{\mathcal{S}(z)^{2g+n}}\prod_{i>1}\mathcal{S}(a_iz),
\end{equation}
where $\mathcal{S}(z)=\frac{\sinh(z/2)}{z/2}=\sum_{k\ge0}\frac{z^{2k}}{2^{2k}(2k+1)!}$ and $[z^{2g}](\cdot)$ stands for the coefficient of $z^{2g}$ in the power series. The above formula has been proved to hold for $g=0$ and $g=1$ due to the previously mentioned properties of the even and odd components of the moduli space in these genera. We give more details on the state of the conjecture when we compute intersection numbers for $g=2$ in the next subsection. \newline

Furthermore, the authors in \cite{css21} provide a formula for intersections of (spin) strata of \emph{residueless} meromorphic differentials with top powers of $\psi$-classes. For a vector of odd numbers $a=(a_0,a_1,\dots,a_n)$ such that $\sum_{i=0}^na_i=2g-1+n$, $a_0>0$ and $a_i<0$ for $1\le i\le n$, let
\begin{equation*}
    \mathcal{A}^{\mathfrak{R}(n-1),\mathrm{spin}}_g(a):=\int_{\Mb{g,n+1}}\psi_0^{2g-1}\HresSpin{g}{a_0-1,a_1-1,\dots,a_n-1}.
\end{equation*}
The notation $\mathfrak{R}(n-1)$ refers to the fact that we consider the stratum of differentials with vanishing residues at the poles indexed by $1,\dots,n-1$ (and thus vanishing residues at all $n$ poles). The formula for the intersection number above will involve a sum over $\operatorname{TR}(g,a)$, which is the set of genus $g$ twisted graphs (see \cite[Definition 2.1]{css21} for twisted graphs) with legs $a$ such that the vertex $v_0\in V(\Gamma)$ carrying the marking $0$ has genus $g$. Given $(\Gamma,I)\in\operatorname{TR}(g,a)$ and a vertex $v\in V(\Gamma)$, denote by $I(v)$ the vector of integer twists $I(h)$ indexed by the half-edges $h\in H(\Gamma)$ incident to $v$. A straightforward analysis reveals that each vertex $v\in V(\Gamma)$ has exactly one half-edge adjacent to it with a positive twist, whose value is denoted by $I(v)^+$. Define
\begin{equation*}
    F(\Gamma,I)=\prod_{v\in V(\Gamma)}-\bigl(-I(v)^+\bigr)^{n(v)-2}
\end{equation*}
where $n(v)$ is the number of half edges adjacent to $v$. Then \cite[Proposition 6.2]{css21}
\begin{equation}\label{eq:AspinResF}
    \mathcal{A}^{\mathfrak{R}(n-1),\mathrm{spin}}_g(a)=\sum_{(\Gamma,I)\in\operatorname{TR}(g,a)}-F(\Gamma,I)\cdot\Aspin{g}{I(v_0)}.
\end{equation}
Strictly speaking, the result above was only stated and proved in the non-spin scenario, but it can be generalized directly to the spin setting (see the discussion in the proof of \cite[Theorem 6.10]{css21}). 

With these results in hand, we compute the coefficients of $u_1u_1^{(2)}$ and $u_1^{(1)}u_1^{(1)}$ in $\Pspin{1,3}$.
\begin{lemma}\label{lem:g1drint}
    The following is true:
    \begin{equation*}
        \int_{\Mb{1,4}}\DR_1(-a_1-a_2,0,a_1,a_2)\lambda_1\HresSpin{1}{0,4,-2,-2} = \frac{(a_1+a_2)^2}{12}.
    \end{equation*}
\end{lemma}
\begin{proof}
    As before, if $\pi:\Mb{1,4}\rightarrow\Mb{1,3}$ is the morphism that forgets the first marked point, then $\pi_*\DR_1(-a_1-a_2,0,a_1,a_2)=(a_1+a_2)^2\left[\Mb{1,3}\right]$. Therefore the equation in the statement is equivalent to
    \begin{gather}
        \int_{\Mb{1,3}}\lambda_1\HresSpin{1}{4,-2,-2} = \frac{1}{12}.
    \end{gather}
    On $\Mb{1,3}$ we have $\lambda_1=\psi_1-\dz{1,2}-\dz{1,3}-\dz{1,2,3}$. By using axiom (iii) of a partial CohFT again, we see that both integrals in the product
    \begin{equation*}
        \int_{\Mb{1,3}}\dz{1,2}\HresSpin{1}{4,-2,-2}=\left(\int_{\Mb{0,3}}\HresSpin{0}{4,-2,-4}\right)\left(\int_{\Mb{1,2}}\HresSpin{1}{-2,2}\right)
    \end{equation*}
    vanish due to degree reasons. The same is true for $\dz{1,3}$ and $\dz{1,2,3}$. Thus it remains to show
    \begin{equation*}
        \int_{\Mb{1,3}}\psi_1\HresSpin{1}{4,-2,-2}=\frac{1}{12},
    \end{equation*}
    which in the previous notation corresponds to $\mathcal{A}^{\mathfrak{R}(1),\mathrm{spin}}_1(5,-1,-1)=\frac{1}{12}$. Equation \hyperref[eq:AspinResF]{(\ref*{eq:AspinResF})} involves a sum over $\operatorname{TR}(1,a)$ with $a=(5,-1,-1)$, which consists of the two twisted graphs
    \begin{equation*}
        (\Gamma_1,I_1)=\,
        \begin{tikzpicture}[baseline={([yshift=-.5ex]current bounding box.center)}, node distance={20mm}, pin distance={7mm},every pin edge/.style={thin}]
            \node[draw, circle, pin=left:$1$, label={[label distance=0.1mm,font=\tiny]175:$5$}, pin=45:$2$, label={[label distance=0.1mm,font=\tiny]25:\textendash$1$}, pin=below right:$3$, label={[label distance=0.3mm,font=\tiny]275:\textendash$1$}] (1) {$1$};
        \end{tikzpicture}
        ,\qquad (\Gamma_2,I_2)=\,
        \begin{tikzpicture}[baseline={([yshift=-.5ex]current bounding box.center)}, node distance={20mm}, pin distance={7mm},every pin edge/.style={thin}]
            \node[draw, circle, pin=left:$1$, label={[label distance=0.1mm,font=\tiny]175:$5$}, label={[label distance=0.1mm,font=\tiny]5:\textendash$3$}] (1) {1};
            \node[draw, circle, pin=45:$2$, label={[label distance=0.1mm,font=\tiny]25:\textendash$1$}, pin=315:$3$, label={[label distance=0.3mm,font=\tiny]275:\textendash$1$},label={[label distance=0.1mm,font=\tiny]175:$3$}] (0) [right of = 1] {\phantom{0}};
            \draw (1) -- (0);
        \end{tikzpicture}.
    \end{equation*}
    The twists at the half-edges are indicated by a smaller font. Here $F(\Gamma_1,I_1)=5$ and $F(\Gamma_2,I_2)=-3$, so
    \begin{equation*}
        \mathcal{A}_1^{\mathfrak{R}(1),\mathrm{spin}}(5,-1,-1)=-5\,\Aspin{1}{-1,-1,5}+3\,\Aspin{1}{-3,5}=\frac{1}{12},
    \end{equation*}
    where the two intersection numbers appearing on the right are computed using \hyperref[eq:csspsi]{(\ref*{eq:csspsi})}.
\end{proof}

\subsection{Intersection numbers in genus \texorpdfstring{$g=2$}{}}

What remains to prove \hyperref[prop:Pspins]{Proposition \ref*{prop:Pspins}} is to compute the coefficient of $u_1^{(4)}$ in $\Pspin{1,3}$ and $\Pspin{2,2}$. The result is given in the following lemma.
\begin{lemma}\label{lem:g2int}
    The following is true:
    \begin{gather*}
        \int_{\Mb{2,3}}\DR_2(-a,0,a)\lambda_2\HresSpin{2}{0,4,-2} = \frac{37a^4}{1152}, \\
        \int_{\Mb{2,3}}\DR_2(-a,0,a)\lambda_2\HresSpin{2}{2,2,-2} = \frac{7a^4}{5760}.
    \end{gather*}
\end{lemma}
We use two different methods to prove this. The first one is shorter but depends on the use of the formula \hyperref[eq:csspsi]{(\ref*{eq:csspsi})} for $g=2$, which for the moment is conjectural. The second method produces more general formulas using only known results and more direct computations, but is considerably longer.

\subsubsection{First method: spin DR conjecture}

The validity of equation \hyperref[eq:csspsi]{(\ref*{eq:csspsi})} depends on certain properties of the $1$-twisted spin double ramification cycles $\DR^\mathrm{spin}_g(a)$ (for $a=(a_1,\dots,a_n)$ with $\sum_i^na_i=2g-2+n$), a family of cycles that the authors of \cite{css21} introduce in their text as a spin analogue of Pixton's formula \cite{jppz17,pix23,bhpss23}. These conjectural properties are listed as Assumption 1.3 in \cite{css21}:
\begin{itemize}
    \item The cycle $\DR^\mathrm{spin}_g(a_1,\dots,a_n)$ is polynomial in the $a_i$;
    \item It holds that $\pi^*\DR^\mathrm{spin}_g(a_1,\dots,a_n)=\DR^\mathrm{spin}_g(a_1,\dots,a_n,1)$, where $\pi:\Mb{g,n+1}\rightarrow\Mgnb$ is the morphism forgetting the last marked point.
    \item If $P$ is a monomial in classes $\psi_i$ for which $a_i$ is negative, then
    \begin{equation*}
        \int_{\Mgnb}\Hspin{g}{a_1-1,\dots,a_n-1}\cdot P=\int_{\Mgnb}\DR^\mathrm{spin}_g(a_1,\dots,a_n)\cdot P.
    \end{equation*}
\end{itemize} 
The first property is true in the non-spin case, see \cite{pix23poly} or \cite{spe24}. The second one follows from a direct computation involving the definition of $\DR^\mathrm{spin}_g(a)$, and the third one is a consequence of \cite[Conjecture 2.5]{css21}. This last conjecture is a generalization of \cite[Conjecture A]{fp18} (which was proved in \cite{hs21,bhpss23}) to the spin context. As mentioned previously, these properties can be shown to hold for genera $g=0$ and $g=1$. For $g\ge2$ they remain to be proven. \newline

In the first proof of \hyperref[lem:g2int]{Lemma \ref*{lem:g2int}} we will make use of the following formula from \cite{bhs22}:
\begin{equation}\label{eq:drl2}
    a^{-4}\DR_2(-a,a)\lambda_2=\psi_2^4 - 
    \left[\begin{tikzpicture}[baseline={([yshift=-1.25ex]current bounding box.center)}, node distance={20mm}, pin distance={7mm},every pin edge/.style={thin}]
        \node[draw, circle, pin=left:$1$, label={[label distance=0.1mm,font=\scriptsize]5:$\psi$}] (a) {$1$};
        \node[draw, circle, label={[label distance=0.1mm,font=\scriptsize]5:$\psi^2$}, pin=right:$3$] (b) [right of = 1] {$1$};
        \draw (a)--(b);
    \end{tikzpicture}\right].\qquad(a\neq0)
\end{equation}
This is the one point case of a more general formula \cite{bgr19} related to the strong DR/DZ equivalence conjecture \cite{bur15,bdgr18}. A proof of the strong DR/DZ equivalence conjecture was completed by \cite{bls24}. By taking the pull-back of \hyperref[eq:drl2]{(\ref*{eq:drl2})} along the morphism $\pi:\Mb{2,3}\rightarrow\Mb{2,2}$ forgetting the second marked point one obtains
\begin{align}
    \nonumber a^{-4}&\DR_2(-a,0,a)\lambda_2 = \psi_3^4 -
    \left[\begin{tikzpicture}[baseline={([yshift=-0.5ex]current bounding box.center)}, node distance={20mm}, pin distance={7mm},every pin edge/.style={thin}]
        \node[draw, circle, pin=left:$1$, label={[label distance=0.1mm,font=\scriptsize]5:$\psi^3$}] (a) {$2$};
        \node[draw, circle, pin=above right:$2$, pin=below right:$3$] (b) [right of = a] {\phantom{0}};
        \draw (a)--(b);
    \end{tikzpicture}\right]
    \\
    \nonumber & -
    \left[\begin{tikzpicture}[baseline={([yshift=-0.5ex]current bounding box.center)}, node distance={20mm}, pin distance={7mm},every pin edge/.style={thin}]
        \node[draw, circle, pin=above left:$1$, pin=below left:$2$, label={[label distance=0.1mm,font=\scriptsize]5:$\psi$}] (a) {$1$};
        \node[draw, circle, pin=right:$3$, label={[label distance=0.1mm,font=\scriptsize]5:$\psi^2$}] (b) [right of = a] {$1$};
        \draw (a)--(b);
    \end{tikzpicture}\right] - 
    \left[\begin{tikzpicture}[baseline={([yshift=-0.5ex]current bounding box.center)}, node distance={20mm}, pin distance={7mm},every pin edge/.style={thin}]
        \node[draw, circle, pin=left:$1$, label={[label distance=0.1mm,font=\scriptsize]5:$\psi$}] (a) {$1$};
        \node[draw, circle, pin=above right:$2$, pin=below right:$3$, label={[label distance=0.1mm,font=\scriptsize]355:$\psi^2$}] (b) [right of = a] {$1$};
        \draw (a)--(b);
    \end{tikzpicture}\right]
    \\
    \nonumber & + 
    \left[\begin{tikzpicture}[baseline={([yshift=-0.5ex]current bounding box.center)}, node distance={20mm}, pin distance={7mm},every pin edge/.style={thin}]
        \node[draw, circle, pin=left:$1$, label={[label distance=0.1mm,font=\scriptsize]5:$\psi$}] (a) {$1$};
        \node[draw, circle, label={[label distance=0.1mm,font=\scriptsize]5:$\psi$}] (b) [right of = a] {$1$};
        \node[draw, circle, pin=above right:$2$, pin=below right:$3$] (c) [right of = b] {\phantom{0}};
        \draw (a)--(b)--(c);
    \end{tikzpicture}\right]
    \\
    \label{eq:DRlbig} & +
    \left[\begin{tikzpicture}[baseline={([yshift=-4ex]current bounding box.center)}, node distance={20mm}, pin distance={7mm},every pin edge/.style={thin}]
        \node[draw, circle, pin=left:$1$] (a) {$1$};
        \node[draw, circle, pin=above:$2$] (b) [right of = a] {\phantom{0}};
        \node[draw, circle, pin=right:$3$,label={[label distance=0.1mm,font=\scriptsize]5:$\psi^2$}] (c) [right of = b] {$1$};
        \draw (a)--(b)--(c);
    \end{tikzpicture}\right].
\end{align}
Denote the five boundary classes appearing above by $B_1,B_2,B_3,B_4,B_5\in H^4\bigl(\Mb{2,3}\bigr)$ respectively.
\begin{proof}[First proof of \text{\hyperref[lem:g2int]{Lemma \ref*{lem:g2int}}}]
    The intersection product $\int_{\Mb{2,3}}B_i\HresSpin{2}{0,4,-2}$ vanishes for every $i=1,\dots,5$. For example the first and third integrals in the product
    \begin{align*}
        \int_{\Mb{2,3}}B_4\HresSpin{2}{0,4,-2}&=\left(\int_{\Mb{1,2}}\psi_2\HresSpin{1}{0,0}\right)\left(\int_{\Mb{1,2}}\psi_2\HresSpin{1}{-2,2}\right) \\
        &\times\left(\int_{\Mb{0,3}}\HresSpin{0}{4,-2,-4}\right)
    \end{align*}
    vanish. Thus from \hyperref[eq:DRlbig]{(\ref*{eq:DRlbig})} we obtain
    \begin{align*}
        \int_{\Mb{2,3}}\DR_2(-a,0,a)\lambda_2\HresSpin{2}{0,4,-2} & = \int_{\Mb{2,3}}\psi_3^4\HresSpin{2}{0,4,-2}\cdot a^4 \\
        & = \Aspin{2}{-1,1,5}\cdot a^4 \overset{\hyperref[eq:csspsi]{(\ref*{eq:csspsi})}}{=}\frac{37a^4}{1152}.
    \end{align*}
    For the other integral in the statement, a similar analysis leads one to conclude \newline $\int_{\Mb{2,3}}B_i\HresSpin{2}{2,2,-2}=0$ for $i\neq3$. For $i=3$,
    \begin{align*}
        \int_{\Mb{2,3}}B_3\HresSpin{2}{2,2,-2} & = \left(\int_{\Mb{1,2}}\psi_2\HresSpin{1}{2,-2}\right)\left(\int_{\Mb{1,3}}\psi_2^2\HresSpin{1}{2,-2,0}\right) \\
        & = \Aspin{1}{-1,3}\cdot\Aspin{1}{-1,1,3} \overset{\hyperref[eq:csspsi]{(\ref*{eq:csspsi})}}{=} \frac{1}{64}.
    \end{align*}
    Therefore
    \begin{align*}
        \int_{\Mb{2,3}}\DR_2(-a,0,a)\lambda_2\HresSpin{2}{2,2,-2} & = \left(\int_{\Mb{2,3}}\psi_3^4\HresSpin{2}{2,2,-2} - \frac{1}{64}\right)a^4 \\
        & = \left(\Aspin{2}{-1,3,3} - \frac{1}{64}\right)a^4 \overset{\hyperref[eq:csspsi]{(\ref*{eq:csspsi})}}{=} \frac{7a^4}{5760},
    \end{align*}
    as claimed.
\end{proof}

\subsubsection{Second method: Picard relation}\label{secondmethod}

Here we prove quasi-polynomial formulas from which the integrals of \hyperref[lem:g2int]{Lemma \ref*{lem:g2int}} follow, without resorting to equation \hyperref[eq:csspsi]{(\ref*{eq:csspsi})}. The proof demonstrates how explicit computations in genus $2$ can be done by leveraging a relation in the Picard group. We start by explicating some results that will help to reduce our genus $2$ integrals to known integrals. \newline

For an indexing set $I\subset[n]$, let $\Gamma_h^I$  denote the stable graph of genus $g$ with $n$ legs having a single edge separating a genus $h$ vertex carrying the indices $I$ from a genus $g-h$ vertex with the remaining legs. Let $\delta_{1|1}$ be the class of the closure of the substack of genus 2 stable curves of compact type whose dual graph has two genus 1 vertices, and let $\delta_{rt}$ contains the compact-type genus 2 stable curves with rational tails. \newline

Recall that $\Pic(\Mgnb)$ is generated by the determinant of the Hodge bundle, the $\psi$-classes and the boundary classes. While these generators are linearly dependent for $g=0,1,2$, they are free for all higher genera by \cite{AC}. As in \cite{Mum83}, for example genus 2 modular form considerations provide a relation between these generators for $g=2$. Together with Mumford's GRR computations relating $\kappa_1$ and $\lambda_1$, this gives 
 the following relation, restricting formula (8.5) of \textit{loc.cit.} to compact-type:
\begin{proposition}\label{prop:picrelation}
The following relation holds:
\[\kappa_1 = \frac{7}{5}\delta_{1|1}\in \Pic(\mathcal{M}^{ct}_{2})\otimes \Q.\]
By pulling back this relation along forgetful maps, one obtains \[ \kappa_1 - \sum_{i=1}^n \psi_i = \frac{7}{5}\delta_{1|1} - \delta_{rt} \in \Pic(\mathcal{M}^{ct}_{2,n})\otimes \Q.\]
\end{proposition}

In \cite[Theorem 6]{sau19} it is shown how to reduce classes of strata of differentials with residue conditions to classes of strata of differentials without residue conditions on the IVC and in \cite[Proposition 8.3]{cmz22} for the multi-scale differentials. In \cite[Lemma 5.7]{css21} the argument for the spin-parity version is written down without residue-conditions, using how spin-parity redistributes along degenerations. We require the following special case of the spin-parity version. \newline 

Denote by $SB(\alpha)$ the \textit{simple bi-colored graphs}, i.e. twisted stable level graphs (as in Definition 7.2 of \cite{sau19}) with one vertex at level $0$, one vertex at level $-1$, and twists given by $\alpha+1$. Denote by $SB(\alpha)_i$ the simple bi-colored graphs with leg $i$ at the lower level, and by $SB^{\res}(\alpha)$ the simple bi-colored graphs such that at least one pole sits at the lower vertex. 
\begin{proposition}\label{resles}
    For $g\ge0$, let $\a=(\alpha_1,\dots,\a_n)$ be a partition of $2g-2$ consisting of even numbers, with exactly two entries being negative. Then 
    \begin{align*}
        & \HresSpin{g}{\alpha} = -(\alpha_i+1)\psi_i\Hspin{g}{\alpha} \\
        & + \left(\sum_{(\Gamma,I) \in SB(\alpha)_i } \hspace{-0.35em}-\hspace{-0.35em}\sum_{(\Gamma,I) \in SB(\alpha)^{\res} }\right) \frac{m(\Gamma,I)}{|\Aut(\Gamma,I)|}(\xi_\Gamma)_*\bigg(\Hspin{g_0}{I_0-1}\otimes \Hspin{g_{-1}}{I_{-1}-1}\bigg).
    \end{align*}
Here, $g_l$ and $I_l$ are respectively the genus and tuple of twists at level $l$, and $m(\Gamma,I)$ is the multiplicity of the twisted graph.
\end{proposition}
\begin{proof}
Recall the projection $p:\overline{B}_g(\a_1,\dots,\a_n) \longrightarrow \Mb{g,n}$ from the moduli space of multi-scale differentials. In the case of only two poles, $\overline{B}_g^\mathrm{res}(\a_1,\dots,\a_n) \subset \overline{B}_g(\a_1,\dots,\a_n)$ is of codimension $1$. Then, \cite[Proposition 5.9]{won24} expresses the difference of the classes of spin-parity components $[\overline{B}_g^\mathrm{res}(\a_1,\dots,\a_n)]^{\mathrm{spin}}$ as a linear combination of the spin-parity version $\xi^\mathrm{spin}$ of the first Chern class of the tautological quotient line bundle over $\overline{B}_g^\mathrm{res}(\a_1,\dots,\a_n)$ and a weighted sum over spin-parity versions of boundary divisors of $\overline{B}_g^\mathrm{res}(\a_1,\dots,\a_n)$ described by simple bi-colored graphs with at most $1$ pole at the top vertex. Similarly, Proposition 5.10 of \textit{op.cit.} describes the weighted sum over spin-parity versions of boundary divisors described by $SB_i(\a)$ as a linear combination of $\xi^\mathrm{spin}$ and the spin-parity version of the $\psi$-class $(p^*\psi_i)^\mathrm{spin}$. Taking the difference of these two equations in the Chow ring of $\overline{B}_g(\a_1,\dots,\a_n)$ gives the desired formula up there. The push-forward along $p$ of said boundary divisors is computed in Corollary 1.3 of \textit{op.cit.} The projection formula finishes the proof.
\end{proof}

We use the following splitting formula, obtained as above by taking differences in \cite[Proposition 5.10]{won24} for distinct markings and pushing-forward using the residueless variant of Corollary 1.3 of \textit{op.cit.}
\begin{proposition}\label{splittingformula}
    Given $g,n\geq 0$, let $s,t$ be distinct indices in $\{1,\dots,n\}$ and $\alpha \in \Z^n$ a partition of $2g-2$ consisting of even numbers. Then
    \begin{align*}
        &\big((\alpha_s+1)\psi_s - (\alpha_t+1)\psi_t\big)\cdot \HresSpin{g}{\alpha} \\
        &\quad = \sum_{(\Gamma,I)\in SB(\alpha)}f_{s,t}(\Gamma)\frac{m(\Gamma,I)}{|\Aut(\Gamma,I)|} (\xi_\Gamma)_*\bigg(\HresSpin{g_0}{I_0-1}\otimes \HresSpin{g_{-1}}{I_{-1}-1}\bigg) .
    \end{align*}
    The function $f_{s,t}(\Gamma)$ equals $0$ if $s$ and $t$ sit on the same vertex, equals $1$ if $s$ sits on level $-1$ and $t$ sits on level $0$, and equals $-1$ otherwise.
\end{proposition}

Then, the following relation allows us in particular to express the integrals over strata of differentials in terms of computable integrals. Indeed, weighted fundamental classes of moduli spaces of twisted differentials of \cite{fp18}
, are related to $1$-twisted double ramification cycles by \cite{hs21}. The twisted double ramification cycles are then expressed in terms of a Pixton formula by \cite{bhpss23}, in turn proving the conjecture in \cite{fp18}
 on the equality of the weighted fundamental classes and Pixton's cycles in the strictly meromorphic case. We shall not need the full capacity of Pixton formulas for our computations. \newline

Let $a \in \Z^n$ a partition of $2g -2+n$. Consider (isomorphism classes of) \textit{simple star graphs} $(\Gamma,I)\in \text{Star}_g(a)$, as defined for example in \cite[Definition 2.3]{css21}, as stable graphs of genus $g$ with twists determined by $a$, with the edges between the \textit{central vertex} $v_0$ and the remaining \textit{outer vertices} $v\in V_{\text{Out}}$, such that twists on outer vertices are positive.
\begin{proposition}[Star-graph formula, \cite{hs21}]\label{stargraph}
For $a \in \Z^n$ a partition of $2g - 2+n$ such that $a \not\in \Z^n_{>0}$, we have
\[\DR_g^1(a) = \hspace{-0.5em} \sum_{(\Gamma,I)\in \mathrm{Star}_g(a)} \frac{m(\Gamma,I)}{|\Aut(\Gamma,I)|} \cdot (\xi_{\Gamma})_*\left( [\overline{\H}_{g(v_0)}(I(v_0)-1)\otimes\bigotimes_{v \in V_{\text{Out}}}[\overline{\H}_{g(v)}(I(v)-1)]\right).\]
\end{proposition}
The spin-parity version of these results is the content of Conjecture 2.5 of \cite{css21}, and the route below circumvents the use of that formula.

\subsubsection*{Preliminary genus $1$ integrals}
Here we collect and compute some preliminary integrals in genus $g=1$ that appear in the hands-on method for the second proof of \hyperref[lem:g2int]{Lemma \ref*{lem:g2int}}. To begin, a straightforward generalization of the argument in the proof of \hyperref[lem:g1int1]{Lemma \ref*{lem:g1int1}} shows that for $\a\ge1$ even,
\begin{equation}\label{eq:g1prelim}
    \int_{\Mb{1,2}} \lambda_1 \HresSpin{1}{\alpha, -\alpha} = \frac{\alpha^2+2}{48}.
\end{equation}

\begin{lemma}\label{lem:g1integral}
Let $(\e_1,\e_2,\e_3)$ be a partition of $0$ consisting of even integers with $\e_1,\e_2<0$ and $\e_3>0$. Then the following holds:
\begin{align*}
\int_{\Mb{1,3}} \lambda_1 \HresSpin{1}{\e_1,\e_2,\e_3} &= \frac{1}{48} \, \e_1^{3} - \frac{1}{48} \, \e_1 \e_2^{2} - \frac{1}{48} \, \e_1 \e_3^{2} - \frac{1}{48} \, \e_1^{2} - \frac{1}{48} \, \e_2^{2} - \frac{1}{48} \, \e_3^{2} - \frac{1}{12}\,.
\end{align*}
In particular for $\a<0$ an even integer,
\[ \int_{\Mb{1,3}} \lambda_1 \HresSpin{1}{\a,-2,2-\a} = \frac{1}{24}\a^2-\frac{1}{12}\a-\frac{1}{4}\,.\]
\end{lemma}
\begin{proof}
    By \hyperref[resles]{Proposition \ref*{resles}},
        \begin{align*} 
            \int&_{\Mb{1,3}} \lambda_1 \HresSpin{1}{\e_1,\e_2,\e_3} =  -(\e_1+1) \int_{\Mb{1,3}}\lambda_1\psi_1\Hspin{1}{\e_1,\e_2,\e_3} \\
            & \quad- (1-\e_1)\left(\int_{\Mb{1,2}}\lambda_1\Hspin{1}{\e_1,-\e_1}\right)\left(\int_{\Mb{0,3}}\Hspin{0}{\e_2,\e_3,\e_1}\right).
        \end{align*}
    Indeed, the only graphs to consider in the second term above have a genus 0 and genus 1 component at different levels, such that either marking $1$ is at the lower level, while all poles are upstairs (empty), or minus the graphs with $i$ upstairs but there is a pole on the lower level. The only option for this is the twisted graph with underlying stable graph $\Gamma^{\{1\}}_1$. The second term is computed via equation \hyperref[eq:g1prelim]{(\ref*{eq:g1prelim})} and the first term can be computed as follows. We have $\lambda_1 = \psi_1 -\delta_0^{\{1,2\}} - \delta_0^{\{1,3\}} -\delta_0^{\{1,2,3\}}$, and the intersection of the $\psi$-class with the two first boundary classes vanishes, and the integral with $\delta_0^{\{1,2,3\}}$ vanishes for dimension reasons. Thus
    \begin{align*}
        \int_{\Mb{1,3}}\lambda_1\psi_1\Hspin{1}{\e_1,\e_2,\e_3} & = \int_{\Mb{1,3}}\psi_1^2\Hspin{1}{\e_1,\e_2,\e_3} \\
        & = \int_{\Mb{1,3}}\psi_1^2[\overline{\H}_1(\e_1,\e_2,\e_3)] - 2\int_{\Mb{1,3}}\psi_1^2[\overline{\H}_1(\e_1/2,\e_2/2,\e_3/2)] \\
        & = \frac{1}{24}(1 + \frac{1}{2}\e_2^2 + \frac{1}{2}\e_3^2).
    \end{align*}
    The last equality is obtained by applying \cite[Theorem 1.1]{css21} and the equality (4) in the same reference.
\end{proof}


\begin{lemma}\label{lem:g1n3psiintegral}
Let $\e = (\e_1,\e_2,\e_3)$ be a partition of $0$ consisting of even numbers, where only one is negative and the other two are nonnegative. Then the following holds:
\begin{align*}
    \int_{\Mb{1,3}} \lambda_1 \psi_1 \HresSpin{1}{\e_1,\e_2,\e_3} = & \; \frac{1}{216}\e_1^2 - \frac{1}{216}\e_1\e_2 + \frac{5}{432}\e_2^2 - \frac{1}{216}\e_1\e_3 - \frac{1}{54}\e_2\e_3 \\
    & + \frac{5}{432}\e_3^2 +\frac{1}{24} + \frac{\delta_{\e_1,0}}{24}.
\end{align*}
\end{lemma}
\begin{proof}
    Since only one entry of $\e$ is negative, the residueless condition is automatic. Thus
    \begin{align*}
        \Hspin{1}{\e_1,\e_2,\e_3}=[\overline{\H}_1(\e_1,\e_2,\e_3)] =[\overline{\H}_1(\e_1,\e_2,\e_3)] - 2[\overline{\H}_1(\e_1/2,\e_2/2,\e_3/2)].
    \end{align*}
    Therefore it suffices to compute $\int_{\Mb{1,3}} \lambda_1 \psi_1[\overline{\H}_1(\g_1,\g_2,\g_3)]$ with $\g=(\g_1,\g_2,\g_3)$ a partition of $0$, possibly containing odd numbers, with only one negative entry. We apply the star graph formula of \hyperref[stargraph]{Proposition \ref*{stargraph}}, using the fact that in genus $g=1$ the twisted DR cycle $\DR_1^1(a+1)$ in the log-convention coincides with the untwisted DR cycle $\DR_1(a)$ in the non-log-convention:
    \begin{align*}
        \int_{\Mb{1,3}} \lambda_1 & \psi_1[\overline{\H}_1(\g_1,\g_2,\g_3)] \\
        & = \int_{\Mb{1,3}} \lambda_1 \psi_1 \DR_1(\g_1,\g_2,\g_3) -\left(\int_{\Mb{1,1}}\lambda_1[\overline{\H}_1(0)]\right)\left(\int_{\Mb{0,4}}\psi_1[\overline{\H}_0(\g_1,\g_2,\g_3,-2)]\right) \\ 
        & \quad - (1-\g_1)\left(\int_{\Mb{0,3}}[\overline{\H}_0(\g_2,\g_3,\g_1-2)]\right)\left(\int_{\Mb{1,2}}\lambda_1\psi_1[\overline{\H}_1(\g_1,-\g_1)]\right).
    \end{align*}
    Indeed, the only simple star graphs that contribute have as underlying stable graph the trivial one or $\Gamma_0^{\{1,2,3\}}$ or $\Gamma^{\{1\}}_1$. Evaluating further one obtains
    \begin{align*}
        & = \int_{\Mb{1,3}} \lambda_1 \psi_1 \DR_1(\g_1,\g_2,\g_3) - \frac{1}{24} - \delta_{\g_1,0}\int_{\Mb{1,2}}\lambda_1\psi_1 \\
        & = \frac{1}{108}\g_1^2 - \frac{1}{108}\g_1\g_2 + \frac{5}{216}\g_2^2 - \frac{1}{108}\g_1\g_3 - \frac{1}{27}\g_2\g_3 + \frac{5}{216}\g_3^2 - \frac{1}{24} - \frac{\delta_{\g_1,0}}{24}.
    \end{align*}
    The integrals over the DR cycles can be computed explicitly by the polynomiality property and the \textit{admcycles} Sage package \cite{admcycles}.
\end{proof}

\begin{lemma}\label{lem:g1n4integral}
Let $\e = (\e_1,\e_2,\e_3,\e_4)$ be a partition of $0$ consisting of even numbers, such that $\e_1,\e_2 > 0$ and $\e_3,\e_4<0$. Then
\begin{align*} 
    & \int_{\Mb{1,4}} \lambda_1\psi_1  \HresSpin{1}{\e} \\
    = &\; \frac{1}{48}\left(\e_{1}^{3} + 3\e_{1}^{2} \e_{3} + 4\e_{1} \e_{3}^{2} + 2\e_{3}^{3} - \e_{1}^{2} - 2\e_{1}\e_{3} - 2\e_{3}^{2} + 2\e_{1} + 2\e_{3} - 2\right)\delta_{\e_1 + \e_3 < 1} \\
    & + \frac{1}{48}\left(\e_{1}^{3} + 3\e_{1}^{2}\e_{4} + 4\e_{1}\e_{4}^{2} + 2\e_{4}^{3} - \e_{1}^{2} - 2\e_{1} \e_{4} - 2\e_{4}^{2} + 2\e_{1} + 2\e_{4} - 2\right)\delta_{\e_1+\e_4 < 1} \\
    & - \frac{1}{256} \, \e_{1}^{3} + \frac{1}{384} \, \e_{1}^{2} \e_{2} - \frac{11}{768} \, \e_{1} \e_{2}^{2} - \frac{7}{384} \, \e_{1}^{2} \e_{3} - \frac{11}{384} \, \e_{1} \e_{2} \e_{3} - \frac{1}{24} \, \e_{2}^{2} \e_{3} - \frac{11}{768} \, \e_{1} \e_{3}^{2} - \frac{7}{384} \e_{1}^{2} \e_{4} \\
    & - \frac{11}{384} \, \e_{1} \e_{2} \e_{4} - \frac{1}{24} \, \e_{2}^{2} \e_{4} + \frac{5}{384} \, \e_{1} \e_{3} \e_{4} - \frac{11}{768} \, \e_{1} \e_{4}^{2} - \frac{19}{768} \, \e_{1}^{2} - \frac{5}{128} \, \e_{1} \e_{2} - \frac{43}{768} \, \e_{2}^{2} + \frac{1}{384} \e_{1} \e_{3} \\
    & + \frac{5}{384} \, \e_{2} \e_{3}- \frac{11}{768} \, \e_{3}^{2}  + \frac{1}{384} \, \e_{1} \e_{4} + \frac{5}{384} \, \e_{2} \e_{4} + \frac{5}{384} \, \e_{3} \e_{4} - \frac{11}{768} \, \e_{4}^{2} - \frac{1}{24} \, \e_{1} - \frac{1}{24} \, \e_{3} - \frac{1}{24} \, \e_{4} \\
    & - \frac{1}{12}.
\end{align*}

\end{lemma}
\begin{proof}
    By \hyperref[resles]{Proposition \ref*{resles}} we have
    \begin{align*}
        & \int_{\Mb{1,4}} \lambda_1\psi_1  \HresSpin{1}{\e} = -(\e_1+1)\int_{\Mb{1,4}}
        \lambda_1\psi_1^2\Hspin{1}{\e} \\ 
        + &\, (\e_1-1)\left(\cancel{\int_{\Mb{1,2}}\lambda_1\psi_1\HresSpin{1}{\e_1,-\e_1}}\right) \hspace{-0.2em} \left(\int_{\Mb{0,4}}\Hspin{0}{\e_2,\e_3,\e_4,\e_1-2}\right) \\
        + &\, (\e_1+\e_2-1)\left(\int_{\Mb{0,3}}\Hspin{0}{\e_3,\e_4,\e_1+\e_2-2}\right) \hspace{-0.2em} \left(\int_{\Mb{1,3}}\lambda_1\psi_1\Hspin{1}{\e_1,\e_2,-\e_2-\e_1}\right) \\
        - &\, (1-\e_1-\e_3)\left(\int_{\Mb{1,3}}\lambda_1\psi_1\Hspin{1}{\e_1,\e_3,-\e_3-\e_1}\right) \hspace{-0.2em} \left(\int_{\Mb{0,3}}\Hspin{0}{\e_2,\e_4,\e_1+\e_3-2}\right) \\
        - &\, (1-\e_1-\e_4)\left(\int_{\Mb{1,3}}\lambda_1\psi_1\Hspin{1}{\e_1,\e_4,-\e_4-\e_1}\right) \hspace{-0.2em} \left(\int_{\Mb{0,3}}\Hspin{0}{\e_2,\e_3,\e_1+\e_4-2}\right).
    \end{align*} 
    Indeed, the graphs contributing to this sum (but the first term) have either $i$ on the lower vertex with all poles on the upper vertex, or (minus) $i$ on the upper vertex with a single pole on the lower vertex. Note that the penultimate line vanishes unless $\e_1+\e_3<1$, while the last line vanishes unless $\e_1+\e_4<1$. The first integral decomposes as
    \begin{align*}
        \int_{\Mb{1,4}}\lambda_1\psi_1^2\Hspin{1}{\e} = \int_{\Mb{1,4}}
        \lambda_1\psi_1^2[\overline{\H}_1(\e)] - 2\int_{\Mb{1,4}}
        \lambda_1\psi_1^2[\overline{\H}_1(\e/2)], 
    \end{align*}
    the terms of which can be related to an integral over a DR cycle by the star graph formula of \hyperref[stargraph]{Proposition \ref*{stargraph}}. Using $\DR_1^1(\e+1) = \DR_1(\e)$, we have 
    \[\int_{\Mb{1,4}}\lambda_1\psi_1^2\DR_1(\e) = \int_{\Mb{1,4}}\lambda_1\psi_1^2[\overline{\H}_1(\e)] + \frac{1}{24}\int_{\Mb{0,5}}\psi_1^2[\overline{\H}_0(\e,-2)].\] Indeed, the star graphs with genus $0$ vertex and $2$ or $3$ legs including leg number $1$ give zero due to the $\psi^2_1$ class, and dimension reasons let the contributions from genus $1$ vertices with $1$ or $2$ legs including number 1 vanish.
    Using \textit{admcycles} one computes 
    \begin{align*}
        \int_{\Mb{1,4}} \lambda_1\psi_1^2\DR_1(\e) &= \frac{1}{128} \e_{1}^{2} - \frac{1}{192} \e_{1} \e_{2} + \frac{11}{384} \e_{2}^{2} - \frac{1}{192} \e_{1} \e_{3} - \frac{5}{192} \e_{2} \e_{3} \\ & \quad+ \frac{11}{384}\e_{3}^{2} - \frac{1}{192} \e_{1} \e_{4} - \frac{5}{192} \e_{2} \e_{4} - \frac{5}{192} \e_{3} \e_{4} + \frac{11}{384} \e_{4}^{2}. 
    \end{align*}
    Therefore,
    \begin{align*}
        \int_{\Mb{1,4}}\lambda_1\psi_1^2 \Hspin{1}{\e} & = \frac{1}{256} \e_{1}^{2} - \frac{1}{384} \e_{1} \e_{2} + \frac{11}{768} \e_{2}^{2} - \frac{1}{384} \e_{1} \e_{3} - \frac{5}{384} \e_{2} \e_{3} \\ 
        & \quad + \frac{11}{768} \e_{3}^{2} - \frac{1}{384} \e_{1} \e_{4} - \frac{5}{384} \e_{2} \e_{4} - \frac{5}{384} \e_{3} \e_{4} + \frac{11}{768} \e_{4}^{2} +\frac{1}{24}.
    \end{align*}
    This, together with \hyperref[lem:g1n3psiintegral]{Lemma \ref*{lem:g1n3psiintegral}} for the remaining integrals, gives the desired result. 
\end{proof}

\textbf{The first integral of \hyperref[lem:g2int]{Lemma \ref*{lem:g2int}.}} The following result will be needed in the subsequent computation.
\begin{lemma}\label{lem:g2n2integral}
    Let $\b_1$ be an even integer and $\b_2=2-\b_1$. Set $\beta = (\b_1,\b_2)$. Then
    \[\int_{\Mb{2,2}} \lambda_2 \psi_1 \HresSpin{2}{\beta} = \begin{cases} 0, & \beta_1=0 \;\mathrm{or}\; \beta_2 = 0, \\
    \frac{1}{4608}\beta_1^4 - \frac{1}{960}\beta_1^3 + \frac{1}{384}\beta_1^2 - \frac{7}{1440}\beta_1+ \frac{17}{1920}, &\mathrm{else}.
    \end{cases}\]
\end{lemma}
\begin{proof}
    The case of $\beta_1=0$ or $\beta_2 = 0$ follows from dimension counting. Otherwise, integrals of the form $\int_{\Mb{2,2}} \lambda_2 \psi_1 \HresSpin{2}{\beta}$ are treated as unknowns and can be computed by solving the invertible system
    \begin{align}\label{matrix1}
        \int_{\Mb{2,2}} \lambda_2 \begin{pmatrix}
        \psi_1 + \psi_2 - \kappa_1\\
        4(\beta_1+1)\psi_1 - \kappa_1\\
        (\beta_1 +1) \psi_1 - (\beta_2+1)\psi_2
        \end{pmatrix}
        \HresSpin{2}{\beta} =: \begin{pmatrix} A\\ B\\ C\end{pmatrix},
    \end{align}
    in case $A,B,C$ are known. We now compute these 3 integrals $A,B,C$. \\

\noindent \textbf{A:} \hyperref[prop:picrelation]{Proposition \ref*{prop:picrelation}} rewrites this integral, with $\delta_{1|1} = \delta_1^{\{1,2\}} + \delta_1^{\{1\}} \text{ and } \delta_{rt} = \delta_0^{\{1,2\}}$, as
\begin{align*}
    \int_{\Mb{2,2}}&\lambda_2\left(-\frac{7}{5}\delta_{1|1} + \delta_{rt}\right) \HresSpin{2}{\beta} \\
    = & -\frac{7}{5}\left(\int_{\Mb{1,3}}\lambda_1\HresSpin{1}{\beta_1,\beta_2,-\beta_1-\beta_2}\right)\left(\int_{\Mb{1,1}}\lambda_1\HresSpin{1}{0}\right) \\
    & - \frac{7}{5}\left(\int_{\Mb{1,2}}\lambda_1\HresSpin{1}{\beta_1,-\beta_1}\right)\left(\int_{\Mb{1,2}}\lambda_1\HresSpin{1}{\beta_2,-\beta_2}\right) \\
    & + \left(\cancel{\int_{\Mb{2,1}}\lambda_2\HresSpin{2}{2}}\right)\left(\int_{\Mb{0,3}}\HresSpin{1}{\beta_1,\beta_2,-2-\beta_1-\beta_2}\right) \\
    = & -\frac{7 \beta_1^4}{11520} + \frac{7 \beta_1^3}{2880} - \frac{7 \beta_1^2}{2880} - \frac{7}{320},
\end{align*}
where we used \hyperref[lem:g1n4integral]{Lemma \ref*{lem:g1n4integral}} and equation \hyperref[eq:g1prelim]{(\ref*{eq:g1prelim})}. \newline

\noindent \textbf{B:}
Assume first that $\beta_1 <0$. Let $\pi:\Mb{2,3}\rightarrow\Mb{2,2}$ be the morphism forgetting the third marked point. Now $\pi^*\psi_i = \psi_i - \delta_0^{\{i,3\}}$, $\kappa_0 = 4$ and $\kappa_1 = \pi_*(\psi_3^2)$, so 
\begin{equation*}
    B = \int_{\Mb{2,2}} \lambda_2\bigl((4(\beta_1+1)\psi_1 - \kappa_1\bigr)
    \HresSpin{2}{\beta} = \int_{\Mb{2,3}} \lambda_2\psi_3 ((\beta_1+1)\psi_1 - \psi_3)\HresSpin{2}{\beta,0}.
\end{equation*}
The splitting formula of \hyperref[splittingformula]{Proposition \ref*{splittingformula}} applied to the right-hand side implies that three terms contribute. We draw the twisted graph associated to each term:

\begin{align*}
    \left[\vcenter{\hbox{\begin{tikzpicture}[
        baseline={([yshift=-.5ex]current bounding box.center)},
        node distance={20mm},
        pin distance={7mm},
        every pin edge/.style={thin}]
        \node[draw, circle,
        pin=left:$3$,
        label={[label distance=0.1mm,font=\tiny]175:$1$},
        label={[label distance=0.1mm,font=\tiny]265:$3$}] (1) {2};
        \node[draw, circle,
        pin=315:$2$,
        pin=225:$1$,
        label={[label distance=0.1mm,font=\tiny]340:$\beta_2\!+\!1$},
        label={[label distance=0.3mm,font=\tiny]205:$\beta_1\!+\!1$},
        label={[label distance=0.1mm,font=\tiny]95:\textendash$3$}] (0) [below of = 1] {0};
        \draw (1) -- (0);
    \end{tikzpicture}}}\right] 
    & : \quad 3\left(\int_{\Mb{2,2}}\lambda_2\psi_1\HresSpin{2}{0,2}\right)\left(\int_{\Mb{0,3}}\HresSpin{0}{\beta_1,\beta_2,-4}\right), \\ 
    \left[\vcenter{\hbox{\begin{tikzpicture}[
        baseline={([yshift=-.5ex]current bounding box.center)},
        node distance={20mm},
        pin distance={7mm},
        every pin edge/.style={thin}]
        \node[draw, circle,
        pin=left:$3$,
        label={[label distance=0.1mm,font=\tiny]175:$1$},
        label={[label distance=0.1mm,font=\tiny]265:$1$}] (1) {1};
        \node[draw, circle,
        pin=315:$2$,
        pin=225:$1$,
        label={[label distance=0.1mm,font=\tiny]340:$\beta_2\!+\!1$},
        label={[label distance=0.3mm,font=\tiny]205:$\beta_1\!+\!1$},
        label={[label distance=0.1mm,font=\tiny]95:\textendash$1$}] (0) [below of = 1] {1};
        \draw (1) -- (0);
    \end{tikzpicture}}}\right] 
    & : \quad \left(\int_{\Mb{1,2}}\lambda_1\psi_1\HresSpin{1}{0,0}\right)\left(\int_{\Mb{1,3}}\lambda_1\HresSpin{1}{\beta_1,\beta_2,-2}\right), \\
    \left[\vcenter{\hbox{\begin{tikzpicture}[
        baseline={([yshift=-.5ex]current bounding box.center)},
        node distance={20mm},
        pin distance={7mm},
        every pin edge/.style={thin}]
        \node[draw, circle,
        pin=left:$1$,
        label={[label distance=-0.9mm,font=\tiny]175:$\beta_1\!+\!1$},
        label={[label distance=-0.9mm,font=\tiny]265:\textendash$\beta_1\!+\!1$}] (1) {1};
        \node[draw, circle,
        pin=315:$2$,
        pin=225:$3$,
        label={[label distance=0.1mm,font=\tiny]340:$\beta_2\!+\!1$},
        label={[label distance=0.1mm,font=\tiny]195:$1$},
        label={[label distance=-0.9mm,font=\tiny]95:\textendash$\beta_2\!+\!1$}] (0) [below of = 1] {1};
        \draw (1) -- (0);
    \end{tikzpicture}}}\right] 
    & : \quad 
    \begin{matrix} -(1-\beta_1)\left(\displaystyle\int_{\Mb{1,2}}\lambda_1\HresSpin{1}{\beta_1,-\beta_1}\right) \\ \times\left(\displaystyle\int_{\Mb{1,3}}\lambda_1\psi_2\HresSpin{1}{\beta_2,0,\beta_1-2}\right). \end{matrix}
\end{align*}
The first term vanishes, and so
\begin{align*}
    B 
    & = -\frac{1}{24} \int_{\Mb{1,3}}\lambda_1\HresSpin{1}{\beta_1,\beta_2,-2} \\ 
    & \quad + 2(\beta_1-1)\left(\int_{\Mb{1,2}}\lambda_1\HresSpin{1}{\beta_1,-\beta_1}\right)\left(\int_{\Mb{1,2}}\lambda_1\HresSpin{1}{\beta_2,\beta_1-2}\right) \\ 
    & =\frac{\beta_1}{48} - \frac{\beta_1^2}{64} + \frac{\beta_1^3}{96} - \frac{5 \beta_1^4}{1152} + \frac{\beta_1^5}{1152}.
\end{align*}
The last equality is due to \hyperref[lem:g1integral]{Lemma \ref*{lem:g1integral}}. Similar reasoning shows that $B$ admits the same expression for $\beta_1 >0$, so that $B$ is polynomial on its whole domain\footnote{ \textit{A priori} the level graphs differ depending on the sign of $\beta_1$, but the integrals coincide.}. \newline

\noindent \textbf{C:} The splitting formula of \hyperref[splittingformula]{Proposition \ref*{splittingformula}} decomposes these integrals into the weighted sums
\begin{align*}
    C & = \int_{\Mb{2,2}}\lambda_2((\beta_1+1)\psi_1-(\beta_2+1)\psi_2)\Hspin{2}{\beta_1,\beta_2} \\ &=
    (\beta_1-1)\left(\int_{\Mb{1,2}}\lambda_1\Hspin{1}{\beta_1,2-\beta_1}\right)\left(\int_{\Mb{1,2}}\lambda_1\Hspin{1}{\beta_2,2-\beta_2}\right)\\
    & = \frac{\beta_1^5}{2304} - \frac{5 \beta_1^4}{2304} + \frac{\beta_1^3}{192} - \frac{\beta_1^2}{144} + \frac{5 \beta_1}{576} - \frac{1}{192}.
\end{align*}
Indeed, the only relevant $(g,n) = (2,2)$ twisted level graph is $\Gamma^{\{1\}}_1$ with, in case $\beta_1<0$, twists $\beta_1+1$ and $\beta_2+1$ on the upper and lower vertex respectively, counted with a minus sign due to the level function. For $\beta_1>0$ the minus sign comes from the twist instead of the level function, so that we get the same result. \newline

\noindent Now that we know $A,B,C$, the corresponding matrix problem \hyperref[matrix1]{(\ref*{matrix1})} is 
\begin{align*}
    & \begin{pmatrix}
        1 && 1 && -1 \\
        4(\beta_1+1) && 0 && -1 \\
        (\beta_1+1) && -(3-\beta_1) && 0
    \end{pmatrix}
    \begin{pmatrix}
        \int_{\Mb{2,2}}\lambda_2\psi_1\HresSpin{2}{\beta}\\
        \int_{\Mb{2,2}}\lambda_2\psi_2\HresSpin{2}{\beta}\\
        \int_{\Mb{2,2}}\lambda_2\kappa_1\HresSpin{2}{\beta}
    \end{pmatrix} 
    = \begin{pmatrix}A\\ B\\ C \end{pmatrix}, \\ 
    & \text{ with } \begin{pmatrix}A\\ B\\ C\end{pmatrix} =
    \begin{pmatrix}
        -7 \beta_1^4/11520 + 7 \beta_1^3/2880 - 7 \beta_1^2/2880 - 7/320 \\
        \beta_1/48 - \beta_1^2/64 + \beta_1^3/96 - 5 \beta_1^4/1152 + \beta_1^5/1152\\
        \beta_1^5/2304 - 5 \beta_1^4/2304 + \beta_1^3/192 - \beta_1^2/144 + 5 \beta_1/576 - 1/192
    \end{pmatrix}.
\end{align*}
Therefore
\begin{align*}
    \int_{\Mb{2,2}}\lambda_2\psi_1\HresSpin{2}{\beta_1} &= \left(B - A - \frac{1}{3-\beta_1}C\right)/\,\left(4\beta_1+3-\frac{\beta_1+1}{3-\beta_1}\right) \\
    & = \frac{1}{4608}\beta_1^4 - \frac{1}{960}\beta_1^3 + \frac{1}{384}\beta_1^2 - \frac{7}{1440}\beta_1 + \frac{17}{1920},
\end{align*}
as claimed.
\end{proof}

The next proposition implies the first equation of \hyperref[lem:g2int]{Lemma \ref*{lem:g2int}}.
\begin{proposition}\label{lem:firstintegral}
    For $\b\ge3$, the following holds:
    \[\int_{\Mb{2,3}}\DR_2(-a,0,a)\lambda_2\HresSpin{2}{0,2\beta-2,-2\beta+4} = \frac{12\beta^4-64\beta^3+148\beta^2-152\beta+65}{5760}a^4.\]
\end{proposition}
\begin{proof}
For brevity, let $\e_1 := 2\beta-2 \geq 4$, and $ \e_2 := -2\beta+4 \leq -2$. Firstly,
\[\int_{\Mb{2,3}}\DR_2(-a,0,a)\lambda_2\HresSpin{2}{0,\e_1,\e_2} = \int_{\Mb{2,2}}\pi_*\DR_2(0,a,-a)\lambda_2\HresSpin{2}{\e_1,\e_2}.\]
An application of Hain's formula in \textit{admcycles} yields
\[\lambda_2\pi_*\DR_2(0,a,-a) = \frac{\lambda_2}{8}\left(7\psi_2+ \kappa_1 - \psi_1 - 7 \delta_0^{\{1,2\}} -3 \delta_1^{\{1\}} - 3 \delta_1^{\{1,2\}}\right)a^4.\]
By the Picard relation of \hyperref[prop:picrelation]{Proposition \ref*{prop:picrelation}}, we can rewrite this as 
\begin{align*}
    \lambda_2\pi_*\DR_2(0,a,-a) & = \frac{\lambda_2}{8}\left(8\psi_2 +\frac{7}{5}\delta_{1|1} - \delta_{rt} - 7 \delta_0^{\{1,2\}} -3 \delta_1^{\{1\}} - 3 \delta_1^{\{1,2\}}\right)a^4 \\
    & = \lambda_2\left(\psi_2 -  \delta_0^{\{1,2\}} -\frac{1}{5} \delta_1^{\{1\}} - \frac{1}{5} \delta_1^{\{1,2\}}\right)a^4.
\end{align*}
Since $\int_{\Mb{2,1}}\lambda_2\HresSpin{2}{2} = 0$, the integrals against $\delta_0^{\{1,2\}}$ vanish. Hence, we only have to consider the genus 1 integrals computed in \hyperref[lem:g1integral]{Lemma \ref*{lem:g1integral}} and the genus 2 integral computed in \hyperref[lem:g2n2integral]{Lemma \ref*{lem:g2n2integral}}. Putting it all together brings us the following expression for the coefficient of $a^4$ in $\int_{\Mb{2,3}}\DR_2(-a,0,a)\lambda_2\HresSpin{2}{0,\e_1,\e_2}$:
\begin{align*}
    & \int_{\Mb{2,2}}\hspace{-0.65em}\psi_2\lambda_2\HresSpin{2}{\e_1,\e_2} -\frac{1}{5} \int_{\Mb{2,2}}\hspace{-0.65em}\delta_1^{\{1\}}\lambda_2\HresSpin{2}{\e_1,\e_2} - \frac{1}{5} \int_{\Mb{2,2}}\hspace{-0.65em}\delta_1^{\{1,2\}}\lambda_2\HresSpin{2}{\e_1,\e_2} \\
    & = \frac{1}{7680}\e_2^4 - \frac{1}{1440}\e_2^3 + \frac{13}{5760}\e_2^2 - \frac{7}{1440}\e_2 + \frac{11}{1920}.
\end{align*}
In terms of $\beta$, this reads as desired.
\end{proof}

The initial value of $\beta = 3$ in the formula above gives \[\int_{\Mb{2,3}}\DR_2(-a,0,a)\lambda_2\HresSpin{2}{0,4,-2} = \frac{37a^4}{1152},\] which proves the first part of \hyperref[lem:g2int]{Lemma \ref*{lem:g2int}}. The class $\Hspin{2}{0,4,-2}$ can be computed using the algorithm of \cite{won24} and verifies this result. \newline

\textbf{The second integral of \hyperref[lem:g2int]{Lemma \ref*{lem:g2int}.}} The idea here is similar to \hyperref[lem:g2n2integral]{Lemma \ref*{lem:g2n2integral}}, albeit that the complexity is increased due to the absence of trivial orders.
\begin{proposition}\label{g2n3integral} 
    The following holds:
    \[\int_{\Mb{2,3}}\DR_2(-a,0,a)\lambda_2\HresSpin{2}{2\gamma -2,2,-2\gamma + 2} = \frac{4+3\delta_{\gamma,2}}{5760}a^4, \quad \gamma \geq 2.\]
\end{proposition}
\noindent The computation scheme is as follows. By Hain's formula,
\begin{align}
    \nonumber & \at{\DR_2(-a,0,a)}{\Mb{2,3}^\mathrm{ct}}= \frac{a^4}{8} \biggl[\psi_1^2 + 2\psi_1\psi_3+\psi_3^2 -2\Bigl(\psi_1\delta_1^{\{1\}} + \psi_1\delta^{\{3\}}_1 + \psi_3\delta_1^{\{1\}} + \psi_3\delta_1^{\{3\}} \\ 
    \nonumber & + \psi_1\delta_2^{\{1\}} + \psi_3\delta_2^{\{3\}}\Bigr) + \bigl(\delta_1^{\{1\}}\bigr)^2 + 2\delta_1^{\{1\}}\delta_1^{\{3\}} + 2\delta^{\{1\}}_1\delta_2^{\{1\}} + 2\delta_1^{\{1\}}\delta^{\{3\}}_2 + \bigl(\delta^{\{3\}}_1\bigr)^2 +2\delta_1^{\{3\}}\delta_2^{\{1\}} \\
    \label{eq:hainhain} & + 2\delta_1^{\{3\}}\delta_2^{\{3\}} + \bigl(\delta_2^{\{1\}}\bigr)^2 + 2\delta_2^{\{1\}}\delta_2^{\{3\}}  + \bigl(\delta_2^{\{3\}}\bigl)^2 \Big].
\end{align}
Let $\a=(\a_1,\a_2,\a_3):=(2\g-2,2,-2\g+2)$ with $\g\ge2$. By the expression above, we must compute $3$ types of integrals, which we call \emph{$H$-integrals}:
\begin{align*}
    H_{ij}(\alpha) & := \int_{\Mb{2,3}} \lambda_2\psi_i\psi_j \HresSpin{2}{\alpha}, \quad i,j\in \{1,3\}, \\
    H_{i;h}^I(\alpha) & := \int_{\Mb{2,3}} \lambda_2 \psi_i\delta_h^I \HresSpin{2}{\alpha}, \quad i \in \{1,3\}, \;\, h\in\{1,2\}, \;\, I = \begin{cases}\{1\}\;\text{or}\;\{3\}, & h=1,\\ \{i\},&h=2.\end{cases}, \\
    H_{h;l}^{I;J}(\alpha) & := \int_{\Mb{2,3}} \lambda_2\delta_h^I\delta_l^J\HresSpin{2}{\alpha}, \quad h,l\in\{1,3\}, \;\, I,J\in\bigl\{\{1\},\{3\}\bigr\}.
\end{align*}

\noindent $\bm{H_{ij}:}$ Consider the system 
\begin{align}
    \label{matrixproblemHij} & \begin{pmatrix} A\\B\\C\\D\\E\end{pmatrix} := \int_{\Mb{2,3}} \lambda_2 
    \begin{pmatrix} 
    \psi_1\psi_1+\psi_1\psi_2+\psi_1\psi_3 - \psi_1\kappa_1\\
    5(\alpha_1+1) \psi_1\psi_1 - \psi_1\kappa_1 \\
    (\alpha_1+1)\psi_1\psi_1 - (\alpha_3+1)\psi_1\psi_3  \\
    (\alpha_1+1)\psi_1\psi_1 - (\alpha_2+1)\psi_1\psi_2  \\
    (\alpha_1+1)\psi_3\psi_1 - (\alpha_3+1)\psi_3\psi_3  \\
    \end{pmatrix}
    \HresSpin{2}{\alpha} \\
    \nonumber = &\; \begin{pmatrix} 
        1 & 1 & 1 & - 1 & 0\\
        5(\alpha_1+1) & 0 & 0 &-1 & 0\\
        (\alpha_1+1) & 0 & -(\alpha_3+1) & 0 & 0\\
        (\alpha_1 + 1) & - (\alpha_2+1) & 0 & 0 & 0\\
        0 & 0 & (\alpha_1 + 1) & 0 & -(\alpha_3+1) 
    \end{pmatrix}
    \begin{pmatrix}
        \int_{\Mb{2,3}} \lambda_2 \psi_1\psi_1\HresSpin{2}{\alpha} \\
        \int_{\Mb{2,3}} \lambda_2 \psi_1    \psi_2    
        \HresSpin{2}{\alpha}\\
        \int_{\Mb{2,3}} \lambda_2 \psi_1    \psi_3    
        \HresSpin{2}{\alpha}\\
        \int_{\Mb{2,3}} \lambda_2 \psi_1    \kappa_1    
        \HresSpin{2}{\alpha}\\
        \int_{\Mb{2,3}} \lambda_2 \psi_3    \psi_3    
        \HresSpin{2}{\alpha}   
    \end{pmatrix}.
\end{align}

\noindent Let us compute $A,B,C,D$ and $E$. \newline

\noindent \textbf{A:} 
By \hyperref[prop:picrelation]{Proposition \ref*{prop:picrelation}}, 
\[A =  \int_{\Mb{2,3}} \lambda_2 \psi_1 \left(-\frac{7}{5}\delta_{1|1} + \delta_{rt}\right) \HresSpin{2}{\alpha}.\] 
Now since $\delta_{1|1} = \delta_1^{\{1,2\}}+\delta_1^{\{1,3\}}+\delta_1^{\{2,3\}} + \delta_1^{\{1,2,3\}} \text{ and } \delta_{rt} = \delta_2^{\{1\}}+\delta_2^{\{2\}} +\delta_2^{\{3\}} + \delta_0^{\{1,2,3\}}$ on $\Mb{2,3}$, we have
\begin{align*} 
   A = & -\frac{7}{5} \left(\int_{\Mb{1,3}} \lambda_1\psi_1 \HresSpin{1}{2\gamma-2,2,-2\gamma}\right)\left( \int_{\Mb{1,2}}\lambda_1 \HresSpin{1}{-2\gamma+2,2\gamma-2}\right) \\ 
   & - \frac{7}{5}\left(\int_{\Mb{1,3}} \lambda_1\psi_1 \HresSpin{1}{2\gamma-2,-2\gamma+2,0}\right)\left( \int_{\Mb{1,2}}\lambda_1 \HresSpin{1}{2,-2}\right) \\ 
   & - \frac{7}{5}\left(\cancel{\int_{\Mb{1,3}} \lambda_1 \HresSpin{1}{2,-2\gamma+2,2\gamma-4}}\right) \left(\int_{\Mb{1,2}}\lambda_1\psi_1 \HresSpin{1}{2\gamma-2,-2\gamma+2}\right) \\ 
   & - \frac{7}{5}\left(\int_{\Mb{1,4}} \lambda_1\psi_1 \HresSpin{1}{2\gamma-2, 2,-2\gamma+2,-2}\right)\left(\int_{\Mb{1,1}}\lambda_1\HresSpin{1}{0}\right) \\ 
   & + \left(\int_{\Mb{2,2}}\lambda_2\psi_1 \HresSpin{2}{2\gamma-2,-2\gamma+4}\right)\left({\int_{\Mb{0,3}}\HresSpin{0}{2,-2\gamma+2,2\gamma-6}}\right) \\ 
   & + \left(\int_{\Mb{2,2}} \lambda_2\HresSpin{2}{\alpha_2,2-\alpha_2}\right)\left(\cancel{\int_{\Mb{0,3}}\psi_1\HresSpin{0}{\alpha_1,\alpha_3,-2-\alpha_1-\alpha_3}}\right) \\ 
   & + \left(\int_{\Mb{2,2}} \lambda_2\HresSpin{2}{\alpha_3,2-\alpha_3}\right)\left(\cancel{\int_{\Mb{0,3}}\psi_1\HresSpin{0}{\alpha_1,\alpha_2,-2-\alpha_1-\alpha_2}}\right) \\ 
   & + \left(\cancel{\int_{\Mb{2,1}}\lambda_2 \HresSpin{2}{2}}\right)\left(\int_{\Mb{0,4}}\psi_1\HresSpin{0}{\alpha_1,\alpha_2,\alpha_3,-2-\a_1-\a_2-\a_3}\right). 
\end{align*}
Note that the crossed-out terms vanish for dimension reasons. By the Lemmata~\ref{lem:g1n3psiintegral}, \ref{lem:g1n4integral}, \ref{lem:g2n2integral}, this evaluates to
\begin{align*}
    A = \begin{cases} -\frac{7}{80}, & \gamma = 2, \\
    -\frac{1}{160} \, \gamma^{4} - \frac{1}{360} \, \gamma^{3} + \frac{13}{480} \, \gamma^{2} - \frac{11}{360} \, \gamma - \frac{1}{1920}, & \gamma \ge3.\end{cases}
\end{align*}

\noindent \textbf{B:} By a similar reasoning as for the B integral of \hyperref[lem:g2n2integral]{Lemma \ref*{lem:g2n2integral}}, 
\[\int_{\Mb{2,3}} \lambda_2 \psi_1 (5(\alpha_1+1) \psi_1 - \kappa_1)\HresSpin{2}{\alpha} = \int_{\Mb{2,4}} \lambda_2 \psi_1 \psi_4((\alpha_1+1) \psi_1 - \psi_4)\HresSpin{2}{\alpha,0}.\] 
The splitting formula of \hyperref[splittingformula]{Proposition \ref*{splittingformula}} equates this with
\begin{align*}
    B & = (\alpha_1-3)\left(\int_{\Mb{0,4}}\psi_3\HresSpin{0}{\alpha_2,\alpha_3,0,\alpha_1-4}\right)\left(\int_{\Mb{2,2}}\lambda_2\psi_1\HresSpin{2}{\alpha_1,2-\alpha_1}\right) \\ 
    & \quad + \quad\quad\quad\,\,\left(\int_{\Mb{1,2}}\lambda_1\psi_1\HresSpin{1}{0,0}\right)\left(\int_{\Mb{1,4}}\lambda_1\psi_1\HresSpin{1}{\alpha_1,\alpha_2,\alpha_3,-2}\right) \\
    &\quad+ (\alpha_1-1)\left(\int_{\Mb{1,4}}\lambda_1\psi_3\HresSpin{1}{\alpha_2,\alpha_3,0,\alpha_1-2}\right)\left(\cancel{\int_{\Mb{1,2}}\lambda_1\psi_1\HresSpin{1}{\alpha_1,-\alpha_1}}\right) \\
    &\quad+ (1-\alpha_3)\left(\int_{\Mb{1,3}}\lambda_1\psi_2\HresSpin{1}{\alpha_3,0,-\alpha_3}\right)\left(\int_{\Mb{1,3}}\lambda_1\psi_1\HresSpin{1}{\alpha_1,\alpha_2, \alpha_3-2}\right) \\
    &\quad- (\alpha_2-1)\left(\int_{\Mb{1,3}} \lambda_1\psi_1 \HresSpin{1}{\alpha_1, \alpha_3,\alpha_2-2} \right)\left(\int_{\Mb{1,3}}\lambda_1\psi_2\HresSpin{1}{\alpha_2,0,-\alpha_2}\right).
\end{align*}
This evaluates to
\begin{align*}
    B = 
    \begin{cases} \frac{29}{96}, & \g =2,  \\
    \frac{5}{144} \, \gamma^5 - \frac{21}{160} \, \gamma^{4} + \frac{241}{720} \, \g^3 - \frac{83}{144} \, \gamma^{2} + \frac{113}{192} \, \gamma - \frac{103}{384}, & \gamma \ge 3.
    \end{cases}
\end{align*}

\noindent \textbf{C:}
Again by the splitting formula, 
\begin{align*}
    &\int_{\Mb{2,3}}\lambda_2\psi_1((\alpha_1+1)\psi_1 - (\alpha_3+1)\psi_3)\HresSpin{2}{\alpha_1,\alpha_2,\alpha_3} \\
    & = (\alpha_1-3)\left(\int_{\Mb{2,2}}\lambda_2\psi_1\HresSpin{2}{\alpha_1,2-\alpha_1}\right)\left(\int_{\Mb{0,3}}\HresSpin{0}{\alpha_2,\alpha_3,\alpha_1-4}\right)\delta_{\alpha_1\geq4} \\
    & \quad + (1-\alpha_3)\left(\int_{\Mb{1,2}}\lambda_1\HresSpin{1}{\alpha_3,-\alpha_3}\right)\left(\int_{\Mb{1,3}}\lambda_1\psi_1\HresSpin{1}{\alpha_1,\alpha_2,\alpha_3-2}\right) \\
    & \quad+ (\alpha_1-1)\left(\int_{\Mb{1,3}}\lambda_1\HresSpin{1}{\alpha_2,\alpha_3,\alpha_1-2}\right)\left(\cancel{\int_{\Mb{1,2}}\lambda_1\psi_1\HresSpin{1}{\alpha_1,-\alpha_1}}\right).
\end{align*}
Thus, 
\begin{align*}
    C = \begin{cases} \frac{11}{64}, & \gamma = 2, \\ 
    \frac{1}{48} \, \gamma^{5} - \frac{139}{1440} \, \gamma^{4} + \frac{67}{240} \, \gamma^{3} - \frac{139}{288} \, \gamma^{2} + \frac{277}{576} \, \gamma - \frac{253}{1152}, & \gamma \ge3.
\end{cases}
\end{align*}

\noindent \textbf{D:}
This integral coincides with swapping the second and third marking in the splitting formula expression for $C$, and it can then be computed to yield:
\begin{align*}
    D = \begin{cases} -\frac{1}{64}, & \gamma = 2, \\ 
    \frac{1}{144} \, \gamma^{5} - \frac{89}{1440} \, \gamma^{4} + \frac{161}{720} \, \gamma^{3} - \frac{31}{72} \, \gamma^{2} + \frac{259}{576} \, \gamma - \frac{253}{1152}, & \gamma \ge 3.
    \end{cases}
\end{align*}

\noindent \textbf{E:}
This integral can be computed by a splitting formula as for $C$ up to a different genus $0$ term, relating to the vanishing when multiplied by a $\psi$-class, and the effect of replacing $\psi_1$ with $\psi_3$ in the left-hand side on the expansion terms, to yield:
\begin{align*}
E = \frac{1}{48} \, \gamma^{5} - \frac{161}{1440} \, \gamma^{4} + \frac{49}{144} \, \gamma^{3} - \frac{781}{1440} \, \gamma^{2} + \frac{153}{320} \, \gamma - \frac{107}{640}.\\
\end{align*}

\begin{lemma}
    With these expressions for $A,B,C,D,E$, the matrix problem \hyperref[matrixproblemHij]{(\ref*{matrixproblemHij})} is solved to yield:
    \begin{align*}
        H_{11}(\alpha) = & \begin{cases} \frac{17}{480}, & \gamma = 2, \\ \frac{1}{288} \, \gamma^{4} - \frac{1}{120} \, \gamma^{3} + \frac{29}{1440} \, \gamma^{2} - \frac{7}{240} \, \gamma + \frac{131}{5760} & \gamma \ge3. \end{cases} \\
        H_{13}(\alpha) = & \begin{cases} \frac{21}{320}, & \gamma = 2, \\ \frac{1}{144} \, \gamma^{4} - \frac{1}{36} \, \gamma^{3} + \frac{53}{720} \, \gamma^{2} - \frac{11}{120} \, \gamma + \frac{21}{320}, & \gamma \ge 3. \end{cases} \\
        H_{33}(\alpha) = & \begin{cases} \frac{131}{5760}, & \gamma = 2, \\ \frac{1}{288} \, \gamma^{4} - \frac{7}{360} \, \gamma^{3} + \frac{77}{1440} \, \gamma^{2} - \frac{1}{16} \, \gamma + \frac{13}{384}, & \gamma \ge 3. \end{cases}
    \end{align*}
\end{lemma}

\noindent $\bm{H_{i;h}^I:}$ These integrals are collected in the following lemma.
\begin{lemma}
The only non-vanishing integrals of type $H_{i;h}^I$ are the following: 
\begin{align*}
    H_{1;1}^{\{3\}}(\alpha)&= \frac{1}{144} \, \gamma^{4} - \frac{1}{72} \, \gamma^{3} + \frac{1}{48} \, \gamma^{2} - \frac{1}{48} \, \gamma + \frac{1}{64},\\
    H_{3;1}^{\{1\}}(\alpha) & = \frac{1}{144} \, \gamma^{4} - \frac{1}{24} \, \gamma^{3} + \frac{5}{48} \, \gamma^{2} - \frac{17}{144} \, \gamma + \frac{11}{192}, \\
    H_{1,2}^{\{1\}}(\alpha) & =\left(\frac{1}{288} \, \gamma^{4} - \frac{1}{45} \, \gamma^{3} + \frac{9}{160} \, \gamma^{2} - \frac{5}{72} \, \gamma + \frac{47}{1152}\right)\delta_{\gamma \ge3}, \\
    H_{3,2}^{\{3\}}(\alpha) & = \frac{1}{288} \, \gamma^{4} - \frac{1}{180} \, \gamma^{3} + \frac{1}{160} \, \gamma^{2} + \frac{3}{640}.
\end{align*}
\end{lemma}
\begin{proof}
    The case $h=1$ follows from intersections with boundary classes, and inserting the already computed values of the genus $1$ integrals. For example:
    \begin{align*}
        H_{1;1}^{\{3\}}(\alpha) & =  \left(\int_{\Mb{1,2}}\lambda_1 \HresSpin{1}{\alpha_3,-\alpha_3}\right)\left(\int_{\Mb{1,3}}\lambda_1\psi_1\HresSpin{1}{\alpha_1,\alpha_2, \alpha_3-2}\right).
    \end{align*}
    For $h=2$ one obtains for example the following product of integrals, where the genus $2$ integrals are computed in \hyperref[lem:g2n2integral]{Lemma \ref*{lem:g2n2integral}}:
    \begin{align*}
        H_{1,2}^{\{1\}}(\alpha) & = \left(\int_{\Mb{2,2}}\lambda_2\psi_1 \HresSpin{2}{\alpha_1,2-\alpha_1}\right)\left(\int_{\Mb{0,3}}\HresSpin{0}{\alpha_2,\alpha_3,-2-\alpha_2-\alpha_3}\right).
    \end{align*}
\end{proof}

\noindent $\bm{H_{h;l}^{I;J}:}$ Recall \cite[Appendix A]{GP03} that the stacky fibre product of two gluing maps $\xi_{\Gamma_1}, \xi_{\Gamma_2}$ is the disjoint union of moduli spaces $\Mb{\Gamma}$ over (isomorphism classes of) \textit{generic bi-specialisations}, which are spans $\Gamma_1 \leftarrow \Gamma \rightarrow \Gamma_2$ of specialisations of stable graphs such that edges in $\Gamma$ can be found in either $\Gamma_1$ or $\Gamma_2$ (under the given morphisms). In case these edges are found in both specialisations, we get $\psi$-classes as excess intersection contribution to the intersection of the cocycle classes corresponding to $\Gamma_{1,2}$, as shown in \cite[Proposition 2.19]{vS}:
\begin{proposition}\label{prop:excessint} 
    The following holds:
    \begin{align*}
        \delta_{\Gamma_1}\cdot\delta_{\Gamma_2} = \sum_{ \textrm{ generic }\Gamma_1 \leftarrow \Gamma \rightarrow \Gamma_2} (\xi_{\Gamma})_*\left(\prod_{\{h,h'\} \textrm{ in both } \Gamma_{1} \text{ and } \Gamma_2} (-\psi_h - \psi_{h'})\right).
    \end{align*}
\end{proposition}
Here $\psi_h$ comes from the $\psi$-class for the marking corresponding to $h$ in the component of $\Mb{\Gamma}$ corresponding to the vertex that has $h$ as half-edge. It is the pull-back under the projection to said component. This excess intersection formula for gluing maps will be used to compute the remaining $H$-integrals, by reducing them to sums over compact-type generic bi-specialisations to the graphs $\Gamma^I_h$ and $\Gamma_l^J$. \\

\noindent Consider integrals with $h=l=2$. Note that $\delta_2^{\{1\}}\delta_2^{\{3\}} = 0$, so $H_{2;2}^{\{3\};\{1\}} = H_{2;2}^{\{1\};\{3\}} = 0.$ Now since $\bigl(\delta_2^{\{1\}}\bigr)^2 = - (\xi_{\Gamma_2^1})_*(\psi_2\otimes 1)$, 
\begin{align*}
    H_{2;2}^{\{1\};\{1\}}(\alpha) &= -\int_{2,2}\lambda_2\psi_2 \HresSpin{2}{\alpha_1,\alpha_2+\alpha_3}\delta_{\alpha_2+\alpha_3 < 0}\\ &=\left(-\frac{1}{288} \, \gamma^{4} + \frac{7}{360} \, \gamma^{3} - \frac{7}{160} \, \gamma^{2} + \frac{31}{720} \, \gamma - \frac{23}{1152}\right)\delta_{\gamma\ge3}.
\end{align*}
By symmetry, and since both $\alpha_1,\alpha_2$ are positive for our range, 
\begin{align*}
    H_{2;2}^{\{3\};\{3\}}(\alpha) & = -\int_{\Mb{2,2}}\lambda_2\psi_2 \HresSpin{2}{\alpha_3,\alpha_2+\alpha_1} \\
    & = \frac{1}{288} \, \gamma^{4} + \frac{1}{120} \, \gamma^{3} - \frac{1}{96} \, \gamma^{2} + \frac{7}{720} \, \gamma - \frac{17}{1920}.
\end{align*}

\noindent Consider integrals with $h=l=1$.
The only generic compact-type bi-specialisation of $(\Gamma_1^{\{1\}},\Gamma_1^{\{1\}})$ is $\Gamma_1^{\{1\}}$ itself and so $\bigl(\delta_1^{\{1\}}\bigr)^2 = -(\xi_{\Gamma_1^{\{1\}}})_*(\psi_2\otimes 1) -(\xi_{\Gamma_1^{\{1\}}})_*(1\otimes\psi_3)$. Therefore
\begin{align*}
    H_{1;1}^{\{1\};\{1\}}(\alpha) & = -\left(\cancel{\int_{\Mb{1,2}}\lambda_1\psi_2\HresSpin{1}{\a_1,-\alpha_1}}\right)\left(\int_{\Mb{1,3}}\lambda_1\HresSpin{1}{\alpha_2,\alpha_3, -\alpha_2-\alpha_3}\right) \\ 
    & \quad- \left(\int_{\Mb{1,2}}\lambda_1\HresSpin{1}{\alpha_1,-\alpha_1}\right)\left(\int_{\Mb{1,3}}\lambda_1\psi_3\HresSpin{1}{\alpha_2,\alpha_3, -\alpha_2-\alpha_3}\right) \\ 
    & = -\frac{1}{144} \, \gamma^{4} + \frac{1}{36} \, \gamma^{3} - \frac{1}{18} \, \gamma^{2} + \frac{1}{18} \, \gamma - \frac{5}{192},
\end{align*}
in which the first integral vanishes by dimension reasons as $\alpha_1\neq 0$. By symmetry, 
\[H_{1;1}^{\{3\};\{3\}}(\alpha) = -\frac{1}{144} \, \gamma^{4} + \frac{1}{36} \, \gamma^{3} - \frac{1}{18} \, \gamma^{2} + \frac{1}{18} \, \gamma - \frac{5}{192}.\]

\noindent The generic compact-type bi-specialisation of $(\Gamma_1^{\{1\}},\Gamma_1^{\{3\}})$ is obtained by replacing in $\Gamma_1^{\{3\}}$ the genus $1$ vertex and markings $1,2$ with $\Gamma_0^{\{2\}} \in G_{1,2}$. In this case, the specialisations have no common edge and thus no $\psi$-classes appear. As a result, 
\begin{align*} 
    H_{1;1}^{\{1\};\{3\}}(\alpha) = & \; \left(\int_{\Mb{1,2}}\lambda_1\HresSpin{1}{\alpha_1,-\alpha_1}\right)\left(\int_{\Mb{0,3}}\HresSpin{0}{\alpha_2,\alpha_1-2, \alpha_3-2}\right) \\ & \times \left(\int_{\Mb{1,2}}\lambda_1\HresSpin{1}{\alpha_3,-\alpha_3}\right) = \frac{\alpha_1^2+2}{48}\cdot\frac{\alpha_3^2+2}{48}.
\end{align*}

\noindent Consider integrals with $h=2, l=1$.
The only generic compact-type bi-specialisation of $(\Gamma^{\{1\}}_2,\Gamma^{\{1\}}_1)$ is to replace the genus $1$ vertex with two legs in $\Gamma^{\{1\}}_1$ by $\Gamma_0^{\{1,2\}} \in G_{0,3}$. Thus, by similar reasoning as above: 
\begin{align*}
    H_{2;1}^{\{1\};\{1\}}(\alpha) = & \; \left(\int_{\Mb{1,2}}\lambda_1\HresSpin{1}{\alpha_1,-\alpha_1}\right)\left(\int_{\Mb{1,2}}\lambda_1\HresSpin{1}{\alpha_1-2,2-\alpha_1}\right) \\ 
    & \times\left(\int_{\Mb{0,3}}\HresSpin{0}{\alpha_2,\alpha_3,\alpha_1-4}\right)=\frac{\alpha_1^2+2}{48}\frac{(\alpha_1-2)^2+2}{48}\cdot\delta_{\alpha_1\geq 4}.
\end{align*} 
By symmetry, 
\[H_{2;1}^{\{3\};\{3\}}(\alpha)=\frac{\alpha_3^2+2}{48}\frac{(\alpha_3-2)^2+2}{48}.\]
Lastly, no generic compact-type bi-specialisations of $(\Gamma_1^I, \Gamma_2^J)$ exist when $I \neq J$, so 
\begin{equation*}
    H_{2;1}^{\{1\};\{3\}}(\alpha) = H_{2;1}^{\{3\};\{1\}}(\alpha) = 0.
\end{equation*}

\noindent\textit{Proof of \hyperref[g2n3integral]{Proposition \ref*{g2n3integral}}.} By using \hyperref[eq:hainhain]{(\ref*{eq:hainhain})}  and the $H$-integrals one obtains
\begin{align*}
    &\int_{\Mb{2,3}}\DR_2(-a,0,a)\lambda_2\HresSpin{2}{2\gamma-2,2, 2- 2\gamma} \\
    = &\; \frac{a^4}{8} \Big[H_{11}(\alpha) + 2H_{13}(\alpha)+H_{33}(\alpha) -2\bigl(H_{1;1}^{\{1\}}(\alpha) + H_{1;1}^{\{3\}}(\alpha) + H_{3;1}^{\{1\}}(\alpha) + H_{3;1}^{\{3\}}(\alpha) \\
    & \quad + H_{1;2}^{\{1\}}(\alpha) + H_{3;2}^{\{3\}}(\alpha)\bigr) + H_{1;1}^{\{1\};\{1\}}(\alpha) + 2H_{1;1}^{\{1\};\{3\}}(\alpha) + 2H_{1;2}^{\{1\};\{1\}}(\alpha) + 2H_{1;2}^{\{1\};\{3\}}(\alpha) \\ 
    & \quad + H_{1;1}^{\{3\};\{3\}}(\alpha) + 2H_{1;2}^{\{3\};\{1\}}(\alpha) + 2H_{1;2}^{\{3\};\{3\}}(\alpha) + H_{2;2}^{1;1}(\alpha) + 2H_{2;2}^{\{1\};\{3\}}(\alpha) + H_{2;2}^{\{3\};\{3\}}(\alpha)\Big] \\
    = &\; \begin{cases} \frac{7a^4}{5760}, & \gamma = 2, \\ \frac{a^4}{1440}, & \gamma \ge 3.
    \end{cases}
\end{align*}
\qed

\begin{center}\section{KP and BKP hierarchies}\label{sec:kp}\end{center}
\subsection{The KP hierarchy}
For a more comprehensive introduction to the KP hierarchy, see \cite{mjd00} or \cite{dic03} or \cite{hir04}. For all integers $i\ge1$ and $k\ge0$, let $f_i$ be a function of the time variables $T_1,T_2,T_3,\dots$, and let $f^{(k)}_i:=\d_x^kf_i$ be the $k^\text{th}$ derivative of $f_i$ with respect to the first variable $x:=T_1$. In what follows, we treat each $f_i$ as a formal variable. The ring of differential polynomials associated to these formal variables is $\mathcal{R}_f:=\C[f_*^{(*)}]$, where an asterisk indicates that a subscript or superscript could take any value. A pseudo-differential operator in the variables $f_*^{(*)}$ is a Laurent series of the form
\begin{equation*}
    A = \sum_{n\le m}a_n\d_x^n,\qquad m\in\Z,\quad a_n\in\R_f.
\end{equation*}
We denote the positive part of $A$ by $A_+:=\sum_{n=0}^ma_n\d_x^n$, and define $\operatorname{res} A$ to be $a_{-1}$. Let $\PO_f$ denote the space of pseudo-differentials in the variables $f_*^{(*)}$. The product $\circ$ of pseudo-differential operators is defined by linearly extending the rule
\begin{equation*}
    \bigl(a\d_x^k\bigr)\circ\bigl(b\d_x^l) := \sum_{p\ge0}\binom{k}{p}a\,b^{(p)}\d_x^{k+l-p},\qquad a,b\in\R_f,\quad k,l\in\Z,
\end{equation*}
where $b^{(p)}$ denotes $\d_x^pb$. This endows $\PO_f$ with the structure of an associative algebra. The commutator bracket of $A,B\in\PO_f$ is $[A,B]=A\circ B-B\circ A$. Let
\begin{equation*}
    L:=\d_x+\sum_{i\ge1}f_i\d_x^{-i}\in\PO_f.
\end{equation*}
The KP hierarchy is the system of compatible evolutionary PDEs in the variables $f_i$ depending on the time variables $x=T_1,T_2,T_3,\dots$ defined by
\begin{equation*}
    \dd{L}{T_k}=[(L^k)_+,L],\qquad k\ge1.
\end{equation*}
Thus the equations of the hierarchy are
\begin{equation}\label{eq:kph}
    \dd{f_i}{T_k}=\operatorname{Coef}_{\d_x^{-i}}[(L^k)_+,L]=:S_{i,k},\qquad i,k\ge1,
\end{equation}
where $\operatorname{Coef}_{\d_x^{-i}}[(L^k)_+,L]$ indicates the coefficient of $\d_x^{-i}$ of $[(L^k)_+,L]$. By compatibility of the equations we mean
\begin{equation*}
    \frac{\d}{\d T_\ell}\dd{f_i}{T_k} = \frac{\d}{\d T_k}\dd{f_i}{T_\ell},  \quad i,k,\ell\ge1.
\end{equation*}
For a proof that this is indeed the case, see \cite[Theorem 4.8]{bur22}. Notice that $L_+=\d_x$, so 
\begin{equation*}
    \dd{f_i}{T_1} = S_{i,1} = f_i^{(1)},\quad i\ge1,
\end{equation*}
which justifies the identification $x=T_1$. Define three gradings on $\R_f$: 
\begin{itemize}
    \item The differential grading $\deg_{\d_x}$ is given by $\deg_{\d x}f_i^{(k)}:=k$. The corresponding homogeneous component of $\R_f$ of degree $d$ is denoted by $\R_f^{[d]}$. Let $\R_f^\mathrm{ev}=\bigoplus_{d\ge0}\R_f^{[2d]}$.
    \item The grading $\deg$ is given by $\deg\,f^{(k)}_i:=i+1+k$.
    \item The grading $\widetilde{\deg}$ is given by $\widetilde{\deg}f_i^{(k)}:=1$. The corresponding homogeneous component of degree $d$ is denoted by $\R_{f;d}$, and let $\R_{f;\ge l}:=\bigoplus_{d\ge l}\R_{f;d}$.
\end{itemize}  
We extend $\deg$ from $\R_f$ to a grading on $\PO_f$ by setting $\deg\,\d_x:=1$. Then, for example, from equation \hyperref[eq:kph]{(\ref*{eq:kph})} it follows that $\deg S_{i,k}=i+k+1$. Let $f_{\le m}^{(*)}$ denote the collection of all variables $f_i^{(*)}$ with $i\le m$. One can readily show that
\begin{equation}\label{eq:sik}
    S_{i,k}=kf_{i+k-1}^{(1)}+\widetilde{S}_{i,k}(f^{(*)}_{\le i+k-2})
\end{equation}
for some differential polynomial $\widetilde{S}_{i,k}\in\R_{f;\ge1}$. Let us write some nontrivial differential polynomials $S_{i,k}$:
\begin{gather}
    \label{eq:s13} S_{1,3} = 3f_3^{(1)} + 3f_2^{(2)}+f_1^{(3)} + 3f_1f_1^{(1)}, \\
    \label{eq:s33} S_{3,3} = 3f_5^{(1)} + 3f_4^{(2)} + f_3^{(3)} + 3f_1f_3^{(1)} + 9f_1^{(1)}f_3 - 3f_1f_2^{(2)}   + 6f_2f_2^{(1)} - 3f_1^{(2)}f_2
\end{gather}

We now pass to the normal coordinates of the KP hierarchy. Notice that $\deg(\operatorname{res}L^k)=k+1$, and $\frac{\d}{\d f_k}\operatorname{res}L^k=k$, so $\operatorname{res}L^k-kf_k$ only depends on $f_{\le k-1}^{(*)}$. Therefore the change of variables $f_\a\mapsto w_\a(f^{(*)}_*):=\operatorname{res} L^\a$ is invertible. Clearly,
\begin{equation}\label{eq:waf}
    w_\a(f_*^{(*)}) = \a f_\a + \widetilde{w}_\a(f_{\le\a-1}^{(*)}).
\end{equation}
A direct computation gives the first few values of $w_\a(f_*^{(*)})$:
\begin{align}
    \label{eq:w1} w_1 = & \; f_1, \\
    w_2 = & \; 2f_2+f_1^{(1)}, \\
    \label{eq:w3} w_3 = & \; 3f_3 + 3f_2^{(1)} + f_1^{(2)} + 3f_1^2, \\
    w_4 = & \; 4f_4 + 6f_3^{(1)} + 4f_2^{(2)} + 12f_1f_2 + f_1^{(3)} + 6f_1f_1^{(1)}, \\
    \label{eq:w5} w_5 = & \; 5f_5 + 10f_4^{(1)} + 10f_3^{(2)} + 20f_1f_3 + 5f_2^{(3)} + 20f_1f_2^{(1)} + 10f_2^2 + 10f_1^{(1)}f_2 + f_1^{(4)} \\
    \nonumber & + 10f_1f_1^{(2)} + 5f_1^{(1)}f_1^{(1)} + 10f_1^3.
\end{align}
The KP hierarchy \hyperref[eq:kph]{(\ref*{eq:kph})} written in the variables $w_\a$ has the form
\begin{equation*}
    \centering\dd{w_\a}{T_\b}=\d_xR_{\a\b},\quad\a,\b\ge1,
\end{equation*}
where
\begin{gather*}
    R_{\a\b}\in\R_{w;\ge1},\\
    \deg R_{\a\b}=\a+\b,\\
    R_{\a,1}=R_{1,\a}=w_\a,\\
    R_{\a\b}=R_{\b\a}.
\end{gather*}
In \cite[Lemma 3.9]{brz24} it was proved that $R_{\a\b}\in\R_{w;\ge1}^\mathrm{ev}$, as opposed to just $R_{\a\b}\in\R_{w;\ge1}$.

\subsection{The BKP hierarchy}

There is an involution on the algebra of pseudo-differential operators given by
\begin{equation*}
    \Biggl(\sum_{n\le m}a_n\d_x^n\Biggr)^\dagger:=\sum_{n\le m}(-\d_x)^n\circ a_n.
\end{equation*}
The BKP hierarchy \cite{dkm81,djkm82} is a reduction of the KP hierarchy, obtained by imposing the constraint
\begin{equation}\label{eq:bkpred}
    L^\dagger = -\d_xL\d_x^{-1}
\end{equation}
on the pseudo-differential operator $L=\d_x+\sum_{i\ge1}f_i\d_x^{-i}$. Comparing the coefficients of $\d_x^{-i},i\ge1$, in equation \hyperref[eq:bkpred]{(\ref*{eq:bkpred})} one finds that
\begin{equation}\label{eq:fbkp}
    f_{2i}=-if_{2i-1}^{(1)}-\frac12\sum_{j=1}^{2i-2}\binom{2i-1}{j-1}f_j^{(2i-j)},\qquad i\ge1.
\end{equation}
Therefore, under the BKP constraint, we have
\begin{equation*}
    f_2 = -f_1^{(1)}, \; f_4 = -2f_3^{(1)} + f_1^{(3)}, \; \dots
\end{equation*}
In general, the variables $f_\mathrm{even}$ indexed by even numbers can be expressed solely in terms of variables indexed by odd numbers:
\begin{equation}\label{eq:fev}
    f_{2i}\bigl(f_\mathrm{odd}^{(*)}\bigr)=\sum_{j=1}^i\binom{2i}{2j-1}\frac{(1-2^{2j})B_{2j}}{j}f_{2i-2j+1}^{(2j-1)}, \qquad i \ge1,
\end{equation}
where $B_{2j}$ denotes the Bernoulli number.
To prove this, one writes $f_{2i}=\sum_{j=1}^ic_{i,j}f_{2i-2j+1}^{(2j-1)}$ and proves that $c_{i,j}=\binom{2i}{2j-1}\frac{(1-2^{2j})B_{2j}}{j}$ for all $j\le i$ and $i\ge1$ by induction on $i$. One uses the property \hyperref[eq:fbkp]{(\ref*{eq:fbkp})} together with the standard property of the Bernoulli numbers $\sum_{j=0}^n\binom{n+1}{j}B_j=\delta_{n,0}$ and the identity \cite{ll05,mag08}
\begin{equation*}
    \sum_{j=0}^n\binom{2n+1}{2j}(2-2^{2j})B_{2j}=\delta_{n,0}.
\end{equation*}
Moreover, the constraint \hyperref[eq:bkpred]{(\ref*{eq:bkpred})} is invariant under the odd flows of the KP hierarchy (see \cite{zab21}):
\begin{equation*}
    \frac{\d}{\d T_{2k-1}}(L^\dagger+\d_xL\d_x^{-1})=0,\qquad k\ge1.
\end{equation*}
The BKP hierarchy is the system compatible evolutionary PDEs in the odd variables $f_1,f_3,f_5,\dots$ depending on the odd times $x=T_1,T_3,T_5,\dots$ given by
\begin{equation*}
    \dd{f_{2i-1}}{T_{2k-1}}=S^\mathrm{BKP}_{i,k},\qquad i,k\ge1,
\end{equation*}
where $S^\mathrm{BKP}_{i,k}$ denotes $S_{2i-1,2k-1}$ under the substitutions $f_\mathrm{even}^{(*)}\mapsto f_\mathrm{even}^{(*)}\bigl(f_\mathrm{odd}^{(*)}\bigr)$ given by \hyperref[eq:fev]{(\ref*{eq:fev})}. Let us write $S_{1,2}^\mathrm{BKP}$ and $S_{2,2}^\mathrm{BKP}$, starting from \hyperref[eq:s13]{(\ref*{eq:s13})} and \hyperref[eq:s33]{(\ref*{eq:s33})}:
\begin{gather}
    \label{eq:s13bkp} S_{1,2}^\mathrm{BKP} = 3f_3^{(1)} - 2f_1^{(3)} + 6f_1f_1^{(1)}, \\
    \label{eq:s33bkp} S_{2,2}^\mathrm{BKP} = 3f_5^{(5)} - 5f_3^{(3)} + 3f_1f_3^{(1)} + 9f_1^{(1)}f_3 + 3f_1^{(5)} + 3f_1f_1^{(3)} + 9f_1^{(1)}f_1^{(2)}.
\end{gather}

The normal coordinates $w_\a=\operatorname{res}L^\a$ of the KP hierarchy give normal coordinates $w_{2\a-1}$ of the BKP hierarchy. Once the BKP constraint \hyperref[eq:bkpred]{(\ref*{eq:bkpred})} is imposed, equations \hyperref[eq:w1]{(\ref*{eq:w1})}-\hyperref[eq:w1]{(\ref*{eq:w5})} combined with \hyperref[eq:fev]{(\ref*{eq:fev})} give the first few values of $w_{2\a-1}(f_\mathrm{odd}^{(*)})$:
\begin{align}
    \label{eq:w1bkp} w_1 & = f_1, \\
    w_3 & = 3f_3 - 2f_1^{(2)} + 3f_1^2, \\
    \label{eq:w5bkp} w_5 & = 5f_5 - 10f_3^{(2)} + 20f_1f_3 + 6f_1^{(4)} - 10f_1f_1^{(2)} + 5f_1^{(1)}f_1^{(1)} + 10f_1^3.
\end{align}
Inverting this yields the first few values of $f_{2\a-1}(w_\mathrm{odd}^{(*)})$:
\begin{align}
    \label{eq:f1bkp} f_1 & = w_1, \\
    f_3 & = \frac13w_3 + \frac23w_1^{(2)} - w_1^2, \\
    \label{eq:f5bkp} f_5 & = \frac15w_5 + \frac23w_3^{(2)} - \frac43w_1w_3 + \frac{2}{15}w_1^{(4)} - \frac{14}{3}w_1w_1^{(2)} - 5w_1^{(1)}w_1^{(1)} + 2w_1^3.
\end{align}
In normal coordinates, the BKP hierarchy has the following form:
\begin{equation}\label{eq:bkphier}
    \dd{w_{2\a-1}}{T_{2\b-1}}=\d_xR^\mathrm{BKP}_{\a\b},\qquad\a,\b\ge1,
\end{equation}
where $R^\mathrm{BKP}_{\a\b}$ denotes $R_{2\a-1,2\b-1}$ with $w_\mathrm{even}^{(*)}\mapsto w_\mathrm{even}^{(*)}(w_\mathrm{odd}^{(*)})$, and
\begin{gather}
    \label{eq:bkpfirst} R_{\a\b}^\mathrm{BKP}\in\R_{w;\ge1}^\mathrm{ev}, \\
    \deg R_{\a\b}^\mathrm{BKP}=2\a+2\b-2, \\
    R_{\a,1}^\mathrm{BKP}=R_{1,\a}^\mathrm{BKP}=w_{2\a-1}, \\
    R_{\a\b}^\mathrm{BKP}=R_{\b\a}^\mathrm{BKP}.
\end{gather}
Combining \hyperref[eq:sik]{(\ref*{eq:sik})} and \hyperref[eq:waf]{(\ref*{eq:waf})} allows one to conclude that
\begin{equation}
     R_{\a\b}^\mathrm{BKP} = \frac{(2\a-1)(2\b-1)}{2\a+2\b-3}w_{2\a+2\b-3} + \widetilde{R}^\mathrm{BKP}_{\a\b}(w^{(*)}_{\le 2\a+2\b-5}), \quad \widetilde{R}^\mathrm{BKP}_{\a\b}\in\R^\mathrm{ev}_{w;\ge1},
\end{equation}
while combining \hyperref[eq:s13bkp]{(\ref*{eq:s13bkp})}-\hyperref[eq:s33bkp]{(\ref*{eq:s33bkp})} with \hyperref[eq:w1bkp]{(\ref*{eq:w1bkp})}-\hyperref[eq:f5bkp]{(\ref*{eq:f5bkp})} gives
\begin{equation}\label{eq:R33}
    R_{2,2}^\mathrm{BKP} = \frac95w_5 - w_3^{(2)} - 3w_1w_3 + \frac15w_1^{(4)} + 3w_1w_1^{(2)} + 3w_1^3.
\end{equation}
In the next section we show that this is all the information that is needed to uniquely determine all of the differential polynomials $R_{\a\b}^\mathrm{BKP}$, thanks to the commutativity of the flows of the BKP hierarchy.

\begin{center}\section{Reconstruction}\label{sec:rec}\end{center}
\subsection{Main result}

In \hyperref[sec:DRcohft]{Section \ref*{sec:DRcohft}} the following compatible system of PDEs was obtained by reducing the DR hierarchy associated to the partial CohFT $c_{g,n}(\otimes_{i=1}^n\a_i)=\HresSpin{g}{2\a_1,\dots,2\a_n}$ of \hyperref[prop:cohft]{Proposition \ref*{prop:cohft}} to the case where the differentials have two zeros, and switching to normal coordinates:
\begin{equation}\label{eq:normaldr}
    \dd{v_\a}{t^\b}=\d_x\Qspin{\a\b},\qquad\a,\b\ge1,
\end{equation}
where the differential polynomials $\Qspin{\a\b}$ satisfy
\begin{gather}
    \label{eq:normaldr1} \Qspin{\a\b}\in\widehat{\R}_{v;\ge1}^{\mathrm{ev};[0]}, \\
    \deg\Qspin{\a\b}=2\a+2\b-2, \\
    \Qspin{\a,1}=\Qspin{1,\a}=v_\a, \\
    \Qspin{\a\b}=\Qspin{\b\a}, \\
    \label{eq:Qab} \Qspin{\a\b} = v_{\a+\b-1} + \widetilde{Q}^\mathrm{spin}_{\a\b}(v^{(*)}_{\le\a+\b-2};\e), \quad \widetilde{Q}^\mathrm{spin}_{\a\b}\in\widehat{\R}^{\mathrm{ev};[0]}_{v;\ge1}, \\
    \label{eq:Q22} \Qspin{2,2} = v_3 + v_1v_2 + \frac13v_1^3 - \frac{\e^2}{6}v_2^{(2)}  - \frac{\e^2}{6}v_1v_1^{(2)} + \frac{\e^4}{180}v_1^{(4)}.
\end{gather}
The main result of this paper is the following.
\begin{theorem}\label{thm:main}
    The systems \hyperref[eq:normaldr]{(\ref*{eq:normaldr})} and \hyperref[eq:bkphier]{(\ref*{eq:bkphier})} are related by the change of variables
    \begin{equation*}
        v_\a = -\frac{w_{2\a-1}}{2\a-1},\quad t^\b=(2\b-1)T_{2\b-1}, \quad \a\ge1,\b\ge1,
    \end{equation*}
    together with the substitution $\e^2=2$.
\end{theorem}
\begin{proof}
    The substitution sends properties \hyperref[eq:normaldr1]{(\ref*{eq:normaldr1})}-\hyperref[eq:Q22]{(\ref*{eq:Q22})} of the reduced DR hierarchy to properties \hyperref[eq:bkpfirst]{(\ref*{eq:bkpfirst})}-\hyperref[eq:R33]{(\ref*{eq:R33})} of the BKP hierarchy. Then the result follows from \hyperref[thm:bkpr]{Theorem \ref*{thm:bkpr}}.
\end{proof}

\begin{remark}\label{rem:Prelations}
    Tracing back the definitions, one finds (with $\e^2=2$) 
    \begin{equation}\label{eq:traceback}
        \partial_x{\Pspin{\alpha\beta}} = \frac{\mathfrak{T}_u^f\left(\res[L^{2\alpha-1},(L^{2\beta-1})_+]\right)}{(2\a-1)(2\b-1)} - \frac{\partial}{\partial t^\beta} \widetilde{P}^\mathrm{spin}_{1,\a}.
    \end{equation}
    Here $\mathfrak{T}^f_u$ denotes the coordinate transformation from $f^{(*)}_\mathrm{odd}$ to $u^{(*)}_*$ determined by 
    \begin{equation*}
        \Pspin{1,\a}(u^{(*)}_*;2) = -\frac{1}{2\a-1}\res L^{2\a-1}(f^{(*)}_\mathrm{odd})
    \end{equation*}
    and equation \hyperref[eq:fev]{(\ref*{eq:fev})}. Note however that $\res[L^{2\alpha-1},(L^{2\beta-1})_+]$ is a function of $f^{(*)}_{\leq 2\alpha+2\beta-3}$, while the coordinate transformation requires knowledge of $\Pspin{1,\a+\b-1}$, which concerns a trivial order condition but is of the same degree as $\Pspin{\a\b}$, to compute where to send $f_{2\alpha+2\beta-3}$. In contrast, the reconstruction result for the $\Qspin{\a\b}$ below determines differential polynomials of strictly higher degree. \newline

    \noindent For $(\alpha,\beta) = (3,2)$, equation \hyperref[eq:traceback]{(\ref*{eq:traceback})} reads $\partial_x{\Pspin{32}} + \frac{\partial}{\partial t^2} \widetilde{P}^\mathrm{spin}_{13} = \frac{1}{15}\mathfrak{T}_u^f\left(\res[L^{5},(L^{3})_+]\right)$.
    We know $\Pspin{1,\alpha}$ for $\alpha\leq 3$ and $\Pspin{22}$ from \hyperref[prop:Pspins]{Proposition \ref*{prop:Pspins}} for part of the coordinate transformation to $f_\mathrm{odd}^{(*)}$ and the computation of $\frac{\partial}{\partial t^2}\widetilde{P}^\mathrm{spin}_{13}$ respectively. The degree 8 differential polynomials $\Pspin{14}(u_{\le4}^{(*)};2)$ and $\Pspin{32}(u_{\le4}^{(*)};2)$ have terms involving intersection numbers in genus $g\leq 3$. Non-trivial genus $0$ terms can be computed via \cite[Example 3.4]{br23} and genus $1$ terms are computed similarly to \hyperref[lem:g1hain]{Lemma \ref*{lem:g1hain}} and \hyperref[lem:g1int1]{Lemma \ref*{lem:g1int1}}. In genus $2$, comparing coefficients gives the relation:
    \[\int_{\DR_2(-a,0,a)}\lambda_2\HresSpin{2}{4,2,-4} = \int_{\DR_2(-a,0,a)}\lambda_2\HresSpin{2}{0,6,-4} -  \frac{797}{5760}a^4.\] 
    This is confirmed by Propositions \ref{lem:firstintegral} and \ref{g2n3integral}. Using the values for these genus $2$ integrals we get the following in genus $3$:
    \[\int_{\DR_3(-a,0,a)}\lambda_3\HresSpin{3}{4,2,-2} = \int_{\DR_3(-a,0,a)}\lambda_3\HresSpin{3}{0,6,-2} - \frac{3933193}{111283200}a^6.\] 
    The same coordinate transformation produces similar relations for $(\alpha,\beta) = (2,3)$ and shows how redistributing the orders of the zeroes affects the values.
\end{remark}

\subsection{BKP hierarchy reconstruction}

\begin{theorem}\label{thm:bkpr}
    The commutativity of the flows $\left\{\frac{\d}{\d t^\b}\right\}_{\b\ge1}$ of the system \hyperref[eq:normaldr]{(\ref*{eq:normaldr})}, together with the properties \hyperref[eq:normaldr1]{(\ref*{eq:normaldr1})}-\hyperref[eq:Q22]{(\ref*{eq:Q22})}, completely determines all the differential polynomials $\Qspin{\a\b}$.
\end{theorem}
The proof of the theorem is split into several steps. In what follows we write $Q_{\a\b}:=\Qspin{\a\b}$ for brevity.

\begin{lemma}
    In the situation of \hyperref[thm:bkpr]{Theorem \ref*{thm:bkpr}}, all the polynomials $Q_{\a\b}$ with $\a,\b\ge3$ can be determined recursively from the polynomials $Q_{\g,2}$ with $\g\le\a+\b-2$.
\end{lemma}
\begin{proof}
    For $p,q\ge1$, consider the compatibility relations $\frac{\d}{\d t^q}\dd{v_p}{t^2}=\frac{\d}{\d t^2}\dd{v_p}{t^q}$ of the system. These are equivalent to $\dd{Q_{p,2}}{t^q}=\dd{Q_{p,q}}{t^2}$. By writing $Q_{p,2}=v_{p+1}+\widetilde{Q}_{p,2}$ and $Q_{p,q}=v_{p+q-1}+\widetilde{Q}_{p,q}$ one obtains
    \begin{equation}\label{eq:pq2b}
        \d_xQ_{p+1,q} = \d_xQ_{p+q-1,2} + \sum_{i=1}^{p+q-2}\sum_{j\ge0}\dd{\widetilde{Q}_{p,q}}{v_i^{(j)}}\d_x^{j+1}Q_{i,2} - \sum_{i=1}^p\sum_{j\ge0}\dd{\widetilde{Q}_{p,2}}{v_i^{(j)}}\d_x^{j+1}Q_{i,q}.
    \end{equation}
    Setting $(p,q)=(\a-1,\b)$ gives 
    \begin{equation*}
        \d_xQ_{\a\b} = \d_xQ_{\a+\b-2,2} + \sum_{i=1}^{\a+\b-3}\sum_{j\ge0}\dd{\widetilde{Q}_{\a-1,\b}}{v_i^{(j)}}\d_x^{j+1}Q_{i,2} - \sum_{i=1}^{\a-1}\sum_{j\ge0}\dd{\widetilde{Q}_{\a-1,2}}{v_i^{(j)}}\d_x^{j+1}Q_{i,\b}.
    \end{equation*}
    Let $2d:=\deg Q_{\a\b}=2\a+2\b-2$. The first term on the right hand side of the above equation contains the polynomial $Q_{d-1,2}$, while all the remaining terms contain polynomials $Q_{\g\delta}$ with $\deg Q_{\g\delta}\le2d-2$. This allows one to recursively determine $Q_{\a\b}$ starting from the polynomials $Q_{\g,2}$ with $\g\le\a+\b-2$ and the initial condition $Q_{2,2} = v_3 + v_1v_2 - \frac{\e^2}{6}v_2^{(2)} + \frac13v_1^3 - \frac{\e^2}{6}v_1v_1^{(2)} + \frac{\e^4}{180}v_1^{(4)}$.
\end{proof}
By the above lemma, to prove \hyperref[thm:bkpr]{Theorem \ref*{thm:bkpr}} it suffices to determine the polynomials $Q_{\a,2}$ starting from the knowledge of $Q_{2,2}$ given by \hyperref[eq:Q22]{(\ref*{eq:Q22})}. Note that the proof of this determination uses a non-constructive argument only at the end. \newline

We introduce some notation. For a differential polynomial $P\in\widehat{\R}_v$, we denote the coefficient of $\e^h$ in $P$ by $P^{(h)}\in\R_v$ and the $\widetilde{\deg}=d$ part of $P$ by $P_d\in\widehat{\R}_{v;d}$. Moreover, denote the $\widetilde{\deg}\le l$ (resp. $\widetilde{\deg}\ge l$) part of $P$ by $P_{\le l}\in\widehat{\R}_{v;\le l}$ (resp. $P_{\ge l}\in\widehat{\R}_{v;\ge l}$). In the case $P=Q_{\a,2}$ we can write
\begin{equation*}
    Q_{\a,2} = \sum_{g=0}^{\a}\e^{2g}(Q_{\a,2})^{(2g)} = \sum_{g=0}^\a\e^{2g}\sum_{d=1}^{\a+1-g}\bigl(Q_{\a,2}\bigr)_d^{(2g)}.
\end{equation*}
Due to the initial conditions we have $\bigl(Q_{\a,2}\bigr)_1^{(0)}=v_{\a+1}$.
\begin{lemma}\label{lem:deg2}
    For every $\a\ge2$,
    \begin{equation*}
        \bigl(Q_{\a,2}\bigr)_2^{(0)} = \sum_{i=1}^{\a-1}v_iv_{\a+1-i}.
    \end{equation*}
\end{lemma}
\begin{proof}
    For $\a=2$ this is true due to the initial condition \hyperref[eq:Q22]{(\ref*{eq:Q22})}. Let $\a\ge3$. Since $\deg\,Q_{\a,2}=2\a+2$, one may write $\bigl(Q_{\a,2}\bigr)_2^{(0)} = \frac12\sum_{i=1}^\a\lambda_iv_iv_{\a+1-i}$
    for some constants $\lambda_i\in\C$ satisfying $\lambda_i=\lambda_{\a+1-i}$. Let us show that $\lambda_1=\lambda_\a=1$ and $\lambda_i=\lambda_{\a+1-i}=2$ for $2\le i\le\a-1$.
    Equation \hyperref[eq:pq2b]{(\ref*{eq:pq2b})} with $(p,q)=(2,\a)$ gives
    \begin{align}
        \label{eq:bonusab} \d_x\mathcal{Q}_\a := \d_xQ_{\a,3}-\d_xQ_{\a+1,2} = & \sum_{i=1}^\a\sum_{j\ge0}\dd{\widetilde{Q}_{\a,2}}{v_i^{(j)}}\d_x^{j+1}Q_{i,2} - v_1\d_xQ_{\a,2} + \frac{\e^2}{6}\d_x^3Q_{\a,2} \\
        \nonumber & - \left(v_2+v_1^2-\frac{\e^2}{6}v_1^{(2)}\right)v_\a^{(1)} + \frac{\e^2}{6}v_1v_\a^{(3)} - \frac{\e^4}{180}v_\a^{(5)}.
    \end{align}
    We consider the degree $\widetilde{\deg}=2$ part of the $\e^0$ coefficient of the above equation.
    \begin{align*}
        \d_x\left(\bigl(\mathcal{Q}_\a\bigr)_2^{(0)}\right) & = \left(\d_x\mathcal{Q}_\a\right)_2^{(0)} = \sum_{i=1}^\a\sum_{j\ge0}\dd{\bigl(\widetilde{Q}_{\a,2}\bigr)_2^{(0)}}{v_i^{(j)}}\d_x^{j+1}\bigl(Q_{i,2}\bigr)_1^{(0)} - v_1\d_x\bigl(Q_{\a,2}\bigr)_1^{(0)} - v_2v_\a^{(1)} \\
        & = \sum_{i=1}^\a\frac{\d}{\d v_i}\biggl(\frac12\sum_{k=1}^\a\lambda_kv_kv_{\a+1-k}\biggr)v_{i+1}^{(1)} - v_1v_{\a+1}^{(1)} - v_2v_\a^{(1)} \\
        & = \sum_{i=1}^\a\lambda_iv_{\a+1-i}v_{i+1}^{(1)} - v_1v_{\a+1}^{(1)} - v_2v_\a^{(1)} \\
        & = (\lambda_\a-1)v_1v_{\a+1}^{(1)} + \lambda_1v_2^{(1)}v_\a + (\lambda_{\a-1}-1)v_2v_\a^{(1)} + \sum_{i=2}^{\a-2}\lambda_iv_{\a+1-i}v_{i+1}^{(1)}.
    \end{align*}
    The left-hand side of the equation lies in the image of the operator $\d_x$. For the right-hand side to lie in this image it must be true that $\lambda_\a-1=0$ and $\lambda_1=\lambda_{\a-1}-1$, so $\lambda_1=\lambda_\a=1$ and $\lambda_2=\lambda_{\a-1}=2$. The remaining sum is
    \begin{equation*}
        \sum_{i=2}^{\a-2}\lambda_iv_{\a+1-i}v_{i+1}^{(1)} = \frac12\sum_{i=2}^{\a-2}\left(\lambda_iv_{\a+1-i}v_{i+1}^{(1)} + \lambda_{i+1}v_{\a+1-i}^{(1)}v_{i+1}\right).
    \end{equation*}
    For this to lie in the image of $\d_x$ it must be true that $\lambda_i=\lambda_{i+1}$ for all $2\le i \le\a-2$, proving the claim.
\end{proof}
Before proceeding further, let us introduce some additional notation. Any differential polynomial $P\in\widehat{\R}_{v;\ge1}$ can be written unambiguously as
\begin{equation*}
    P = \sum_{\g\ge1}\sum_{n\ge1}\sum_{k_1,\dots,k_n\ge0}P_{k_1,\dots,k_n}^\g\bigl(v_1^{(*)},\dots,v_{\g-1}^{(*)}\bigr)\,v_\g^{(k_1)}\cdots v_\g^{(k_n)},
\end{equation*}
with all but finitely many of the $P_{k_1,\dots,k_n}^\g\in\widehat{\R}_v$ being zero. For each $\g\ge1$ define
\begin{equation*}
    P^{[\g]} := \sum_{n\ge1}\sum_{k_1,\dots,k_n\ge0}P_{k_1,\dots,k_n}^\g\bigl(v_1^{(*)},\dots,v_{\g-1}^{(*)}\bigr)\,v_\g^{(k_1)}\cdots v_\g^{(k_n)}.
\end{equation*}
This is the sum of all monomials in $P$ that explicitly contain the variable $v_\g^{(*)}$, but no variables $v_{>\g}^{(*)}$. Moreover, let
\begin{equation}\label{eq:Pge}
    P^{[\ge\g]} := \sum_{\g^\prime\ge\g}P^{[\g^\prime]} \quad \text{and} \quad P^{[\le\g]} := \sum_{\g^\prime\le\g}P^{[\g^\prime]}.
\end{equation}
Let us write $Q_{\a,2}$ using this notation:
\begin{equation*}
    Q_{\a,2} = v_{\a+1} + \sum_{\g=1}^\a Q_{\a,2}^{[\g]}.
\end{equation*}
We further split each $Q_{\a,2}^{[\g]}$ into a linear part of $\widetilde{\deg}=1$ and the remaining part of $\widetilde{\deg}\ge2$:
\begin{equation}\label{eq:Qnotation}
    Q_{\a,2}^{[\g]} = \e^{2\a-2\g+2}\,C_\g\,v_\g^{(2\a-2\g+2)} + Q_{\a,2}^{[[\g]]}, \quad C_\g\in\C, \quad Q_{\a,2}^{[[\g]]}\in\widehat{\R}_{v;\ge2}^{\mathrm{ev};[0]}, \quad 1\le\g\le\a,
\end{equation}
where $Q_{\a,2}^{[[\g]]}:=\bigl(Q_{\a,2}^{[\g]}\bigr)_{\ge2}$. The previous lemma implies that
\begin{equation*}
    Q_{\a,2}^{[[\a]]} = v_1v_\a.
\end{equation*}
Thus $Q_{\a,2}^{[\a]}=\e^2C_\a v_\a^{(2)}+v_1v_\a$, and so one can evaluate the $i=\a$ term of the sum in equation \hyperref[eq:bonusab]{(\ref*{eq:bonusab})} to obtain
\begin{align}
    \label{eq:bonusaby} \d_x\mathcal{Q}_\a := & \;\d_xQ_{\a,3} - \d_xQ_{\a+1,2} = \e^2\Bigl(C_\a+\frac16\Bigr)\d_x^3Q_{\a,2} \\
    \nonumber & + \underbrace{\sum_{i=1}^{\a-1}\sum_{j\ge0}\dd{Q_{\a,2}}{v_i^{(j)}}\d_x^{j+1}Q_{i,2} - \left(v_2+v_1^2-\frac{\e^2}{6}v_1^{(2)}\right)v_\a^{(1)} + \frac{\e^2}{6}v_1v_\a^{(3)} - \frac{\e^4}{180}v_\a^{(5)}}_{=:\d_x\mathcal{Q}_\a^\prime}.
\end{align}
Now consider setting $(p,q)=(2,\a+1)$ in \hyperref[eq:pq2b]{(\ref*{eq:pq2b})}, which gives
\begin{align*}
    \d_xQ_{\a+1,3}-\d_xQ_{\a+2,2} & = \sum_{i=1}^{\a+1}\sum_{j\ge0}\dd{\widetilde{Q}_{\a+1,2}}{v_i^{(j)}}\d_x^{j+1}Q_{i,2} - v_1\d_xQ_{\a+1,2} + \frac{\e^2}{6}\d_x^3Q_{\a+1,2} \\
    & - \Bigl(v_2+v_1^2-\frac{\e^2}{6}v_1^{(2)}\Bigr)v_{\a+1}^{(1)} + \frac{\e^2}{6}v_1v_{\a+1}^{(3)} - \frac{\e^4}{180}v_{\a+1}^{(5)}.
\end{align*}
On the other hand, setting $(p,q)=(\a,3)$ gives
\begin{align*}
    \d_xQ_{\a+1,3}-\d_xQ_{\a+2,2}
    & = \sum_{i=1}^{\a+1}\sum_{j\ge0}\dd{\widetilde{Q}_{\a,3}}{v_i^{(j)}}\d_x^{j+1}Q_{i,2} - \sum_{i=1}^\a\sum_{j\ge0}\dd{\widetilde{Q}_{\a,2}}{v_i^{(j)}}\d_x^{j+1}Q_{i,3}.
\end{align*}
Equating the right-hand sides of the last two equations yields
\begin{align*}
    & - v_1\d_xQ_{\a+1,2} + \frac{\e^2}{6}\d_x^3Q_{\a+1,2} - \Bigl(v_2+v_1^2-\frac{\e^2}{6}v_1^{(2)}\Bigr)v_{\a+1}^{(1)} + \frac{\e^2}{6}v_1v_{\a+1}^{(3)} - \frac{\e^4}{180}v_{\a+1}^{(5)} \\
    & = \sum_{i=1}^{\a+1}\sum_{j\ge0}\dd{\mathcal{Q}_\a}{v_i^{(j)}}\d_x^{j+1}Q_{i,2} - \sum_{i=1}^\a\sum_{j\ge0}\dd{\widetilde{Q}_{\a,2}}{v_i^{(j)}}\d_x^{j+1}Q_{i,3},
\end{align*}
where we have used $\widetilde{Q}_{\a,3}-\widetilde{Q}_{\a+1,2}=Q_{\a,3}-Q_{\a+1,2}=\mathcal{Q}_\a$. To remove the terms involving $Q_{\a+1,2}$ we evaluate the $i=\a+1$ term of the first sum and the $i=\a$ term of the second sum. To do this, we use $Q_{\a,2}^{[\a]}=\e^2C_\a v_\a^{(2)}+v_1v_\a$ and the following facts that can be inferred from equation \hyperref[eq:bonusaby]{(\ref*{eq:bonusaby})}:
\begin{gather*}
    \mathcal{Q}_\a^{[\a+1]} = \e^2\Bigl(C_\a+\frac16\Bigr)v_{\a+1}^{(2)}, \\
    Q_{\a,3}=Q_{\a+1,2} + \e^2\Bigl(C_\a+\frac16\Bigr)\d_x^2Q_{\a,2}+\mathcal{Q}_\a^\prime.
\end{gather*}
This leaves us with
\begin{align*}
    & \e^2\Bigl(C_\a+\frac16\Bigr)v_1\d_x^3Q_{\a,2} + \e^4C_\a\Bigl(C_\a+\frac16\Bigr)\d_x^5Q_{\a,2} +  v_1\d_x\mathcal{Q}_\a^\prime + \e^2C_\a\d_x^3\mathcal{Q}_\a^\prime \\ 
    & - \Bigl(v_2+v_1^2-\frac{\e^2}{6}v_1^{(2)}\Bigr)v_{\a+1}^{(1)} + \frac{\e^2}{6}v_1v_{\a+1}^{(3)} - \frac{\e^4}{180}v_{\a+1}^{(5)} \\
    = & \sum_{i=1}^\a\sum_{j\ge0}\dd{\mathcal{Q}_\a}{v_i^{(j)}}\d_x^{j+1}Q_{i,2} - \sum_{i=1}^{\a-1}\sum_{j\ge0}\dd{Q_{\a,2}}{v_i^{(j)}}\d_x^{j+1}Q_{i,3}.
\end{align*}
Denote the left-hand side of this equation by $P_\a$. Applying the operator $\d_x$ to both sides:
\begin{equation}\label{eq:masterby}
    \d_xP_\a =  \underbrace{\sum_{i=1}^\a\sum_{j\ge0}\dd{(\d_x\mathcal{Q}_\a)}{v_i^{(j)}}\d_x^{j+1}Q_{i,2}}_{=:\Pi_1} - \underbrace{\d_x\left(\sum_{i=1}^{\a-1}\sum_{j\ge0}\dd{Q_{\a,2}}{v_i^{(j)}}\d_x^{j+1}Q_{i,3}\right)}_{=:\Pi_2}.
\end{equation}
As mentioned previously, the objective is to uniquely determine the differential polynomials $Q_{\a,2}$ for every $\a\ge2$. The base case $\a=2$ is given by \hyperref[eq:Q22]{(\ref*{eq:Q22})}. So let $\a\ge3$ and suppose that $Q_{\a^\prime,2}$ is uniquely determined for every $\a^\prime\le\a-1$. To uniquely determine $Q_{\a,2}$, one must determine $Q_{\a,2}^{[\g]}$ for every $1\le\g\le\a$. This, in turn, is equivalent to determining the linear coefficient $C_\g$ and the $\widetilde{\deg}\ge2$ part $Q_{\a,2}^{[[\g]]}$ in the expression \hyperref[eq:Qnotation]{(\ref*{eq:Qnotation})}. The idea of the following proposition is that the ``$v_{\g+2}^{(*)}$ part'' of equation \hyperref[eq:masterby]{(\ref*{eq:masterby})}, by which we mean
\begin{equation*}
    \d_xP_\a^{[\g+2]} = \Pi_1^{[\g+2]}-\Pi_2^{[\g+2]},
\end{equation*}
gives relations that uniquely determine $Q_{\a,2}^{[[\g]]}$, as a function of the subset of linear coefficients $\{C_{\g^\prime}:\g^\prime\ge\g+1\}$. We write
\begin{equation*}
    Q_{\a,2}^{[[\g]]}=Q_{\a,2}^{[[\g]]}(C_{\ge\g+1})
\end{equation*}
to indicate this. For the moment, the linear coefficients remain unknown.
\begin{proposition}\label{prop:deg2}
    For every $1\le\gamma\le\a$, equations \hyperref[eq:bonusaby]{(\ref*{eq:bonusaby})} and \hyperref[eq:masterby]{(\ref*{eq:masterby})} determine the polynomial $Q_{\a,2}^{[[\g]]}$ explicitly as a function of the linear coefficients $C_{\g^\prime}$ with $\g^\prime\ge\g+1$.
\end{proposition}
\begin{proof}
    We proceed by induction on $\a\ge\g\ge1$. The base case $\g=\a$ follows from \hyperref[lem:deg2]{Lemma \ref*{lem:deg2}}, which gives $Q_{\a,2}^{[[\a]]}=v_1v_\a$. Now let $\a-1\ge\g\ge1$, and suppose that $Q_{\a,2}^{[[\g^\prime]]}=Q_{\a,2}^{[[\g^\prime]]}(C_{\ge\g^\prime+1})$ is determined for every $\g^\prime\ge\g+1$. The objective is to determine $Q_{\a,2}^{[[\g]]}$ as a function of $C_{\ge\g+1}$. By degree considerations, this is equivalent to determining $\dd{Q_{\a,2}^{[[\g]]}}{v_\g^{(j)}}$ as a function of $C_{\ge\g+1}$ for every $0\le j\le2\a-2\g$. As a preliminary, let us study the polynomial $\d_x\mathcal{Q}_\a^\prime$ of equation \hyperref[eq:bonusaby]{(\ref*{eq:bonusaby})}. We use the notation \hyperref[eq:Pge]{(\ref*{eq:Pge})}.
    \begin{align*}
        \d_x & \mathcal{Q}_\a^{\prime\,[\ge\g+1]} = \left(\sum_{i=1}^{\a-1}\sum_{j\ge0}\dd{Q_{\a,2}}{v_i^{(j)}}\d_x^{j+1}Q_{i,2}\right)^{[\ge\g+1]} \hspace{-1.5em} - \Bigl(v_2+v_1^2-\frac{\e^2}{6}v_1^{(2)}\Bigr)v_\a^{(1)} + \frac{\e^2}{6}v_1v_\a^{(3)} - \frac{\e^4}{180}v_\a^{(5)} \\
        = & \; \underbrace{\left(\sum_{i=1}^{\a-1}\sum_{j\ge0}\dd{Q_{\a,2}^{[\ge\g+1]}}{v_i^{(j)}}\d_x^{j+1}Q_{i,2}\right)^{[\ge\g+1]} \hspace{-1.5em} - \Bigl(v_2+v_1^2-\frac{\e^2}{6}v_1^{(2)}\Bigr)v_\a^{(1)} + \frac{\e^2}{6}v_1v_\a^{(3)} - \frac{\e^4}{180}v_\a^{(5)}}_{=:\overline{\mathcal{Q}}_\a^\prime} \\
        & + \left(\sum_{i=1}^{\a-1}\sum_{j\ge0}\dd{Q_{\a,2}^{[\g]}}{v_i^{(j)}}\d_x^{j+1}Q_{i,2}\right)^{[\ge\g+1]}.
    \end{align*}
    Notice that $Q_{\a,2}^{[\le\g-1]}$ does not contribute because $Q_{i,2}^{[\ge\g+1]}=0$ for $i\le\g-1$. By the inductive hypotheses (on both $\a$ and $\g$), the term $\overline{\mathcal{Q}}_\a^\prime$ is uniquely determined as a function of $C_{\ge\g+1}$, so we write $\overline{\mathcal{Q}}_\a^\prime=\overline{\mathcal{Q}}_\a^\prime(C_{\ge\g+1})$. The remaining term is equal to
    \begin{align*}
        & \left(\sum_{i=1}^{\a-1}\sum_{j\ge0}\dd{Q_{\a,2}^{[\g]}}{v_i^{(j)}}\d_x^{j+1}Q_{i,2}\right)^{[\g+1]} = \sum_{j\ge0}\dd{Q_{\a,2}^{[\g]}}{v_\g^{(j)}}\d_x^{j+1}Q_{\g,2}^{[\g+1]} \\
        = & \; \sum_{j\ge0}\frac{\d}{\d v_\g^{(j)}}\left(\e^{2\a-2\g+2}C_\g v_\g^{(2\a-2\g+2)} + Q_{\a,2}^{[[\g]]}\right)v_{\g+1}^{(j+1)} \\
        = & \; \e^{2\a-2\g+2}C_\g v_{\g+1}^{(2\a-2\g+3)} + \sum_{j=0}^{2\a-2\g}\dd{Q_{\a,2}^{[[\g]]}}{v_\g^{(j)}}v_{\g+1}^{(j+1)}.
    \end{align*}
    Thus, since $\d_x\mathcal{Q}_\a=\e^2\Bigl(C_\a+\frac16\Bigr)\d_x^3Q_{\a,2}+\d_x\mathcal{Q}_\a^\prime$, we have
    \begin{equation}
        \label{eq:gagp1} \d_x\mathcal{Q}_\a^{[\ge\g+1]} = \e^{2\a-2\g+2}C_\g v_{\g+1}^{(2\a-2\g+3)} + \sum_{j=0}^{2\a-2\g}\dd{Q_{\a,2}^{[[\g]]}}{v_\g^{(j)}}v_{\g+1}^{(j+1)} + \underbrace{\e^2\Bigl(C_\a+\frac16\Bigr)\d_x^3Q_{\a,2}^{[\ge\g+1]} + \overline{\mathcal{Q}}_\a^\prime}_{=:\overline{\mathcal{Q}}_\a},
    \end{equation}
    where $\overline{\mathcal{Q}}_\a = \overline{\mathcal{Q}}_\a(C_{\ge\g+1})$ is uniquely determined. Let us now study the ``$v_{\g+2}^{(*)}$ part'' of equation \hyperref[eq:masterby]{(\ref*{eq:masterby})}, namely $\d_xP_\a^{[\g+2]} = \Pi_1^{[\g+2]}-\Pi_2^{[\g+2]}$. Firstly, from the definition of $P_\a$ and the previous discussion, we have that $\d_xP_\a^{[\g+2]}=\d_xP_\a^{[\g+2]}(C_{\ge\g+1})$ is uniquely determined. Secondly,
    \begin{align}
        \nonumber \Pi_1^{[\g+2]} = & \; \left(\sum_{i=1}^\a\sum_{j\ge0}\dd{(\d_x\mathcal{Q}_\a)}{v_i^{(j)}}\d_x^{j+1}Q_{i,2}\right)^{[\g+2]} = \left(\sum_{i=1}^\a\sum_{j\ge0}\dd{(\d_x\mathcal{Q}_\a^{[\ge\g+1]})}{v_i^{(j)}}\d_x^{j+1}Q_{i,2}\right)^{[\g+2]} \\
        \nonumber \overset{\hyperref[eq:gagp1]{(\ref*{eq:gagp1})}}{=} & \; \left(\sum_{i=1}^\a\sum_{j\ge0}\frac{\d}{\d v_i^{(j)}}\left(\e^{2\a-2\g+2}C_\g v_{\g+1}^{(2\a-2\g+3)} + \sum_{l=0}^{2\a-2\g}\dd{Q_{\a,2}^{[[\g]]}}{v_\g^{(l)}}v_{\g+1}^{(l+1)}\right)\d_x^{j+1}Q_{i,2}\right)^{[\g+2]} \\
        \nonumber & + \underbrace{\left(\sum_{i=1}^\a\sum_{j\ge0}\dd{\overline{\mathcal{Q}}_\a}{v_i^{(j)}}\d_x^{j+1}Q_{i,2}\right)^{[\g+2]}}_{=:\overline{\Pi}_1} \\
        \nonumber = & \; \sum_{j\ge0}\frac{\d}{\d v_{\g+1}^{(j)}}\left(\e^{2\a-2\g+2}C_\g v_{\g+1}^{(2\a-2\g+3)} + \sum_{l=0}^{2\a-2\g}\dd{Q_{\a,2}^{[[\g]]}}{v_\g^{(l)}}v_{\g+1}^{(l+1)}\right)v_{\g+2}^{(j+1)} + \overline{\Pi}_1 \\
        \label{eq:Pi1} = & \; \e^{2\a-2\g+2}C_\g v_{\g+2}^{(2\a-2\g+4)} + \sum_{j=0}^{2\a-2\g}\dd{Q_{\a,2}^{[[\g]]}}{v_\g^{(j)}}v_{\g+2}^{(j+2)} + \overline{\Pi}_1,
    \end{align}
    where $\overline{\Pi}_1=\overline{\Pi}_1(C_{\ge\g+1})$ is uniquely determined by the previous discussion. Thirdly,
    \begin{align}
        \nonumber \Pi_2^{[\g+2]} & = \d_x\left(\sum_{i=1}^{\a-1}\sum_{j\ge0}\dd{Q_{\a,2}}{v_i^{(j)}}\d_x^{j+1}Q_{i,3}\right)^{[\g+2]} = \d_x\left(\sum_{i=1}^{\a-1}\sum_{j\ge0}\dd{Q_{\a,2}^{[\ge\g]}}{v_i^{(j)}}\d_x^{j+1}Q_{i,3}\right)^{[\g+2]} \\
        \nonumber & = \d_x\left(\sum_{i=1}^{\a-1}\sum_{j\ge0}\dd{Q_{\a,2}^{[\g]}}{v_i^{(j)}}\d_x^{j+1}Q_{i,3}\right)^{[\g+2]} \hspace{-1.5em} + \underbrace{\d_x\left(\sum_{i=1}^{\a-1}\sum_{j\ge0}\dd{Q_{\a,2}^{[\ge\g+1]}}{v_i^{(j)}}\d_x^{j+1}(\mathcal{Q}_i+Q_{i+1,2})\right)^{[\g+2]}}_{=:\overline{\Pi}_2} \\
        \nonumber & = \d_x\left(\sum_{j\ge0}\frac{\d}{\d v_\g^{(j)}}\left(\e^{2\a-2\g+2}C_\g v_\g^{(2\a-2\g+2)} + Q_{\a,2}^{[[\g]]}\right)v_{\g+2}^{(j+1)}\right) + \overline{\Pi}_2 \\
        \label{eq:Pi2} & = \e^{2\a-2\g+2}C_\g v_\g^{(2\a-2\g+4)} + \sum_{j=0}^{2\a-2\g}\d_x\left(\dd{Q_{\a,2}^{[[\g]]}}{v_\g^{(j)}}\right)v_{\g+2}^{(j+1)} + \sum_{j=0}^{2\a-2\g}\dd{Q_{\a,2}^{[[\g]]}}{v_\g^{(j)}}v_{\g+2}^{(j+2)} + \overline{\Pi}_2,
    \end{align}
    where $\overline{\Pi}_2=\overline{\Pi}_2(C_{\ge\g+1})$ is uniquely determined by the previous discussion. Then, by combining the expressions \hyperref[eq:Pi1]{(\ref*{eq:Pi1})} and \hyperref[eq:Pi2]{(\ref*{eq:Pi2})}, the equation $\d_xP_\a^{[\g+2]}=\Pi_1^{[\g+2]}-\Pi_2^{[\g+2]}$ yields
    \begin{equation}\label{eq:result}
        \sum_{j=0}^{2\a-2\g}\d_x\left(\dd{Q_{\a,2}^{[[\g]]}}{v_\g^{(j)}}\right)v_{\g+2}^{(j+1)} = -\d_xP_\a^{[\g+2]} + \overline{\Pi}_1 - \overline{\Pi}_2.
    \end{equation}
    Since the right-hand side is uniquely determined as a function of $C_{\ge\g+1}$, so is the left-hand side. This proves the claim.
\end{proof}
\begin{remark}\label{rem}
    For each $1\le\g\le\a$, let us write $Q_{\a,2}^{[[\g]]}$ by separating the different powers of $\e$:
    \begin{equation*}
        Q_{\a,2}^{[[\g]]} = \sum_{g=0}^{\a-\g}\e^{2g}\bigl(Q_{\a,2}^{[[\g]]}\bigr)^{(2g)}.
    \end{equation*}
    From equation \hyperref[eq:result]{(\ref*{eq:result})} one sees that $\bigl(Q_{\a,2}^{[[\g]]}\bigr)^{(2g)}$ is determined by $\left(\d_xP_\a^{[\g+2]}+\overline{\Pi}_1-\overline{\Pi}_2\right)^{(2g)}$. From the definitions of $P_\a,\overline{\Pi}_1$ and $\overline{\Pi}_2$ it is clear $C_\g$ appears in $\left(\d_xP_\a^{[\g+2]}+\overline{\Pi}_1-\overline{\Pi}_2\right)^{(2g)}$ only if $\g\ge\a+1-g$. This is essentially because $C_\g$ is the coefficient of $\e^{2(\a-\g+1)}v_\g^{(2\a-2\g+2)}$ in the expression \hyperref[eq:Qnotation]{(\ref*{eq:Qnotation})} for $Q_{\a,2}^{[\g]}$. Therefore one can make a stronger conclusion than $Q_{\a,2}^{[[\g]]}=Q_{\a,2}^{[[\g]]}(C_{\ge\g+1})$, namely
    \begin{equation*}
        \bigl(Q_{\a,2}^{[[\g]]}\bigr)^{(2g)} = \bigl(Q_{\a,2}^{[[\g]]}\bigr)^{(2g)}(C_{\ge\a+1-g}), \quad 0\le g\le\a-\g, \quad 1\le\g\le\a.
    \end{equation*}
    In particular, in the $g=0$ case we see that $\bigl(Q_{\a,2}^{[[\g]]}\bigr)^{(0)}$ is uniquely determined for every $1\le\g\le\a$ without any dependence on the unknown linear coefficients $C_\g$. Thus $(Q_{\a,2})^{(0)}$ is uniquely determined by the proposition.
\end{remark}
Going through the argument of \hyperref[prop:deg2]{Proposition \ref*{prop:deg2}} in the case $\gamma=\a-1$, one finds
\begin{align*}
    \sum_{j=0}^2\d_x\left(\dd{Q_{\a,2}^{[[\a-1]]}}{v_{\a-1}^{(j)}}\right)v_{\a+1}^{(j+1)} = & \; \left(2v_2^{(1)}+2v_1v_1^{(1)}+\e^2C_\a v_1^{(3)}\right)v_{\a+1}^{(1)} + \e^2\Bigl(3C_\a+\frac12\Bigr)v_1^{(2)}v_{\a+1}^{(2)} \\
    & + \e^2\Bigl(2C_\a+\frac16\Bigr)v_1^{(1)}v_{\a+1}^{(3)},
\end{align*}
This implies
\begin{equation*}
    \dd{Q_{\a,2}^{[[\a-1]]}}{v_{\a-1}} = 2v_2+v_1^2+\e^2C_\a v_1^{(2)}, \;\, \dd{Q_{\a,2}^{[[\a-1]]}}{v_{\a-1}^{(1)}} = \e^2\Bigl(3C_\a+\frac12\Bigr)v_1^{(1)}, \;\, \dd{Q_{\a,2}^{[[\a-1]]}}{v_{\a-1}^{(2)}} = \e^2\Bigl(2C_\a+\frac16\Bigr)v_1,
\end{equation*}
or equivalently
\begin{equation}\label{eq:a-1part}
    Q_{\a,2}^{[[\a-1]]} = \Bigl(\frac{2}{1+\delta_{\a,3}}v_2+v_1^2+\e^2C_\a v_1^{(2)}\Bigr)v_{\a-1} + \e^2\Bigl(3C_\a+\frac12\Bigr)v_1^{(1)}v_{\a-1}^{(1)} + \e^2\Bigl(2C_\a+\frac16\Bigr)v_1v_{\a-1}^{(2)}.
\end{equation}
We proceed by evaluating the $i=\a$ and $i=\a-1$ terms of $\Pi_1$ and $\Pi_2$ in \hyperref[eq:masterby]{(\ref*{eq:masterby})} respectively, by using the newly obtained expression for $Q_{\a,2}^{[[\a-1]]}$. Firstly, we can compute the $v_\a^{(*)}$ part in \hyperref[eq:bonusaby]{(\ref*{eq:bonusaby})}:
\begin{align*}
    \d_x\mathcal{Q}_\a^{[\a]} = & \; \e^2\Bigl(C_\a+\frac16\Bigr)\d_x^3Q_{\a,2}^{[\a]} + \left(\sum_{i=1}^{\a-1}\sum_{j\ge0}\dd{Q_{\a,2}}{v_i^{(j)}}\d_x^{j+1}Q_{i,2}\right)^{[\a]} - \left(v_2+v_1^2-\frac{\e^2}{6}v_1^{(2)}\right)v_\a^{(1)} \\
    & + \frac{\e^2}{6}v_1v_\a^{(3)} - \frac{\e^4}{180}v_\a^{(5)} \\
    = & \; \e^2\Bigl(C_\a+\frac16\Bigr)\d_x^3(\e^2C_\a v_\a^{(2)}+v_1v_\a) + \sum_{j\ge0}\dd{Q_{\a,2}^{[\a-1]}}{v_{\a-1}^{(j)}}v_\a^{(j+1)} + \underbrace{\left(\dd{Q_{\a,2}}{v_1}\right)^{[\a]}}_{=v_\a}\underbrace{\d_xQ_{1,2}}_{=v_2^{(1)}} \\
    & - \left(v_2+v_1^2-\frac{\e^2}{6}v_1^{(2)}\right)v_\a^{(1)} + \frac{\e^2}{6}v_1v_\a^{(3)} - \frac{\e^4}{180}v_\a^{(5)}.
\end{align*}
Note that both the $i=\a-1$ and $i=1$ terms in the sum have contributed. One eventually obtains
\begin{align}
    \nonumber \mathcal{Q}_\a^{[\a]} = & \; \left(v_2+\e^2\Bigl(C_\a+\frac16\Bigr)v_1^{(2)}\right)v_\a + \e^2\Bigl(3C_\a+\frac12\Bigr)v_1^{(1)}v_\a^{(1)} + \e^2\Bigl(3C_\a+\frac12\Bigr)v_1v_\a^{(2)} \\
    \label{eq:qaa} & + \e^4\left(C_{\a-1}+C_\a\Bigl(C_\a+\frac16\Bigr)-\frac1{180}\right)v_\a^{(5)}.
\end{align}
Thus the $i=\a$ term in the sum $\Pi_1$ is
\begin{align*}
    \sum_{j\ge0}&\dd{(\d_x\mathcal{Q}_\a)}{v_\a^{(j)}}\d_x^{j+1}Q_{\a,2} = \d_x\Biggl(\sum_{j\ge0}\dd{\mathcal{Q}_\a}{v_\a^{(j)}}\d_x^{j+1}Q_{\a,2}\Biggr) \\
    \overset{\hyperref[eq:qaa]{(\ref*{eq:qaa})}}{=} & \; \d_x\biggl[\Bigl(v_2+\e^2\Bigl(C_\a+\frac16\Bigr)v_1^{(2)}\Bigr)\d_xQ_{\a,2} + \e^2\Bigl(3C_\a+\frac12\Bigr)v_1^{(1)}\d_x^2Q_{\a,2} + \e^2\Bigl(3C_\a+\frac12\Bigr)v_1\d_x^3Q_{\a,2} \\
    & \hspace{2em} + \e^4\Bigl(C_{\a-1}+C_\a\Bigl(C_\a+\frac16\Bigr)-\frac1{180}\Bigr)\d_x^5Q_{\a,2}\biggr] =: \d_xP_\a^\prime.
\end{align*}
For the $i=\a-1$ term of the sum $\Pi_2$, we use \hyperref[eq:bonusaby]{(\ref*{eq:bonusaby})} to write $Q_{\a-1,3}=\mathcal{Q}_{\a-1}+Q_{\a,2}$, so that
\begin{align*}
    \d_x\Biggl(\sum_{j\ge0}&\dd{Q_{\a,2}}{v_{\a-1}^{(j)}}\d_x^{j+1}Q_{\a-1,3}\Biggr) = \d_x\biggl[\Bigl(2v_2+v_1^2+\e^2C_\a v_1^{(2)}\Bigr)\d_x(\mathcal{Q}_{\a-1} + Q_{\a,2}) \\ 
    & + \e^2\Bigl(3C_\a+\frac12\Bigr)v_1^{(1)}\d_x^2(\mathcal{Q}_{\a-1}+Q_{\a,2}) + \e^2\Bigl(2C_\a+\frac16\Bigr)v_1\d_x^3(\mathcal{Q}_{\a-1}+Q_{\a,2}) \\
    & + \e^4C_{\a-1}\d_x^5(\mathcal{Q}_{\a-1}+Q_{\a,2})\biggr] =: \d_xP_\a^{\prime\prime}.
\end{align*}
With these last two equations in hand, the relation $\d_x P_\a = \Pi_1-\Pi_2$ becomes
\begin{equation}\label{eq:masterb}
    \d_xS_\a = \underbrace{\sum_{i=1}^{\a-1}\sum_{j\ge0}\dd{(\d_x\mathcal{Q}_\a)}{v_i^{(j)}}\d_x^{j+1}Q_{i,2}}_{=:\Sigma_1} - \underbrace{\d_x\left(\sum_{i=1}^{\a-2}\sum_{j\ge0}\dd{Q_{\a,2}}{v_i^{(j)}}\d_x^{j+1}Q_{i,3}\right)}_{=:\Sigma_2},
\end{equation}
where $S_\a=P_\a-P_\a^\prime+P_\a^{\prime\prime}$ is given by
\begin{align}
    \nonumber S_\a = & \; \Bigl(v_2+v_1^2-\frac{\e^2}{6}v_1^{(2)}\Bigr)\d_x\widetilde{Q}_{\a,2} - \frac{\e^2}{6}v_1\d_x^3\widetilde{Q}_{\a,2} + \frac{\e^4}{180}\d_x^5\widetilde{Q}_{\a,2} + v_1\d_x\mathcal{Q}_\a^\prime + \e^2C_\a\d_x^3\mathcal{Q}_\a^\prime \\
    \nonumber & + \left(2v_2+v_1^2+\e^2C_\a v_1^{(2)}\right)\d_x\mathcal{Q}_{\a-1} + \e^2\Bigl(3C_\a+\frac12\Bigr)v_1^{(1)}\d_x^2\mathcal{Q}_{\a-1} \\
    \label{eq:Sa} & + \e^2\Bigl(2C_\a+\frac16\Bigr)v_1\d_x^3\mathcal{Q}_{\a-1} + \e^4C_{\a-1}\d_x^5\mathcal{Q}_{\a-1}.
\end{align}
Notice that by the inductive hypothesis on $\a$, the polynomial $\mathcal{Q}_{\a-1}$ is uniquely determined. What remains to be proven now is that the linear coefficients $C_\g$ of $Q_{\a,2}$ are uniquely determined. \newline

\noindent \textit{Proof of \hyperref[thm:bkpr]{Theorem \ref*{thm:bkpr}}}. \newline
Suppose there are two differential polynomials $Q_{\a,2}\neq\widehat{Q}_{\a,2}$ satisfying \hyperref[eq:bonusaby]{(\ref*{eq:bonusaby})} and \hyperref[eq:masterb]{(\ref*{eq:masterb})}. Let $R_\a:=Q_{\a,2}-\widehat{Q}_{\a,2}$, which we write as
\begin{equation*}
    R_\a = \sum_{h=g}^\a\e^{2h}R_\a^{(2h)}, \quad R_\a^{(2g)}\neq0.
\end{equation*}
Here $g$ is the lowest genus in which $Q_{\a,2}$ and $\widehat{Q}_{\a,2}$ differ. Clearly $g\le\a$, and by \hyperref[rem]{Remark \ref*{rem}} we have $g\ge1$. In what follows, let \begin{equation*}
    \beta:=\a+1-g.
\end{equation*} 
Denote the linear coefficients of $\widehat{Q}_{\a,2}$ from the expression \hyperref[eq:Qnotation]{(\ref*{eq:Qnotation})} by $\widehat{C}_\g$, so $\widehat{C}_\g=C_\g$ for $\g\ge\b+1$. We write the $\widetilde{\deg}=1$ and $\widetilde{\deg}=2$ parts of $R_\a^{(2g)}$ explicitly:
\begin{equation}\label{eq:Ra}
    R_\a^{(2g)} = \lambda v_\b^{(2g)} + \frac12\sum_{i=1}^{\b-1}\sum_{j=0}^{2g}\omega_{ij}v_i^{(j)}v_{\b-i}^{(2g-j)} + (R_\a^{(2g)})_{\ge3}, \quad \omega_{ij}=\omega_{\b-i,2g-j},
\end{equation}
where $\lambda=C_{\b}-\widehat{C}_{\b}$. By \hyperref[prop:deg2]{Proposition \ref*{prop:deg2}} and \hyperref[rem]{Remark \ref*{rem}} it must be that $\lambda\neq0$. Denote by $\widehat{\mathcal{Q}}_\a,\,\widehat{S}_\a,\,\widehat{\Sigma}_1$ and $\widehat{\Sigma}_2$ the other previously defined quantities associated to $\widehat{Q}_{\a,2}$. We study the $\widetilde{\deg}=2$ part of the $\e^{2g}$ coefficient of the difference between $\d_xS_\a=\Sigma_1-\Sigma_2$ and $\d_x\widehat{S}_\a=\widehat{\Sigma}_1-\widehat{\Sigma}_2$:
\begin{equation}\label{eq:masterfinal}
    \left(\d_xS_\a-\d_x\widehat{S}_\a\right)_2^{(2g)} = (\Sigma_1-\widehat{\Sigma}_1)_2^{(2g)} - (\Sigma_2-\widehat{\Sigma}_2)_2^{(2g)}.
\end{equation}
We adopt the convention $\omega_{ij}:=0$ unless $1\le i\le\b-1$ and $0\le j\le 2g$. Firstly, from the expression \hyperref[eq:bonusaby]{(\ref*{eq:bonusaby})} for $\d_x\mathcal{Q}_\a$ we have
\begin{align*}
    (\d_x\mathcal{Q}_\a-\d_x&\widehat{\mathcal{Q}}_\a)_{\le2}^{(2g)} = \Biggl(\e^2\Bigl(C_\a+\frac16\Bigr)\d_x^3Q_{\a,2} - \e^2\Bigl(\widehat{C}_\a+\frac16\Bigr)\d_x^3\widehat{Q}_{\a,2} + \d_x\mathcal{Q}_\a^\prime-\d_x\widehat{\mathcal{Q}}_\a^\prime\Biggr)_{\le2}^{(2g)} \\
    & = \Biggl(\e^2\Bigl(C_\a+\frac16\Bigr)\d_x^3Q_{\a,2} - \e^2\Bigl(\widehat{C}_\a+\frac16\Bigr)\d_x^3\widehat{Q}_{\a,2} + \sum_{i=1}^{\a-1}\sum_{j\ge0}\dd{R_\a}{v_i^{(j)}}\d_x^{j+1}Q_{i,2}\Biggr)_{\le2}^{(2g)} \\
    & = \Biggl(\lambda\d_x^{2g+1}Q_{\b,2} + \sum_{i=1}^{\b-1}\sum_{j=0}^{2g}\omega_{ij}v_{\b-i}^{(2g-j)}\d_x^{j+1}Q_{i,2}\Biggr)_{\le2}^{(0)} \\
    & \hspace{-1.5em} \overset{\hyperref[lem:deg2]{\mathrm{Lemma}\,\ref*{lem:deg2}}}{=} \lambda\d_x^{2g+1}\Bigl(v_{\b+1}+\sum_{i=1}^{\b-1}v_iv_{\b+1-i}\Bigr) + \sum_{i=1}^{\b-1}\sum_{j=0}^{2g}\omega_{ij}v_{\b-i}^{(2g-j)}v_{i+1}^{(j+1)}.
\end{align*}
From this, one finds
\begin{align*}
    (\Sigma_1&-\widehat{\Sigma}_1)_2^{(2g)} = \Biggl(\sum_{i=1}^{\a-1}\sum_{j\ge0}\dd{(\d_x\mathcal{Q}_\a-\d_x\widehat{\mathcal{Q}}_\a)}{v_i^{(j)}}\d_x^{j+1}Q_{i,2}\Biggr)_2^{(2g)} \\
    = & \; \lambda\d_x^{2g+1}\Biggl[\sum_{i=1}^{\a-1}\frac{\d}{\d v_i}\Bigl(v_{\b+1}+\sum_{k=1}^{\b-1}v_kv_{\b+1-k}\Bigr)\d_xQ_{i,2}\Biggr]_2^{(0)} \\
    & + \sum_{i=1}^{\b-1}\sum_{j=0}^{2g}\omega_{ij}v_{\b+1-i}^{(2g+1-j)}v_{i+1}^{(j+1)} + \sum_{i=2}^{\b-\delta_{g,1}}\sum_{j=1}^{2g+1}\omega_{i-1,j-1}v_{\b+1-i}^{(2g+1-j)}v_{i+1}^{(j+1)} \\
    = & \; (1-\delta_{g,1}-\delta_{g,2})\lambda(\d_x^{2g+2}Q_{\b+1,2})_2^{(0)} + \lambda\d_x^{2g+1}\Biggl(\sum_{i=1}^{\b-1}v_{\b+1-i}v_{i+1}^{(1)} + \sum_{i=2}^{\b-\delta_{g,1}}v_{\b+1-i}v_{i+1}^{(1)}\Biggr) \\
    & + \sum_{i=1}^{\b-2}\sum_{j=0}^{2g}\omega_{ij}v_{\b+1-i}^{(2g+1-j)}v_{i+1}^{(j+1)} + \sum_{j=0}^{2g}\omega_{\b-1,j}v_2^{(2g+1-j)}v_\b^{(j+1)} \\
    & + \sum_{i=2}^{\b-1}\sum_{j=1}^{2g+1}\omega_{i-1,j-1}v_{\b+1-i}^{(2g+1-j)}v_{i+1}^{(j+1)} + (1-\delta_{g,1})\sum_{j=1}^{2g+1}\omega_{\b-1,j-1}v_1^{(2g+1-j)}v_{\b+1}^{(j+1)} \\
    = & \; \textcolor{blue}{(1-\delta_{g,1}-\delta_{g,2})\lambda(\d_x^{2g+2}Q_{\b+1,2})_2^{(0)}} + \textcolor{blue}{\lambda\d_x^{2g+2}\left(\sum_{i=1}^{\b-2}v_{\b+1-i}v_{i+1}\right)} \\
    & + (1-\delta_{g,1})\lambda\d_x^{2g+1}(v_1v_{\b+1}^{(1)}) + \lambda\d_x^{2g+1}(v_2v_\b^{(1)}) + \sum_{i=1}^{\b-2}\sum_{j=0}^{2g}\omega_{ij}v_{\b+1-i}^{(2g+1-j)}v_{i+1}^{(j+1)} \\
    & + \sum_{j=0}^{2g}\omega_{1j}v_2^{(j+1)}v_\b^{(2g+1-j)} + \textcolor{blue}{\sum_{i=1}^{\b-2}\sum_{j=0}^{2g}\omega_{i+1,j}v_{i+1}^{(j)}v_{\b+1-i}^{(2g+2-j)}} + \textcolor{blue}{(1-\delta_{g,1})\sum_{j=0}^{2g}\omega_{1j}v_1^{(j)}v_{\b+1}^{(2g+2-j)}}.
\end{align*}
Secondly, from the expression \hyperref[eq:Sa]{(\ref*{eq:Sa})} for $S_\a$ we have
\begin{align*}
    & (S_\a - \widehat{S}_\a)_2^{(2g)} = \Biggl(\Bigl(v_2+v_1^2-\frac{\e^2}{6}v_1^{(2)}\Bigr)\d_xR_\a - \frac{\e^2}{6}v_1\d_x^3R_\a + \frac{\e^4}{180}\d_x^5R_\a + v_1(\d_x\mathcal{Q}_\a^\prime-\d_x\widehat{\mathcal{Q}}_\a^\prime) \\
    & +\e^2C_\a\d_x^3\mathcal{Q}_\a^\prime-\e^2\widehat{C}_\a\d_x^3\widehat{\mathcal{Q}}_\a^\prime + \e^2(C_\a-\widehat{C}_\a)v_1^{(2)}\d_x\mathcal{Q}_{\a-1}+3\e^2(C_\a-\widehat{C}_\a)v_1^{(1)}\d_x^2\mathcal{Q}_{\a-1} \\
    & + 2\e^2(C_\a-\widehat{C}_\a)v_1\d_x^3\mathcal{Q}_{\a-1} + \e^4(C_{\a-1}-\widehat{C}_{\a-1})\d_x^5\mathcal{Q}_{\a-1}\Biggr)_2^{(2g)} \\
    = & \left(v_2\d_xR_\a + v_1(\d_x\mathcal{Q}_\a^\prime-\d_x\widehat{\mathcal{Q}}_\a^\prime) + \e^2C_\a\d_x^3\mathcal{Q}_\a^\prime - \e^2\widehat{C}_\a\d_x^3\widehat{\mathcal{Q}}_\a^\prime + \e^4(C_{\a-1}-\widehat{C}_{\a-1})\d_x^5\mathcal{Q}_{\a-1}\right)_2^{(2g)} \\
    = & \; \lambda v_2v_\b^{(2g+1)} + (1-\delta_{g,1})\lambda v_1v_{\b+1}^{(2g+1)} + \delta_{g,1}\lambda(\d_x^3\mathcal{Q}_\a^\prime)_2^{(0)} + \delta_{g,2}\lambda(\d_x^5\mathcal{Q}_{\a-1})_2^{(0)} \\
    = & \; \lambda v_2v_\b^{(2g+1)} + (1-\delta_{g,1})\lambda v_1v_{\b+1}^{(2g+1)} + \textcolor{blue}{(\delta_{g,1}+\delta_{g,2})\lambda\d_x^{2g+1}\left(\sum_{i=1}^{\b-2}v_{\b+1-i}v_{i+1}\right)}.
\end{align*}
Thirdly, using $Q_{i,3}=Q_{i+1,2}+\mathcal{Q}_i$,
\begin{align*}
    (\Sigma_2&-\widehat{\Sigma}_2)_2^{(2g)} = \d_x\left(\sum_{i=1}^{\a-2}\sum_{j\ge0}\dd{R_\a}{v_i^{(j)}}\d_x^{j+1}Q_{i,3}\right)_2^{(2g)} = (1-\delta_{g,1}-\delta_{g,2})\lambda\d_x^{2g+2}(Q_{\b+1,2}+\mathcal{Q}_\b)_2^{(0)} \\
    & + \d_x\left[\sum_{i=1}^{\a-2}\sum_{j\ge0}\frac{\d}{\d v_i^{(j)}}\left(\frac12\sum_{k=1}^{\b-1}\sum_{l=0}^{2g}\omega_{kl}v_k^{(l)}v_{\b-k}^{(2g-l)}\right)v_{i+2}^{(j+1)}\right] \\
    = & \; (1-\delta_{g,1}-\delta_{g,2})\lambda(\d_x^{2g+2}Q_{\b+1,2})_2^{(0)} + (1-\delta_{g,1}-\delta_{g,2})\lambda\d_x^{2g+2}\left(\sum_{i=1}^{\b-2}v_{\b+1-i}v_{i+1}\right) \\
    & + \d_x\left(\sum_{i=1}^{\b-1-\delta_{g,1}}\sum_{j=0}^{2g}\omega_{ij}v_{\b-i}^{(2g-j)}v_{i+2}^{(j+1)}\right) \\
    = & \; \textcolor{blue}{(1-\delta_{g,1}-\delta_{g,2})\lambda(\d_x^{2g+2}Q_{\b+1,2})_2^{(0)}} + \textcolor{blue}{(1-\delta_{g,1}-\delta_{g,2})\lambda\d_x^{2g+2}\left(\sum_{i=1}^{\b-2}v_{\b+1-i}v_{i+1}\right)} \\
    & + \textcolor{blue}{\sum_{i=1}^{\b-2}\sum_{j=0}^{2g}\omega_{i+1,j}v_{i+1}^{(j)}v_{\b+1-i}^{(2g+2-j)}} + \sum_{i=1}^{\b-2}\sum_{j=0}^{2g}\omega_{i+1,j}v_{i+1}^{(j+1)}v_{\b+1-i}^{(2g+1-j)} \\
    & + (1-\delta_{g,1})\sum_{j=0}^{2g}\omega_{1j}v_1^{(j+1)}v_{\b+1}^{(2g+1-j)} + \textcolor{blue}{(1-\delta_{g,1})\sum_{j=0}^{2g}\omega_{1j}v_1^{(j)}v_{\b+1}^{(2g+2-j)}}.
\end{align*}
When one considers equation \hyperref[eq:masterfinal]{(\ref*{eq:masterfinal})} the terms highlighted in blue cancel, leaving one with
\begin{align*}
    & \lambda\d_x\left(v_2v_\b^{(2g+1)} + (1-\delta_{g,1})v_1v_{\b+1}^{(2g+1)}\right) = (1-\delta_{g,1})\lambda\d_x^{2g+1}(v_1v_{\b+1}^{(1)}) + \lambda\d_x^{2g+1}(v_2v_\b^{(1)}) \\
    & + \sum_{i=1}^{\b-2}\sum_{j=0}^{2g}(\omega_{ij}-\omega_{i+1,j})v_{\b+1-i}^{(2g+1-j)}v_{i+1}^{(j+1)} + \sum_{j=0}^{2g}\omega_{1j}v_2^{(j+1)}v_\b^{(2g+1-j)} \\
    & - (1-\delta_{g,1})\sum_{j=0}^{2g}\omega_{1j}v_1^{(j+1)}v_{\b+1}^{(2g+1-j)}.
\end{align*}
For $g=\a-1$ (so $\b=2$) we rewrite this as
\begin{align*}
    0 = & \; \lambda\d_x^{2\a-1}(v_1v_3^{(1)}) - \lambda\d_x(v_1v_3^{(2\a-1)}) - \sum_{j=0}^{2\a-2}\omega_{1j}v_1^{(j+1)}v_3^{(2\a-1-j)} \\
    & + \lambda\d_x^{2\a-1}(v_2v_2^{(1)}) - \lambda\d_x(v_2v_2^{(2\a-1)}) + \sum_{j=0}^{2\a-2}\omega_{1j}v_2^{(j+1)}v_2^{(2\a-1-j)} \\
    = & \; ((2\a-2)\lambda-\omega_{10})v_1^{(1)}v_3^{(2\a-1)} + \sum_{j=1}^{2\a-2}\left(\lambda\binom{2\a-1}{j+1}-\omega_{1j}\right)v_1^{(j+1)}v_3^{(2\a-1-j)} \\
    & + ((2\a-1)\lambda+2\omega_{10})v_2^{(1)}v_2^{(2\a-1)} + \sum_{j=1}^{2\a-3}\left(\lambda\binom{2\a-1}{j+1}+\omega_{1j}\right)v_2^{(j+1)}v_2^{(2\a-1-j)}
\end{align*}
The first and third term in the last equation vanish separately,
\begin{equation*}
    (2\a-2)\lambda-\omega_{10} = 0 = (2\a-1)\lambda+2\omega_{10},
\end{equation*}
which contradicts $\lambda\neq0$. For all other values of $g$, we have
\begin{align}
    \nonumber 0 = & \; (1-\delta_{g,1})\biggl(\lambda\d_x^{2g+1}(v_1v_{\b+1}^{(1)}) - \lambda\d_x(v_1v_{\b+1}^{(2g+1)}) - \sum_{j=0}^{2g}\omega_{1j}v_1^{(j+1)}v_{\b+1}^{(2g+1-j)}\biggr) \\
    \nonumber & + \lambda\d_x^{2g+1}(v_2v_\b^{(1)}) - \lambda\d_x(v_2v_\b^{(2g+1)})+ \sum_{j=0}^{2g}\omega_{1j}v_2^{(j+1)}v_\b^{(2g+1-j)} \\
    \nonumber & + \sum_{j=0}^{2g}(\omega_{1j}-\omega_{2j})v_{2}^{(j+1)}v_\b^{(2g+1-j)} + \sum_{i=2}^{\b-2}\sum_{j=0}^{2g}(\omega_{ij}-\omega_{i+1,j})v_{\b+1-i}^{(2g+1-j)}v_{i+1}^{(j+1)} \\
    \label{eq:master1} = & \; (1-\delta_{g,1})\biggl((2g\lambda-\omega_{10})v_1^{(1)}v_{\b+1}^{(2g+1)}+\sum_{j=1}^{2g}\left(\lambda\binom{2g+1}{j+1}-\omega_{1j}\right)v_{1}^{(j+1)}v_{\b+1}^{(2g+1-j)}\biggr) \\
    \label{eq:master2} & + (2g\lambda+2\omega_{10}-\omega_{20})v_2^{(1)}v_\b^{(2g+1)} + \sum_{j=1}^{2g}\left(\lambda\binom{2g+1}{j+1}+2\omega_{1j}-\omega_{2j}\right)v_2^{(j+1)}v_\b^{(2g+1-j)} \\
    \label{eq:master3} & + \sum_{i=2}^{\b-2}\sum_{j=0}^{2g}(\omega_{ij}-\omega_{i+1,j})v_{\b+1-i}^{(2g+1-j)}v_{i+1}^{(j+1)}
\end{align}
Each term in the lines \hyperref[eq:master1]{(\ref*{eq:master1})} and \hyperref[eq:master2]{(\ref*{eq:master2})} vanishes separately. When $g=\a$ (so $\b=1$) we have $\omega_{ij}=0$ for all $i,j$, so the fact that $2g\lambda-\omega_{10}=0$ from line \hyperref[eq:master1]{(\ref*{eq:master1})} contradicts $\lambda\neq0$. For the remaining values of $g$ (i.e. $3\le\beta\le\a$), we proceed as follows. If $\b:=2\g+1$ is odd, the sum in line \hyperref[eq:master3]{(\ref*{eq:master3})} is
\begin{equation*}
    \sum_{i=2}^\g\sum_{j=0}^{2g}(2\omega_{ij}-\omega_{i+1,j}-\omega_{i-1,j})v_{i+1}^{(j+1)}v_{2\g+2-i}^{(2g+1-j)},
\end{equation*}
whereas if $\b:=2\g$ is even it is
\begin{align*}
    & \sum_{i=2}^{\g-1}\sum_{j=0}^{2g}(2\omega_{ij}-\omega_{i+1,j}-\omega_{i-1,j})v_{i+1}^{(j+1)}v_{2\g+1-i}^{(2g+1-j)} \\
    + & \sum_{j=0}^{g-1}(2\omega_{\g j}-\omega_{\g+1,j}-\omega_{\g-1,j})v_{\g+1}^{(j+1)}v_{\g+1}^{(2g+1-j)} + (\omega_{\g g}-\omega_{\g+1,g})v_{\g+1}^{(g+1)}v_{\g+1}^{(g+1)}.
\end{align*}
In either case, each term in the sums vanishes separately, so we obtain the relations
\begin{equation*}
    2\omega_{ij}-\omega_{i+1,j}-\omega_{i-1,j} = 0, \quad 2\le i\le \b-2,\quad 0\le j\le 2g.
\end{equation*}
Solving these inductively gives
\begin{equation}\label{eq:relation}
    \omega_{ij}=-(i-2)\omega_{1j}+(i-1)\omega_{2j}, \quad 1\le i\le\b-1,\quad 0\le j\le 2g.
\end{equation}
In both cases it is also true that $\omega_{\g+1,g}=\omega_{\g,g}$, which in light of the relation \hyperref[eq:relation]{(\ref*{eq:relation})} yields $\omega_{2g}=\omega_{1g}$. In the case $g\neq1$, consider the $j=g$ terms in the sums in lines \hyperref[eq:master1]{(\ref*{eq:master1})} and \hyperref[eq:master2]{(\ref*{eq:master2})}, which both vanish separately:
\begin{align*}
    0 & = \lambda\binom{2g+1}{g+1}-\omega_{1g}, \\
    0 & = \lambda\binom{2g+1}{g+1}+2\omega_{1g}-\omega_{2g}=\lambda\binom{2g+1}{g+1}+\omega_{1g}.
\end{align*}
This contradicts $\lambda\neq0$. In the case $g=1$, the $j=1$ term in line \hyperref[eq:master2]{(\ref*{eq:master2})} vanishes:
\begin{equation*}
    0 = 3\lambda + 2\omega_{11} - \omega_{21} = 3\lambda+\omega_{11},
\end{equation*}
so $\omega_{11}=-3\lambda$. However, from the previously obtained expression \hyperref[eq:a-1part]{(\ref*{eq:a-1part})} we have
\begin{equation*}
    (R_\a^{(2)})_2^{[[\a-1]]} = (Q_{\a,2}^{[[\a-1]]} - \widehat{Q}_{\a,2}^{[[\a-1]]})_2^{(2)} = \lambda v_1^{(2)}v_{\a-1} + 3\lambda v_1^{(1)}v_{\a-1}^{(1)} + 2\lambda v_1v_{\a-1}^{(2)},
\end{equation*}
which by the definition \hyperref[eq:Ra]{(\ref*{eq:Ra})} of the $\omega_{ij}$ implies $\omega_{11}=3\lambda$, contradicting $\lambda\neq0$. \hfill $\qed$

\printbibliography[heading=bibintoc, title={References}]

@article{as23,
    AUTHOR = {Alexandrov, A. and Shadrin, S.},
     TITLE = {Elements of spin Hurwitz theory: closed algebraic formulas, blobbed topological recursion, and a proof of the Giacchetto-Kramer-Lewa\'nski conjecture},
   JOURNAL = {Selecta Math. (N.S.)},
  FJOURNAL = {Selecta Mathematica. New Series},
    VOLUME = {29},
      YEAR = {2023},
    NUMBER = {2},
     PAGES = {Paper No. 26, 44}}

@article{ati71,
    AUTHOR = {Atiyah, M. F.},
     TITLE = {Riemann surfaces and spin structures},
   JOURNAL = {Ann. Sci. \'Ecole Norm. Sup. (4)},
  FJOURNAL = {Annales Scientifiques de l'\'Ecole Normale Sup\'erieure. Quatri\`eme S\'erie},
    VOLUME = {4},
      YEAR = {1971},
     PAGES = {47--62}}

@article{bcggm18,
    AUTHOR = {Bainbridge, M. and Chen, D. and Gendron, Q. and Grushevsky, S. and M\"oller, M.},
     TITLE = {Compactification of strata of Abelian differentials},
   JOURNAL = {Duke Math. J.},
  FJOURNAL = {Duke Mathematical Journal},
    VOLUME = {167},
      YEAR = {2018},
    NUMBER = {12},
     PAGES = {2347--2416}}

@article{bcggm19a,
    AUTHOR = {Bainbridge, M. and Chen, D. and Gendron, Q. and Grushevsky, S. and M\"oller, M.},
     TITLE = {Strata of {$k$}-differentials},
   JOURNAL = {Algebr. Geom.},
  FJOURNAL = {Algebraic Geometry},
    VOLUME = {6},
      YEAR = {2019},
    NUMBER = {2},
     PAGES = {196--233}}

@misc{bcggm19b,
    author = {Bainbridge, M. and Chen, D. and Gendron, Q. and Grushevsky, S. and M\"oller, M.},
    title = {The moduli space of multi-scale differentials}, 
    eprint={1910.13492},
    archivePrefix={arXiv},
    primaryClass={math.AG},
    year = {2019}}

@article{bdgr18,
    AUTHOR = {Buryak, A. and Dubrovin, B. and Gu\'er\'e, J. and Rossi, P.},
     TITLE = {Tau-structure for the double ramification hierarchies},
   JOURNAL = {Comm. Math. Phys.},
  FJOURNAL = {Communications in Mathematical Physics},
    VOLUME = {363},
      YEAR = {2018},
    NUMBER = {1},
     PAGES = {191--260}}

@article{bgr19,
    AUTHOR = {Buryak, A. and Gu\'er\'e, J. and Rossi, P.},
     TITLE = {DR/DZ equivalence conjecture and tautological relations},
   JOURNAL = {Geom. Topol.},
  FJOURNAL = {Geometry \& Topology},
    VOLUME = {23},
      YEAR = {2019},
    NUMBER = {7},
     PAGES = {3537--3600}}

@article{bhpss23,
    AUTHOR = {Bae, Y. and Holmes, D. and Pandharipande, R. and Schmitt, J. and Schwarz, R.},
     TITLE = {Pixton's formula and Abel-Jacobi theory on the Picard stack},
   JOURNAL = {Acta Math.},
  FJOURNAL = {Acta Mathematica},
    VOLUME = {230},
      YEAR = {2023},
    NUMBER = {2},
     PAGES = {205--319}}

@article{bhs22,
    AUTHOR = {Buryak, A. and Hern\'andez Iglesias, F. and Shadrin, S.},
     TITLE = {A conjectural formula for {${\rm DR}_g(a,-a)\lambda_g$}},
   JOURNAL = {\'Epijournal G\'eom. Alg\'ebrique},
  FJOURNAL = {\'Epijournal de G\'eom\'etrie Alg\'ebrique. EPIGA},
    VOLUME = {6},
      YEAR = {2022},
     PAGES = {Art. 8, 17}}

@misc{bls24,
    author = {Blot, X. and Lewański, D. and Shadrin, S.},
    title = {On the strong DR/DZ conjecture}, 
    eprint={2405.12334},
    archivePrefix={arXiv},
    primaryClass={math.AG},
    year = {2024}}

@article{boi15,
    AUTHOR = {Boissy, C.},
     TITLE = {Connected components of the strata of the moduli space of meromorphic differentials},
   JOURNAL = {Comment. Math. Helv.},
  FJOURNAL = {Commentarii Mathematici Helvetici. A Journal of the Swiss Mathematical Society},
    VOLUME = {90},
      YEAR = {2015},
    NUMBER = {2},
     PAGES = {255--286}}

@article{br23,
 author = {Buryak, A. and Rossi, P.},
 title = {Counting meromorphic differentials on {{\(\mathbb{CP}^1\)}}},
 fjournal = {Letters in Mathematical Physics},
 journal = {Lett. Math. Phys.},
 issn = {0377-9017},
 volume = {114},
 number = {4},
 pages = {27},
 note = {Id/No 97},
 year = {2024},
}

@article{br16,
    AUTHOR = {Buryak, A. and Rossi, P.},
     TITLE = {Recursion relations for double ramification hierarchies},
   JOURNAL = {Comm. Math. Phys.},
  FJOURNAL = {Communications in Mathematical Physics},
    VOLUME = {342},
      YEAR = {2016},
    NUMBER = {2},
     PAGES = {533--568}}

@article{br21,
    AUTHOR = {Buryak, A. and Rossi, P.},
     TITLE = {Quadratic double ramification integrals and the noncommutative KdV hierarchy},
   JOURNAL = {Bull. Lond. Math. Soc.},
  FJOURNAL = {Bulletin of the London Mathematical Society},
    VOLUME = {53},
      YEAR = {2021},
    NUMBER = {3},
     PAGES = {843--854}}

@article{brz24,
    AUTHOR = {Buryak, A. and Rossi, P. and Zvonkine, D.},
     TITLE = {Moduli spaces of residueless meromorphic differentials and the KP hierarchy},
   JOURNAL = {Geom. Topol.},
  FJOURNAL = {Geometry \& Topology},
    VOLUME = {28},
      YEAR = {2024},
    NUMBER = {6},
     PAGES = {2793--2824}}

@article{bssz15,
    AUTHOR = {Buryak, A. and Shadrin, S. and Spitz, L. and Zvonkine, D.},
     TITLE = {Integrals of {$\psi$}-classes over double ramification cycles},
   JOURNAL = {Amer. J. Math.},
  FJOURNAL = {American Journal of Mathematics},
    VOLUME = {137},
      YEAR = {2015},
    NUMBER = {3},
     PAGES = {699--737}}

@article{bur15,
    AUTHOR = {Buryak, A.},
     TITLE = {Double ramification cycles and integrable hierarchies},
   JOURNAL = {Comm. Math. Phys.},
  FJOURNAL = {Communications in Mathematical Physics},
    VOLUME = {336},
      YEAR = {2015},
    NUMBER = {3},
     PAGES = {1085--1107}}

@misc{bur22,
    author = {Buryak, A.},
    title = {Partial differential equations with an infinite dimensional algebra of symmetries},
    note = {Extended lecture notes for a course taught by the author at the Faculty of Mathematics of the HSE University in Spring 2022.},
    year = {2024}}

@article{cmsz20,
    AUTHOR = {Chen, D. and M\"oller, M. and Sauvaget, A. and Zagier, D.},
     TITLE = {Masur-Veech volumes and intersection theory on moduli spaces of Abelian differentials},
   JOURNAL = {Invent. Math.},
  FJOURNAL = {Inventiones Mathematicae},
    VOLUME = {222},
      YEAR = {2020},
    NUMBER = {1},
     PAGES = {283--373}}

@article{cmz22,
    AUTHOR = {Costantini, M. and M\"oller, M. and Zachhuber, J.},
     TITLE = {The Chern classes and Euler characteristic of the moduli spaces of Abelian differentials},
   JOURNAL = {Forum Math. Pi},
  FJOURNAL = {Forum of Mathematics. Pi},
    VOLUME = {10},
      YEAR = {2022},
     PAGES = {Paper No. e16, 55},
   MRCLASS = {14D23 (30F60 32G15 57R20)},
  MRNUMBER = {4448178}}

@article{cmz24,
    AUTHOR = {Costantini, M. and M\"oller, M. and Zachhuber, J.},
     TITLE = {The area is a good enough metric},
   JOURNAL = {Ann. Inst. Fourier (Grenoble)},
    VOLUME = {74},
      YEAR = {2024},
    NUMBER = {3},
     PAGES = {1017--1059}}

@misc{css21,
    author = {Costantini, M. and Sauvaget, A. and Schmitt, J.},
    title = {Integrals of $\psi$-classes on twisted double ramification cycles and spaces of differentials}, 
    eprint={2112.04238},
    archivePrefix={arXiv},
    primaryClass={math.AG},
    note = {To appear in Transactions of the American Mathematical Society},
    year = {2021}}

@book{dic03,
    AUTHOR = {Dickey, L. A.},
     TITLE = {Soliton equations and Hamiltonian systems},
    SERIES = {Advanced Series in Mathematical Physics},
    VOLUME = {26},
   EDITION = {Second},
 PUBLISHER = {World Scientific Publishing Co., Inc., River Edge, NJ},
      YEAR = {2003},
     PAGES = {xii+408}}

@article{djkm82,
    AUTHOR = {Date, E. and Jimbo, M. and Kashiwara, M. and Miwa, T.},
     TITLE = {Transformation groups for soliton equations IV. A new hierarchy of soliton equations of KP-type},
   JOURNAL = {Phys. D},
  FJOURNAL = {Physica D. Nonlinear Phenomena},
    VOLUME = {4},
      YEAR = {1982},
    NUMBER = {3},
     PAGES = {343--365}}

@article{dkm81,
    AUTHOR = {Date, E. and Kashiwara, M. and Miwa, T},
     TITLE = {Transformation groups for soliton equations II. Vertex
              operators and {$\tau$}\ functions},
   JOURNAL = {Proc. Japan Acad. Ser. A Math. Sci.},
  FJOURNAL = {Japan Academy. Proceedings. Series A. Mathematical Sciences},
    VOLUME = {57},
      YEAR = {1981},
    NUMBER = {8},
     PAGES = {387--392}}

@misc{dz01,
    author = {Dubrovin, B. and Zhang, Y.},
    title = {Normal forms of hierarchies of integrable PDEs, Frobenius manifolds and Gromov - Witten invariants},
    year={2001},
    eprint={math/0108160},
    archivePrefix={arXiv},
    primaryClass={math.DG}}

@article{eop08,
    AUTHOR = {Eskin, A. and Okounkov, A. and Pandharipande, R.},
     TITLE = {The theta characteristic of a branched covering},
   JOURNAL = {Adv. Math.},
  FJOURNAL = {Advances in Mathematics},
    VOLUME = {217},
      YEAR = {2008},
    NUMBER = {3},
     PAGES = {873--888}}

@article{fp00,
    AUTHOR = {Faber, C. and Pandharipande, R.},
     TITLE = {Logarithmic series and Hodge integrals in the tautological ring},
      NOTE = {Appendix by Don Zagier},
   JOURNAL = {Michigan Math. J.},
  FJOURNAL = {Michigan Mathematical Journal},
    VOLUME = {48},
      YEAR = {2000},
     PAGES = {215--252}}

@article{fp18,
    AUTHOR = {Farkas, G. and Pandharipande, R.},
     TITLE = {The moduli space of twisted canonical divisors},
   JOURNAL = {J. Inst. Math. Jussieu},
  FJOURNAL = {Journal of the Institute of Mathematics of Jussieu. JIMJ. Journal de l'Institut de Math\'ematiques de Jussieu},
    VOLUME = {17},
      YEAR = {2018},
    NUMBER = {3},
     PAGES = {615--672}}

@article{get02,
    AUTHOR = {Getzler, E.},
     TITLE = {A Darboux theorem for Hamiltonian operators in the formal calculus of variations},
   JOURNAL = {Duke Math. J.},
  FJOURNAL = {Duke Mathematical Journal},
    VOLUME = {111},
      YEAR = {2002},
    NUMBER = {3},
     PAGES = {535--560}}

@article{gkls25,
    author = {Giacchetto, A. and Kramer, R. and Lewa\'nski, D. and Sauvaget, A.},
    title = {The spin Gromov-Witten/Hurwitz correspondence for $\mathbb{P}^1$},
    journal = {J. Eur. Math. Soc.},
    year = {2025}}

@misc{glk21,
    author = {Giacchetto, A. and Lewa\'nski, D. and Kramer, R.},
    title = {A new spin on Hurwitz theory and ELSV via theta characteristics}, 
    eprint={2104.05697},
    archivePrefix={arXiv},
    primaryClass={math-ph},
    note = {To appear in Selecta Mathematica},
    year = {2021}}

@article{gv05,
    AUTHOR = {Graber, T. and Vakil, R.},
     TITLE = {Relative virtual localization and vanishing of tautological classes on moduli spaces of curves},
   JOURNAL = {Duke Math. J.},
  FJOURNAL = {Duke Mathematical Journal},
    VOLUME = {130},
      YEAR = {2005},
    NUMBER = {1},
     PAGES = {1--37}}

@incollection{hai13,
    AUTHOR = {Hain, R.},
     TITLE = {Normal functions and the geometry of moduli spaces of curves},
 BOOKTITLE = {Handbook of moduli. Vol. I},
    SERIES = {Adv. Lect. Math.},
    VOLUME = {24},
     PAGES = {527--578},
 PUBLISHER = {Int. Press, Somerville, MA},
      YEAR = {2013}}

@book{hir04,
    AUTHOR = {Hirota, R.},
     TITLE = {The direct method in soliton theory},
    SERIES = {Cambridge Tracts in Mathematics},
    VOLUME = {155},
TRANSLATOR = {Nagai, A. and Nimmo, J. and Gilson, C.},
 PUBLISHER = {Cambridge University Press, Cambridge},
      YEAR = {2004},
     PAGES = {xii+200}}

@article{hol21,
    AUTHOR = {Holmes, D.},
     TITLE = {Extending the double ramification cycle by resolving the Abel-Jacobi map},
   JOURNAL = {J. Inst. Math. Jussieu},
  FJOURNAL = {Journal de l'Institut de Math\'ematiques de Jussieu},
    VOLUME = {20},
      YEAR = {2021},
    NUMBER = {1},
     PAGES = {331--359}}

@article{hs21,
    AUTHOR = {Holmes, D. and Schmitt, J.},
     TITLE = {Infinitesimal structure of the pluricanonical double ramification locus},
   JOURNAL = {Compos. Math.},
  FJOURNAL = {Compositio Mathematica},
    VOLUME = {157},
      YEAR = {2021},
    NUMBER = {10},
     PAGES = {2280--2337}}

@article{is22,
    author = {van Ittersum, J. and Sauvaget, A.},
    title = {Cylinder counts and spin refinement of area Siegel-Veech constants},
    journal = {Comment. Math. Helv. },
    fjournal = {Commentarii Mathematici Helvetici},
    year = {2022}}

@article{jppz17,
    AUTHOR = {Janda, F. and Pandharipande, R. and Pixton, A. and Zvonkine, D.},
     TITLE = {Double ramification cycles on the moduli spaces of curves},
   JOURNAL = {Publ. Math. Inst. Hautes \'Etudes Sci.},
  FJOURNAL = {Publications Math\'ematiques. Institut de Hautes \'Etudes Scientifiques},
    VOLUME = {125},
      YEAR = {2017},
     PAGES = {221--266}}

@article{kl07,
    AUTHOR = {Kazarian, M. E. and Lando, S. K.},
     TITLE = {An algebro-geometric proof of Witten's conjecture},
   JOURNAL = {J. Amer. Math. Soc.},
  FJOURNAL = {Journal of the American Mathematical Society},
    VOLUME = {20},
      YEAR = {2007},
    NUMBER = {4},
     PAGES = {1079--1089}}

@article{km94,
    AUTHOR = {Kontsevich, M. and Manin, Yu.},
     TITLE = {Gromov-Witten classes, quantum cohomology, and enumerative geometry},
   JOURNAL = {Comm. Math. Phys.},
  FJOURNAL = {Communications in Mathematical Physics},
    VOLUME = {164},
      YEAR = {1994},
    NUMBER = {3},
     PAGES = {525--562}}

@article{kz03,
    AUTHOR = {Kontsevich, M. and Zorich, A.},
     TITLE = {Connected components of the moduli spaces of Abelian differentials with prescribed singularities},
   JOURNAL = {Invent. Math.},
  FJOURNAL = {Inventiones Mathematicae},
    VOLUME = {153},
      YEAR = {2003},
    NUMBER = {3},
     PAGES = {631--678}}

@article{lee20,
    AUTHOR = {Lee, J.},
     TITLE = {A square root of Hurwitz numbers},
   JOURNAL = {Manuscripta Math.},
  FJOURNAL = {Manuscripta Mathematica},
    VOLUME = {162},
      YEAR = {2020},
    NUMBER = {1-2},
     PAGES = {99--113}}

@article{li02,
    AUTHOR = {Li, J.},
     TITLE = {A degeneration formula of {GW}-invariants},
   JOURNAL = {J. Differential Geom.},
  FJOURNAL = {Journal of Differential Geometry},
    VOLUME = {60},
      YEAR = {2002},
    NUMBER = {2},
     PAGES = {199--293}}

@article{ll05,
    AUTHOR = {Liu, G. and Luo, H.},
     TITLE = {Some identities involving Bernoulli numbers},
   JOURNAL = {Fibonacci Quart.},
  FJOURNAL = {The Fibonacci Quarterly},
    VOLUME = {43},
      YEAR = {2005},
    NUMBER = {3},
     PAGES = {208--212}}

@article{lrz15,
    AUTHOR = {Liu, S. and Ruan, Y. and Zhang, Y.},
     TITLE = {BCFG Drinfeld-Sokolov hierarchies and FJRW-theory},
   JOURNAL = {Invent. Math.},
  FJOURNAL = {Inventiones Mathematicae},
    VOLUME = {201},
      YEAR = {2015},
    NUMBER = {2},
     PAGES = {711--772}}

@article{mag08,
    AUTHOR = {Magli, P.},
     TITLE = {Identities involving {B}ernoulli numbers related to sums of powers of integers},
   JOURNAL = {Fibonacci Quart.},
  FJOURNAL = {The Fibonacci Quarterly},
    VOLUME = {46/47},
      YEAR = {2008},
    NUMBER = {2},
     PAGES = {140--145}}

@book{mjd00,
    AUTHOR = {Miwa, T. and Jimbo, M. and Date, E.},
     TITLE = {Solitons: Differential Equations, Symmetries and Infinite Dimensional Algebras},
    SERIES = {Cambridge Tracts in Mathematics},
    VOLUME = {135},
TRANSLATOR = {Reid, M.},
 PUBLISHER = {Cambridge University Press, Cambridge},
      YEAR = {2000},
     PAGES = {x+108}}

@article{mmn20,
    AUTHOR = {Mironov, A. D. and Morozov, A. and Natanzon, S. M.},
     TITLE = {Cut-and-join structure and integrability for spin Hurwitz numbers},
   JOURNAL = {Eur. Phys. J. C},
  FJOURNAL = {The European Physical Journal C},
    VOLUME = {80},
      YEAR = {2020},
    NUMBER = {97}}

@article{mmno21,
    AUTHOR = {Mironov, A. D. and Morozov, A. and Natanzon, S. M. and Orlov, A. Yu.},
     TITLE = {Around spin Hurwitz numbers},
   JOURNAL = {Lett. Math. Phys.},
  FJOURNAL = {Letters in Mathematical Physics},
    VOLUME = {111},
      YEAR = {2021},
    NUMBER = {5},
     PAGES = {Paper No. 124, 39}}

@article{mum71,
    AUTHOR = {Mumford, D.},
     TITLE = {Theta characteristics of an algebraic curve},
   JOURNAL = {Ann. Sci. \'Ecole Norm. Sup.},
    VOLUME = {4},
      YEAR = {1971},
     PAGES = {181--192}}

@incollection{pix23,
    AUTHOR = {Pixton, A.},
     TITLE = {The double ramification cycle formula},
 BOOKTITLE = {ICM---International Congress of Mathematicians. Vol. 3. Sections 1--4},
     PAGES = {2312--2322},
 PUBLISHER = {EMS Press, Berlin},
      YEAR = {2023}}

@misc{pix23poly,
    author = {Pixton, A.},
    title = {DR cycle polynomiality and related results}, 
    note = {Available on the author's webpage},
    year = {2023}}

@article{ros17,
    AUTHOR = {Rossi, P.},
     TITLE = {Integrability, quantization and moduli spaces of curves},
   JOURNAL = {SIGMA Symmetry Integrability Geom. Methods Appl.},
  FJOURNAL = {SIGMA. Symmetry, Integrability and Geometry. Methods and Applications},
    VOLUME = {13},
      YEAR = {2017},
     PAGES = {Paper No. 060, 29}}

@article{sau19,
    AUTHOR = {Sauvaget, A.},
     TITLE = {Cohomology classes of strata of differentials},
   JOURNAL = {Geom. Topol.},
  FJOURNAL = {Geometry \& Topology},
    VOLUME = {23},
      YEAR = {2019},
    NUMBER = {3},
     PAGES = {1085--1171}}

@article{sch18,
    AUTHOR = {Schmitt, Johannes},
     TITLE = {Dimension theory of the moduli space of twisted {$k$}-differentials},
   JOURNAL = {Doc. Math.},
  FJOURNAL = {Documenta Mathematica},
    VOLUME = {23},
      YEAR = {2018},
     PAGES = {871--894}}

@misc{spe24,
    author = {Spelier, P.},
    title = {Polynomiality of the double ramification cycle}, 
    eprint={2401.17421},
    archivePrefix={arXiv},
    primaryClass={math.AG},
    year = {2024},
    note = {To appear in Algebraic Geometry}}

@article{won24,
    AUTHOR = {Wong, Y. M.},
     TITLE = {An algorithm to compute the fundamental classes of spin components of strata of differentials},
   JOURNAL = {Int. Math. Res. Not. IMRN},
  FJOURNAL = {International Mathematics Research Notices. IMRN},
      YEAR = {2024},
    NUMBER = {6},
     PAGES = {4893--4962}}

@article{wy23,
    AUTHOR = {Wang, Z. and Yang, C.},
     TITLE = {Connected ({$n$},{$m$})-point functions of diagonal 2-BKP tau-functions and spin double Hurwitz numbers},
   JOURNAL = {J. Math. Phys.},
  FJOURNAL = {Journal of Mathematical Physics},
    VOLUME = {64},
      YEAR = {2023},
    NUMBER = {4},
     PAGES = {Paper No. 041702, 17}}

@article{zab21,
    AUTHOR = {Zabrodin, A. V.},
     TITLE = {Kadomtsev-Petviashvili hierarchies of types B and C},
   JOURNAL = {Teoret. Mat. Fiz.},
  FJOURNAL = {Teoreticheskaya i Matematicheskaya Fizika},
    VOLUME = {208},
      YEAR = {2021},
    NUMBER = {1},
     PAGES = {15--38}}

@article{GP03,
    Author = {Graber, T. and Pandharipande, R.},
    Title = {Constructions of nontautological classes on moduli spaces of curves.},
    FJournal = {Michigan Mathematical Journal},
    Journal = {Mich. Math. J.},
    Volume = {51},
    Number = {1},
    Pages = {93--109},
    Year = {2003}}

@article{vS,
    Author = {Schmitt, J. and van Zelm, J.},
    Title = {Intersections of loci of admissible covers with tautological classes},
    FJournal = {Selecta Mathematica. New Series},
    Journal = {Sel. Math., New Ser.},
    Volume = {26},
    Number = {5},
    Pages = {69},
    Note = {Id/No 79},
    Year = {2020}}

@inbook{Mum83,
    Author = "Mumford, D.",
    Editor = "Artin, M. and Tate, J.",
    Title = "Towards an Enumerative Geometry of the Moduli Space of Curves",
    bookTitle = "Arithmetic and Geometry: Papers Dedicated to I.R. Shafarevich on the Occasion of His Sixtieth Birthday. Volume II: Geometry",
    year = "1983",
    publisher = "Birkh{\"a}user Boston",
    pages = "271--328"}

@article{AC,
    author = {Arbarello, E. and Cornalba, M.},
    title = {The Picard groups of the moduli spaces of curves},
    journal = {Topology},
    volume = {26},
    pages = {153--171},
    year = {1987}}

@article{admcycles,
 author = {Delecroix, V. and Schmitt, J. and van Zelm, J.},
 title = {admcycles -- a {Sage} package for calculations in the tautological ring of the moduli space of stable curves},
 fjournal = {The Journal of Software for Algebra and Geometry},
 journal = {J. Softw. Algebra Geom.},
 issn = {1948-7916},
 volume = {11},
 pages = {89--112},
 year = {2021}}

@article{op06,
 author = {Okounkov, A. and Pandharipande, R.},
 title = {The equivariant {Gromov}-{Witten} theory of {{\(\mathbb P^1\)}}},
 fjournal = {Annals of Mathematics. Second Series},
 journal = {Ann. Math. (2)},
 issn = {0003-486X},
 volume = {163},
 number = {2},
 pages = {561--605},
 year = {2006}}

\end{document}